\documentclass[a4paper,10pt]{amsart}
\usepackage{a4wide}

\usepackage[latin1]{inputenc}
\usepackage{subfigure}
\usepackage{mathrsfs}		
\usepackage{bm}             
\usepackage{url}            
\usepackage{paralist}       
\usepackage{amsmath}        
\usepackage{amssymb}
\usepackage{amsfonts}
\usepackage{amscd}
\usepackage{tikz}

\usepackage[abs]{overpic}
\usepackage{rotating}

\newtheorem{thm}{Theorem}[section]
\newtheorem{cor}[thm]{Corollary}
\newtheorem{lem}[thm]{Lemma}
\newtheorem{prop}[thm]{Proposition}

\newtheorem{defn}[thm]{Definition}
\newtheorem{expl}[thm]{Example}

\def\R{{\mathbb R}}       

\def\Do{{\mathsf D}}      
\def\Up{{\mathsf U}}      
\def\As{{\mathsf {As}}}     
\def\Cy{{\mathsf {Cy}}}     

\newcommand{\set}[2]{\left\{#1\vphantom{#2}\right.\;\left|\;\vphantom{#1}#2\right\}}

\date{August 26, 2013}
\title[Minkowski decomposition of associahedra \& related combinatorics]{Minkowski decomposition of associahedra\\ and related combinatorics}
\author{
Carsten Lange
}
\thanks{\noindent Partially supported by DFG Forschergruppe 565 \emph{Polyhedral Surfaces}\\
	Freie Universit\"at Berlin,
	FB Mathematik \& Informatik,
	Arnimallee 6,
	14195~Berlin,
	Germany and \\
    Universit\'e Paris VI,
	Institut de Math{\'e}matiques de Jussieu,
	4 place de Jussieu, 
	75005 Paris, 
	France\\  \texttt{clange@math.fu-berlin.de}
	}

\begin{document}
\maketitle
\begin{abstract}
Realisations of associahedra with linear non-isomorphic normal fans can be obtained by alteration of
the right-hand sides of the facet-defining inequalities from a classical permutahedron. These polytopes 
can be expressed as Minkowski sums and differences of dilated faces of a standard simplex as described by Ardila, Benedetti \& Doker (2010).
The coefficients~$y_I$ of such a Minkowski decomposition can be computed by M\"obius inversion if tight 
right-hand sides~$z_I$ are known not just for the facet-defining inequalities of the associahedron but also for all inequalities 
of the permutahedron that are redundant for the associahedron. 

\medskip
\noindent
We show for certain families of these associahedra:
\begin{compactitem}
	\item how to compute the tight value~$z_I$ for any inequality that is redundant for an associahedron but
	  facet-defining for the classical permutahedron. More precisely, each value~$z_I$ is described in terms of tight values~$z_J$ 
	  of facet-defining inequalities of the corresponding associahedron determined by combinatorial properties of~$I$.
	\item the computation of the values~$y_I$ of Ardila, Benedetti \& Doker can be significantly 
	  simplified and depends on at most four values $z_{a(I)}$, $z_{b(I)}$, $z_{c(I)}$ and $z_{d(I)}$.
	\item the four indices $a(I)$, $b(I)$, $c(I)$ and $d(I)$ are determined by the geometry of
	  the normal fan of the associahedron and are described combinatorially.
	\item a combinatorial interpretation of the values~$y_I$ using a labeled~$n$-gon. This result 
	  is inspired from similar interpretations for vertex coordinates originally described 
	  by \mbox{J.-L.}~Loday and well-known interpretations for the $z_I$-values of facet-defining inequalities.
\end{compactitem}
\end{abstract}

\section{Introduction} \label{sec_gen_permu_and_assoc}

A.~Postnikov defined in~\cite{Postnikov} generalised permutahedra as a subfamily of all convex
polytopes that have the following H-description:
\[
	P_n(\{ z_I\}) := \set{\pmb x \in \mathbb R^n}
						 {\begin{matrix}
							\sum_{i\in[n]}x_i = z_{[n]}\text{ and } 
							\sum_{i\in I}x_i \geq z_I \text{ for }\varnothing \subset I \subset [n]\\
						  \end{matrix}
						  }
\]
where~$[n]$ denotes the set $\{ 1,2,\cdots,n\}$. The classical $(n-1)$-dimensional permutahedron, 
as described for example by G.~M.~Ziegler,~\cite{ziegler}, corresponds to $z_I = \tfrac{|I|(|I|+1)}{2}$ 
for $\varnothing \subset I \subseteq [n]$ (we distinguish between $\subset$ and $\subseteq$!). Obviously, 
some of the above inequalities may be redundant for~$P_n(\{z_I\})$ and, unless the value~$z_I$ is tight, sufficiently 
small increases and decreases of~$z_I$ for a redundant inequality do not change the combinatorial type of~$P_n(\{z_I\})$. 
Although the encoding by all values~$z_I$ is not efficient, Proposition~\ref{prop:ardila} below gives a good reason to 
specify tight values~$z_I$ for all $I\subseteq [n]$. The subfamily of generalised permutahedra is now characterised by 
the additional requirement that~$P_n(\{ z_I\})$ is an element of the deformation cone of the classical permutahedron. 
Equivalently, this means that the normal fan of the generalised permutahedron is a coarsening of the normal fan of the
classical permutahedron or that no facet-defining hyperplane of the permutahedron is moved \emph{past any vertices}, 
compare A.~Postnikov, V.~Reiner, and L.~Williams,~\cite{PostnikovReinerWilliams}. This fine 
distinction and additional condition is easily overlooked but essential. For example, Proposition~\ref{prop:ardila} 
does not hold for arbitrary polytopes~$P_n(\{ z_I\})$, we illustrate this by a simple example in Section~\ref{sec_cyclohedron}. 
Fundamental examples are dilations of the standard simplex $\Delta_n=\operatorname{conv}\{ e_1, e_2, \cdots, e_n\}$
where $e_i$ denotes the $i^{th}$ standard basis vector of $\R^n$.

For any two polytopes~$P$ and~$Q$, the Minkowski sum $P+Q$ is defined as~$\{ p+q \ | \ p\in P,\ q\in Q \}$. In contrast, 
we define the Minkowski difference~$P-Q$ of~$P$ and~$Q$ only if there is a polytope~$R$ such that~$P=Q+R$. For more
details on Minkowski differences we refer to~\cite{Schneider}. We are interested in decompositions of generalised 
permutahedra into Minkowski sums and differences of dilated faces of the $(n-1)$-dimensional standard simplex~$\Delta_n$,
where the faces~$\Delta_I$ of~$\Delta_n$ are given by $\operatorname{conv}\{ e_i\}_{i\in I}$ for $I\subseteq [n]$. If a 
polytope~$P$ is the Minkowski sum and difference of dilated faces of~$\Delta_n$, we say that~$P$ has a Minkowski 
decomposition into faces of the standard simplex. The following two results are known key observations.

\begin{lem}[{\cite[Lemma~2.1]{ArdilaBenedettiDoker}}]
	$P_n(\{z_I\})+P_n(\{z^{\prime}_I\}) = P_n(\{z_I + z^{\prime}_I\})$.
\end{lem}
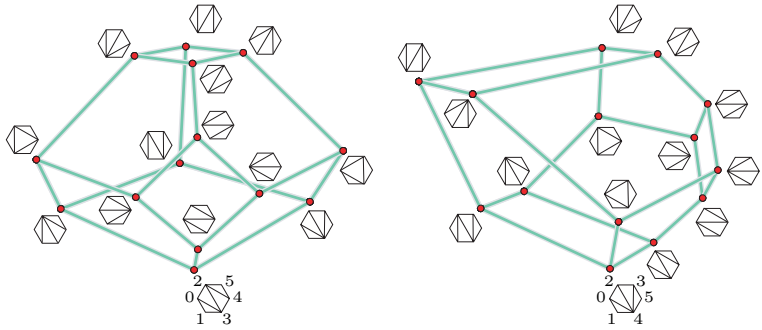
\begin{figure}
      \begin{center}
      \begin{minipage}{0.95\linewidth}
         \begin{center}
         \begin{overpic}
            [width=0.7\linewidth]{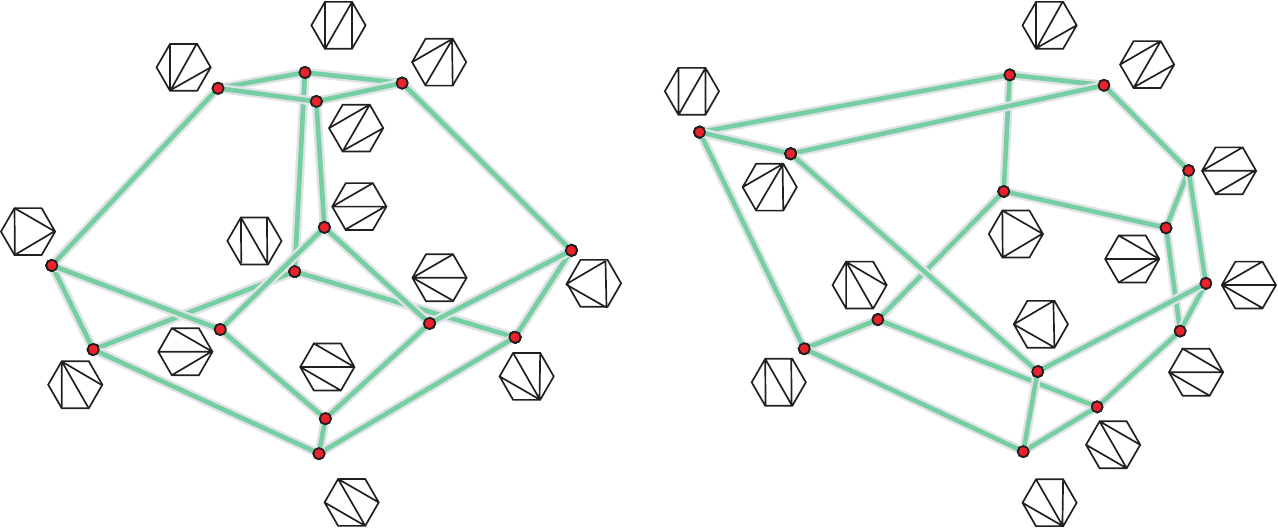}
            \put(67,4){\tiny $0$}
            \put(71,-4){\tiny $1$}
            \put(70,10){\tiny $2$}
            \put(81,-4){\tiny $3$}
            \put(85,4){\tiny $4$}
            \put(82,10){\tiny $5$}
            \put(221,4){\tiny $0$}
            \put(225,-4){\tiny $1$}
            \put(224,10){\tiny $2$}
            \put(236,10){\tiny $3$}
            \put(235,-4){\tiny $4$}
            \put(239,4){\tiny $5$}
         \end{overpic}
         \end{center}
         \caption[]{Two $3$-dimensional associahedra~$\As^c_3=P_4(\{\tilde z_I^c\})$ with vertex coordinates computed for
 					differently choosen Coxeter elements according to~\cite{HohlwegLange}. The different Coxeter elements 
					are encoded by different labelings of hexagons as indicated. The images shown are isometric copies
					of $3$-polytopes contained in the affine hyperplane $x_1+x_2+x_3+x_4=10$ of $\R^4$.}
         \label{fig:a3_associahedra}
      \end{minipage}
      \end{center}
\end{figure}
If we consider the function $I\longmapsto z_I$ that assigns every subset of~$[n]$ the corresponding tight
value~$z_I$ of~$P_n(\{z_I\})$, then the M\"obius inverse of this function assigns to~$I$ the coefficient~$y_I$
of a Minkowski decomposition of~$P_n(\{z_I\})$ into faces of the standard simplex:
\begin{prop}[{\cite[Proposition~2.3]{ArdilaBenedettiDoker}}]	
	\label{prop:ardila}
	$ $\\ \noindent
	Every generalised permutahedron $P_n(\{z_I\})$ can be written uniquely as a Minkowski 
	sum and difference of faces of~$\Delta_n$:
	\[
		P_n(\{z_I\}) = \sum_{I\subseteq [n]}y_I\Delta_I
	\]
	where $y_I = \sum_{J\subseteq I}(-1)^{|I\setminus J|}z_J$ for each $I\subseteq [n]$.	
\end{prop}
\noindent
In particular, we also have $z_I = \sum_{J\subseteq I} y_J$. A basic example is the classical permutahdron: it is
known to be a zonotope and it is the Minkowski sum of the edges and vertices of~$\Delta_n$. The reader is invited to 
check that the corresponding $z_I$-values obtained by this formula yield precisely the right-hand sides mentioned earlier.

\medskip
We will study Minkowski decompositions of generalised permutahedra that have the same normal fan as~$\As^{c}_{n-1}$. Two 
$3$-dimensional examples of~$\As^c_3$ (with distinct normal fans) are shown in Figure~\ref{fig:a3_associahedra}, we describe 
their construction in detail in Section~\ref{sec_assoc_def}. The normal fans of these polytopes are determined by a Coxeter 
element~$c$ of the symmetric group, but we will avoid the explicit use of Coxeter elements and use a partition $\Do_c\sqcup\Up_c$
of~$[n]$ induced by~$c$ instead. The main result is that the relation between $z_I$- and $y_I$-coordinates of 
Proposition~\ref{prop:ardila} simplified significantly: each $y_I$ can be computed from at most four values~$z_J$ which depend 
on~$I$ and the normal fan of the polytope (or, equivalently, the Coxeter element~$c$ or the corresponding partition of~$[n]$). Moreover, we give
an explicit combinatorial description how to determine these terms~$z_J$. If the we further restrict to the realisations~$\As^{c}_{n-1}$
as described by C.~Hohlweg and C.~Lange in~\cite{HohlwegLange}, we show that the coefficients~$y_I$ can be described
as signed product of path-lengths of a labeled polygon.

We now give examples of Minkowski decompositions of realisations of $2$-dimensional associahedra~$\As^{c_1}_2$ 
and $\As^{c_2}_2$ which are contained in the affine hyperplane $x_1+x_2+x_3=6$ of $\R^3$. We immediately see that
the Minkowski decompositions are distinct since the set of coefficients~$y_I$ differ. These associahedra are 
pentagons that are obtained from the classical permutahedron by making the inequality $x_1+x_3\geq 3$ (respectively 
$x_2\geq 1$) redundant. They are described by the following complete set of tight $z_I$-values 
$z^{c_1}_I$ and $z^{c_2}_I$: \\[2mm]
\centerline{
   \begin{tabular}{l|c|c|c|c|c|c|c}
  	  $I$   		& $\{1\}$ & $\{2\}$ & $\{3\}$ & $\{1,2\}$ & $\{1,3\}$ & $\{2,3\}$ & $\{1,2,3\}$ \\ \hline
	  $z^{c_1}_I$	& $1$ 	  & $1$ 	& $1$ 	  & $3$ 	  & $2$   	  & $3$ 	  & $6$ 		\\ \hline
	  $z^{c_2}_I$	& $1$ 	  & $0$ 	& $1$ 	  & $3$ 	  & $3$   	  & $3$ 	  & $6$
   \end{tabular}
}\\[2mm]
Using Proposition~\ref{prop:ardila}, the reader may verify that
\[
	\As_2^{c_1} = 1\cdot \Delta_{\{1\}} + 1\cdot \Delta_{\{2\}} + 1\cdot \Delta_{\{3\}} 
				+ 1\cdot \Delta_{\{1,2\}} + 0\cdot \Delta_{\{1,3\}} + 1\cdot \Delta_{\{2,3\}} 
				+ 1\cdot \Delta_{\{1,2,3\}}
\]
and
\[
	\As_2^{c_2} = 1\cdot \Delta_{\{1\}} + 0\cdot \Delta_{\{2\}} + 1\cdot \Delta_{\{3\}} 
				+ 2\cdot \Delta_{\{1,2\}} + 1\cdot \Delta_{\{1,3\}} + 2\cdot \Delta_{\{2,3\}} 
				+ (-1)\cdot \Delta_{\{1,2,3\}},
\]
illustrations of these decompositions are given in Figures~\ref{fig:minkowski_decomp_As_c_1} and~\ref{fig:minkowski_decomp_As_c_2}. 

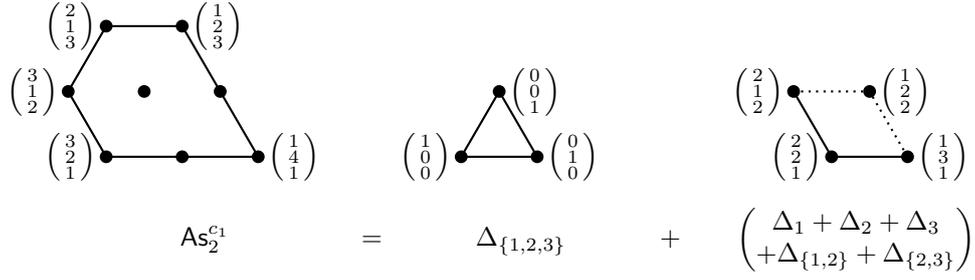
\begin{figure}
      \begin{center}
		 \begin{minipage}{0.95\linewidth}
         \begin{center}
			\begin{tikzpicture}[thick]
				\draw (0:0cm) -- (0:2cm);
				\draw (0:0cm) -- (120:1cm);
				\draw (120:1cm) -- (0cm,1.732cm) -- (1,1.732cm) -- (0:2cm);
				\filldraw [black] (0:0cm) node[anchor=east] {$\left( \begin{smallmatrix} 3\\ 2\\ 1 \end{smallmatrix}\right)$} circle (2pt)
								  (0:1cm) circle (2pt)
								  (120:1cm) node[anchor=east] {$\left( \begin{smallmatrix} 3\\ 1\\ 2 \end{smallmatrix}\right)$} circle (2pt)
								  (2cm,0cm) node[anchor=west] {$\left( \begin{smallmatrix} 1\\ 4\\ 1 \end{smallmatrix}\right)$} circle (2pt)
								  (60:1cm) circle (2pt)
								  (60:2cm) node[anchor=west] {$\left( \begin{smallmatrix} 1\\ 2\\ 3 \end{smallmatrix}\right)$} circle (2pt)
								  (0cm,1.732cm) node[anchor=east] {$\left( \begin{smallmatrix} 2\\ 1\\ 3 \end{smallmatrix}\right)$} circle (2pt)
								  (1.5cm,0.5*1.732cm) circle (2pt);
			\end{tikzpicture}
			\qquad
			\begin{tikzpicture}[thick]
				\draw (0:0cm) -- (60:1cm) -- (0:1cm) -- (0:0cm);
				\filldraw [black] (0:0cm) node[anchor=east] {$\left( \begin{smallmatrix} 1\\ 0\\ 0 \end{smallmatrix}\right)$} circle (2pt)
								  (0:1cm) node[anchor=west] {$\left( \begin{smallmatrix} 0\\ 1\\ 0 \end{smallmatrix}\right)$} circle (2pt)
								  (60:1cm) node[anchor=west] {$\left( \begin{smallmatrix} 0\\ 0\\ 1 \end{smallmatrix}\right)$} circle (2pt);
			\end{tikzpicture}
			\qquad\qquad
			\begin{tikzpicture}[thick]
				\draw (0:0cm) -- (0:1cm);
				\draw (0:0cm) -- (120:1cm);
				\draw [dotted](-0.5cm,0.5*1.732cm) -- (0.5cm,0.5*1.732cm) -- (0:1cm);
				\filldraw [black] (0:0cm) node[anchor=east] {$\left( \begin{smallmatrix} 2\\ 2\\ 1 \end{smallmatrix}\right)$} circle (2pt)
							  	  (0:1cm) node[anchor=west] {$\left( \begin{smallmatrix} 1\\ 3\\ 1 \end{smallmatrix}\right)$} circle (2pt)
								  (120:1cm) node[anchor=east] {$\left( \begin{smallmatrix} 2\\ 1\\ 2 \end{smallmatrix}\right)$} circle (2pt)
								  (60:1cm) node[anchor=west] {$\left( \begin{smallmatrix} 1\\ 2\\ 2 \end{smallmatrix}\right)$} circle (2pt);
		    \end{tikzpicture}
		    $ $\\[2mm]
		    \rule{2.25cm}{0cm}$\As_2^{c_1}$ \rule{1.5cm}{0cm} $=$ \rule{1cm}{0cm} $\Delta_{\{1,2,3\}}$ \rule{1.0cm}{0cm} 
			     $+$ \rule{0.5cm}{0cm} $\left(\begin{matrix} \Delta_1+\Delta_2+\Delta_3\\ +\Delta_{\{1,2\}}+\Delta_{\{2,3\}}\end{matrix}\right)$
	     \end{center}
	     \caption[]{The Minkowski decomposition of the $2$-dimensional assiciahedron~$\As_2^{c_1}$ into faces of the standard simplex is actually a Minkowski sum of some faces of a standard simplex.}
         \label{fig:minkowski_decomp_As_c_1}
         \end{minipage}
      \end{center}
\end{figure}
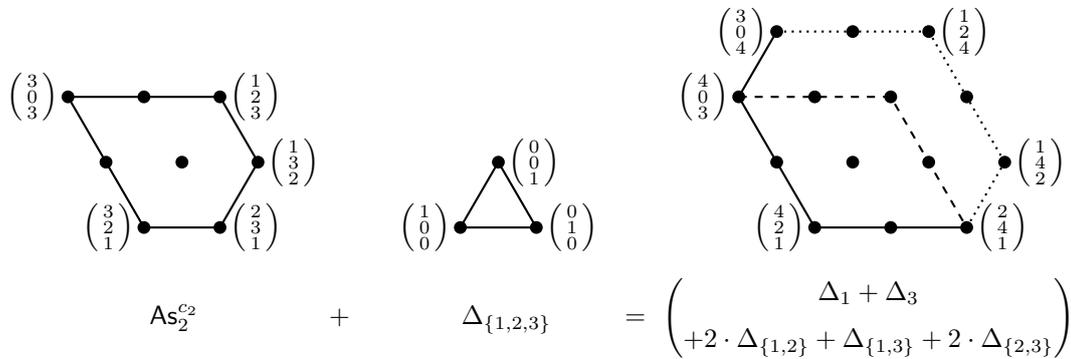
\begin{figure}[b]
      \begin{center}
		 \begin{minipage}{0.95\linewidth}
         \begin{center}
			\begin{tikzpicture}[thick]
				\draw (0:0cm) -- (0:1cm);
				\draw (0:0cm) -- (120:2cm);
				\draw  (120:2cm) -- (60:2cm) -- (1.5cm,0.5*1.732cm) -- (1,0);
				\filldraw [black] (0:0cm) node[anchor=east] {$\left( \begin{smallmatrix} 3\\ 2\\ 1 \end{smallmatrix}\right)$} circle (2pt)
								  (0:1cm) node[anchor=west] {$\left( \begin{smallmatrix} 2\\ 3\\ 1 \end{smallmatrix}\right)$} circle (2pt)
								  (120:1cm) circle (2pt)
								  (120:2cm) node[anchor=east] {$\left( \begin{smallmatrix} 3\\ 0\\ 3 \end{smallmatrix}\right)$} circle (2pt)
								  (60:1cm) circle (2pt)
								  (60:2cm) node[anchor=west] {$\left( \begin{smallmatrix} 1\\ 2\\ 3 \end{smallmatrix}\right)$} circle (2pt)
								  (0cm,1.732cm) circle (2pt)
								  (1.5cm,0.5*1.732cm) node[anchor=west] {$\left( \begin{smallmatrix} 1\\ 3\\ 2 \end{smallmatrix}\right)$} circle (2pt);
			\end{tikzpicture}
			\qquad
			\begin{tikzpicture}[thick]
				\draw (0:0cm) -- (60:1cm) -- (0:1cm) -- (0:0cm);
				\filldraw [black] (0:0cm) node[anchor=east] {$\left( \begin{smallmatrix} 1\\ 0\\ 0 \end{smallmatrix}\right)$} circle (2pt)
								  (0:1cm) node[anchor=west] {$\left( \begin{smallmatrix} 0\\ 1\\ 0 \end{smallmatrix}\right)$} circle (2pt)
								  (60:1cm) node[anchor=west] {$\left( \begin{smallmatrix} 0\\ 0\\ 1 \end{smallmatrix}\right)$} circle (2pt);
			\end{tikzpicture}
			\qquad
			\begin{tikzpicture}[thick]
				\draw (0:0cm) -- (0:2cm);
				\draw (0:0cm) -- (120:2cm);
				\draw (120:2cm) -- (-0.5cm,1.5*1.732cm);
				\draw [dotted](-0.5cm,1.5*1.732cm) -- (1.5cm,1.5*1.732cm);
				\draw [dotted] (2cm,0cm) -- (2.5cm,0.5*1.732cm);
				\draw [dotted] (1.5cm,1.5*1.732cm) -- (2.5cm,0.5*1.732cm);
				\draw [dashed] (120:2cm) -- (60:2cm) -- (0:2cm);
				\filldraw [black] (0:0cm) node[anchor=east] {$\left( \begin{smallmatrix} 4\\ 2\\ 1 \end{smallmatrix}\right)$} circle (2pt)
								  (0:1cm) circle (2pt)
							  	  (0:2cm) node[anchor=west] {$\left( \begin{smallmatrix} 2\\ 4\\ 1 \end{smallmatrix}\right)$} circle (2pt)
								  (120:1cm) circle (2pt)
								  (120:2cm) node[anchor=east] {$\left( \begin{smallmatrix} 4\\ 0\\ 3 \end{smallmatrix}\right)$} circle (2pt)
								  (-0.5cm,1.5*1.732cm) node[anchor=east] {$\left( \begin{smallmatrix} 3\\ 0\\ 4 \end{smallmatrix}\right)$} circle (2pt)
								  (60:1cm) circle (2pt)
								  (60:2cm) circle (2pt)
								  (60:3cm) node[anchor=west] {$\left( \begin{smallmatrix} 1\\ 2\\ 4 \end{smallmatrix}\right)$}circle (2pt)
								  (0cm,1.732cm) circle (2pt)
								  (0.5cm,1.5*1.732cm) circle (2pt)
								  (2cm,1.732cm) circle (2pt)
								  (1.5cm,0.5*1.732cm) circle (2pt)
								  (2.5cm,0.5*1.732cm) node[anchor=west] {$\left( \begin{smallmatrix} 1\\ 4\\ 2 \end{smallmatrix}\right)$} circle (2pt);
		 \end{tikzpicture}
	 	 \\[2mm]
		 \rule{2.0cm}{0cm}$\As_2^{c_2}$ \rule{1.5cm}{0cm} $+$ \rule{1.25cm}{0cm} $\Delta_{\{1,2,3\}}$ \rule{0.75cm}{0cm} 
			$=$ \rule{0.0cm}{0cm} $\left(\begin{matrix}
			 						\Delta_1+\Delta_3\\[2mm] + 2\cdot\Delta_{\{1,2\}} + \Delta_{\{1,3\}} + 2\cdot\Delta_{\{2,3\}}
								  \end{matrix}\right)$\\
		$ $\\
	  \end{center}
	  \caption[]{The Minkowski decomposition of~$\As_2^{c_2}$ into dilated faces of~$\Delta_{[n]}$.}
      \label{fig:minkowski_decomp_As_c_2}
      \end{minipage}
      \end{center}
\end{figure}

We could stop here and be fascinated how the M\"obius inversion relates the description by half spaces to the Minkowski
decomposition. Nevertheless, for associahedra with the same normal fan as $\As^c_{n-1}$, we go beyond the alternating 
sum description for~$y_I$ of Ardila, Benedetti \& Doker. In Theorem~\ref{thm:first_y_I_thm}, we significantly simplify 
the alternating sum for each~$y_I$. In fact, each $y_I$ can be expressed as an alternating sum of at most four non-zero 
values~$z_{a(I)}$, $z_{b(I)}$, $z_{c(I)}$ and $z_{d(I)}$ which are tight right-hand sides for certain facet-defining 
inequalities as specified in the theorem. In other words, we extract combinatorial core data for the M\"obius inversion 
of the function~$z_I$ and answer the question which subsets~$J$ of~$I$ are essential to compute~$y_I$ if the associahedron's 
normal fan is the normal fan of~$\As_{n-1}^c$. These sets~$J$ for
fixed~$I$ depend on the choice~$c$ that determines the normal fan. Figure~\ref{fig:example_compute_y_I} illustrates how 
Theorem~\ref{thm:first_y_I_thm} can be used to compute the coefficients $y_I$ for one of the two examples shown in 
Figure~\ref{fig:a3_associahedra}. If the associahedron coincides with some~$\As^c_{n-1}$, Theorem~\ref{thm:second_y_I_thm} 
states a purely combinatorial interpretation of the values~$y_I$. To illustrate this theorem, we 
recompute~$y_I$ for $\As^{c_1}_2$ and $\As^{c_2}_2$ in Examples~\ref{expl:As^c1_2} and~\ref{expl:As^c2_2}. 

\medskip
The outline of the paper is as follows. Section~\ref{sec_assoc_def} summarises necessary known facts about~$\As_{n-1}^c$ 
and indicates some occurrences of the realisations considered here in the mathematical literature. In 
Section~\ref{sec_tight_values_for_z_I} we introduce the notion of an up and down interval decomposition for 
subsets~$I\subseteq [n]$. This decomposition depends on the choice of a Coxeter element~$c$ (or equivalently on a partition
of~$[n]$ induced by~$c$) and is essential to prove Proposition~\ref{prop_tight_values_for_z_I}. This proposition
gives a combinatorial characterisation of all tight values~$z_I$ for~$\As_{n-1}^c$ needed to evaluate $y_I$ using 
Proposition~\ref{prop:ardila}. The main results, Theorem~\ref{thm:first_y_I_thm} and Theorem~\ref{thm:second_y_I_thm}, 
are then stated in Section~\ref{sec:main_results_and_exapmples}. The proof of Theorem~\ref{thm:first_y_I_thm} is long
and convoluted and deferred to Sections~\ref{sec:proof_main_thm} and \ref{sec_characterisation_of_possible_D_I}, 
while Theorem~\ref{thm:second_y_I_thm} is proved under the assumption of Theorem~\ref{thm:first_y_I_thm} in
Section~\ref{sec:main_results_and_exapmples}. To show that Proposition~\ref{prop:ardila} and Theorem~\ref{thm:first_y_I_thm} 
do not hold for polytopes~$P_n(\{z_I\})$ that are not contained in the deformation cone of the classical permutahedron, 
we briefly study a realisation of a $2$-dimensional cyclohedron in Section~\ref{sec_cyclohedron}. 

\medskip
About the same time as some of these results were achieved, V.~Pilaud and F.~Santos showed that the associahedra~$\As_{n-1}^c$
are examples of brick polytopes,~\cite{PilaudSantos1,PilaudSantos2}. One of their results is that any brick polytope can be 
expressed as a Minkowski sum of other brick polytopes. As a consequence, we have two Minkowski decompositions of~$\As_{n-1}^c$
that are extremal in the following 
sense. The first decomposition of~$\As_{n-1}^c$ has a relatively complicated structure with respect to the coefficients~$y_I$
(possibly negative numbers) but is very simple with respect to the polytopes used (faces of a standard simplex). On the other 
hand, the second decomposition of~$\As_{n-1}^c$ has a simple structure in terms its coefficients (they are either~$0$ or~$1$)
but is more complicated with respect to the polytopes used (brick polytopes). At the time of writing, the exact relationship 
of these two decompositions is not properly understood and remains a joint project of V.~Pilaud with the author.
 
\section{Associahedra as generalised permutahedra}\label{sec_assoc_def}
Associahedra form a family of combinatorially equivalent polytopes and can be realised as generalised permutahedra.
Since the combinatorics of a polytope is encoded in its face lattice, we define an associahedron as a polytope with a face
lattice that is isomorphic to the lattice of sets of non-crossing proper diagonals of a convex and plane $(n+2)$-gon~$Q$ 
ordered by reversed inclusion\footnote{A proper diagonal is a line segment connecting a pair of vertices of~$Q$ whose 
relative interior is contained in the interior of~$Q$. A non-proper diagonal is a diagonal that connects vertices adjacent 
in~$\partial Q$ and a degenerate diagonal is a diagonal where the end-points are equal.}. This description immediately tells 
us that the set of $k$-faces is in bijection to the set of triangulations of~$Q$ with $k$ proper diagonals removed. In 
particular, vertices correspond to triangulations and facets correspond to proper diagonals. Since associahedra turn out 
to be simple polytopes, a result of R.~Blind and P.~Mani-Levitska with an elegant proof due to G.~Kalai,~\cite{BlindMani, Kalai}, 
guarantees that the face lattice is already determined by the $1$-skeleton, so it suffices to specify the vertex-edge 
graph to determine the combinatorics of the face-lattice. This graph is known as the flip graph of 
triangulations of~$Q$. 
In 2004, J.-L.~Loday published a beautiful combinatorial description for the vertex coordinates of associahedra constructed
earlier by S.~Shnider, S.~Sternberg and J.~Stasheff,~\cite{StasheffShnider, ShniderSternberg, Loday}. Loday's description 
is in terms of labeled binary trees dual to the triangulations of~$Q$. The construction of S.~Shnider, S.~Sternberg and 
J.~Stasheff as well as Loday's vertex description was subsequently generalised by C.~Hohlweg and C.~Lange,~\cite{HohlwegLange}. 
The latter construction explicitly describes realisations~$\As_{n-1}^c$ of $(n-1)$-dimensional associahedra and exhibits them
as generalised permutahedra. The construction depends on the choice of a Coxeter element~$c$ of the symmetric group~$\Sigma_n$ 
on~$n$ elements.

We now outline the construction of~$\As^c_{n-1}$ and avoid the explicit use of Coxeter elements. Nevertheless, we use Coxeter 
elements in our notation to distinguish between different realisations. It is known that the Coxeter elements are in bijection 
to the certain partitions~$\Do_c \sqcup \Up_c$ of~$[n]$. We will use these partitions to obtain labelings~$Q_c$ of $Q$ and 
refer to~$\Do_c$ as \emph{down set} and to~$\Up_c$ as \emph{up set}. The partitions are
\[
	\Do_c=\{d_1=1<d_2<\cdots<d_\ell=n\}
	\qquad\text{and}\qquad
	\Up_c=\{u_1<u_2<\cdots<u_m\},
\] 
so $n=\ell + m$, $|\Do_c|=\ell\geq 2$ and $|\Up_c|=m$. We now obtain the $c$-labeling~$Q_c$ of~$Q$ with label set~$[n+1]\cup\{0\}$ 
as follows. Pick two vertices of~$Q$ which are the end-points of a path with~$\ell+2$ vertices on the 
boundary of~$Q$, label the vertices of this path counter-clockwise increasing using the label 
set~$\overline{\Do}_c:=\Do_c\cup \{ 0, n+1\}$ and label the remaining path clockwise increasing using 
the label set~$\Up_c$. The labeling~$Q_c$ has the property that the label set~$\Do_c$ is always on the 
right-hand side of the diagonal $\{0,n+1\}$ oriented from~$0$ to~$n+1$. To illustrate these $c$-labelings~$Q_c$,
observe that there are four distinct partitions $\Do_c\sqcup \Up_c$ for $n=4$ which yield the four labeled
hexagons~$Q_c$ shown in Figure~\ref{fig:c-labelings_hexagon}.
\begin{figure}
	\centerline{
	\begin{tikzpicture}[thick]
		\draw (30:1cm) -- (90:1cm) -- (150:1cm) -- (210:1cm) -- (270:1cm) -- (330:1cm) -- (30:1cm);
		\filldraw [black] (30:1cm)   node[anchor=west] {\scriptsize$4$} circle (2pt)
						  (90:1cm)  node[anchor=south] {\scriptsize$5$} circle (2pt)
						  (150:1cm) node[anchor=east] {\scriptsize$0$} circle (2pt)
						  (210:1cm) node[anchor=east] {\scriptsize$1$} circle (2pt)
						  (270:1cm) node[anchor=north] {\scriptsize$2$} circle (2pt)
						  (330:1cm) node[anchor=west] {\scriptsize$3$} circle (2pt);
		\draw (0cm,-2cm) node{$\Do_c=\{1,2,3,4\}$};
		\draw (0cm,-2.5cm) node{$\Up_c=\varnothing$};
	\end{tikzpicture}
	\hspace{0.5cm}
	\begin{tikzpicture}[thick]
		\draw (30:1cm) -- (90:1cm) -- (150:1cm) -- (210:1cm) -- (270:1cm) -- (330:1cm) -- (30:1cm);
		\filldraw [black] (30:1cm)   node[anchor=west] {\scriptsize$5$} circle (2pt)
						  (90:1cm)  node[anchor=south] {\scriptsize$2$} circle (2pt)
						  (150:1cm) node[anchor=east] {\scriptsize$0$} circle (2pt)
						  (210:1cm) node[anchor=east] {\scriptsize$1$} circle (2pt)
						  (270:1cm) node[anchor=north] {\scriptsize$3$} circle (2pt)
						  (330:1cm) node[anchor=west] {\scriptsize$4$} circle (2pt);
		\draw (0cm,-2cm) node{$\Do_c=\{1,3,4\}$};
		\draw (0cm,-2.5cm) node{$\Up_c=\{2\}$};
	\end{tikzpicture}
	\hspace{0.5cm}
	\begin{tikzpicture}[thick]
		\draw (30:1cm) -- (90:1cm) -- (150:1cm) -- (210:1cm) -- (270:1cm) -- (330:1cm) -- (30:1cm);
		\filldraw [black] (30:1cm)   node[anchor=west] {\scriptsize$5$} circle (2pt)
						  (90:1cm)  node[anchor=south] {\scriptsize$3$} circle (2pt)
						  (150:1cm) node[anchor=east] {\scriptsize$0$} circle (2pt)
						  (210:1cm) node[anchor=east] {\scriptsize$1$} circle (2pt)
						  (270:1cm) node[anchor=north] {\scriptsize$2$} circle (2pt)
						  (330:1cm) node[anchor=west] {\scriptsize$4$} circle (2pt);
		\draw (0cm,-2cm) node{$\Do_c=\{1,2,4\}$};
		\draw (0cm,-2.5cm) node{$\Up_c=\{3\}$};
	\end{tikzpicture}
	\hspace{0.5cm}
	\begin{tikzpicture}[thick]
		\draw (30:1cm) -- (90:1cm) -- (150:1cm) -- (210:1cm) -- (270:1cm) -- (330:1cm) -- (30:1cm);
		\filldraw [black] (30:1cm)   node[anchor=west] {\scriptsize$3$} circle (2pt)
						  (90:1cm)  node[anchor=south] {\scriptsize$2$} circle (2pt)
						  (150:1cm) node[anchor=east] {\scriptsize$0$} circle (2pt)
						  (210:1cm) node[anchor=east] {\scriptsize$1$} circle (2pt)
						  (270:1cm) node[anchor=north] {\scriptsize$4$} circle (2pt)
						  (330:1cm) node[anchor=west] {\scriptsize$5$} circle (2pt);
		\draw (0cm,-2cm) node{$\Do_c=\{1,4\}$};
		\draw (0cm,-2.5cm) node{$\Up_c=\{2,3\}$};
	\end{tikzpicture}
	}
	\caption{The four possible $c$-labelings~$Q_c$ of a hexagon.}
	\label{fig:c-labelings_hexagon}
\end{figure}
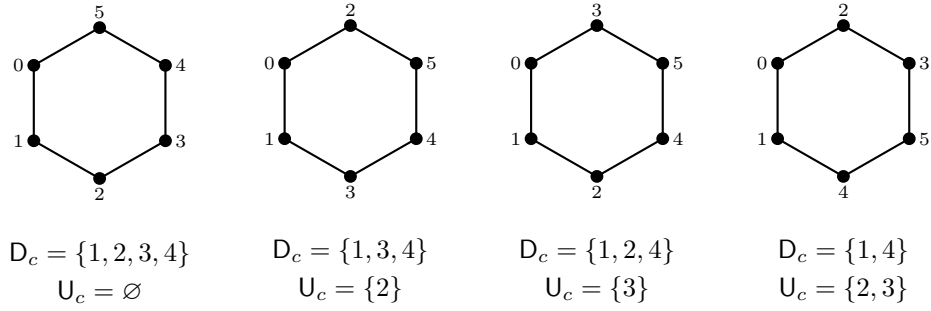
\begin{table}[b]
    \begin{center}
    \begin{minipage}{0.9\linewidth}
	$\Do_c = \{ 1,3,4\}$ and  $\Up_c=\{2\}$:\rule{10cm}{0cm}\\[2mm] 
   \begin{tabular}{l|r|r|r|r|r|r|r|r|r}
  	  $\delta$   & $\{0,3\}$ & $\{0,4\}$ & $\{0,5\}$   & $\{1,2\}$   & $\{1,4\}$ & $\{1,5\}$ & $\{2,3\}$ & $\{2,4\}$   & $\{3,5\}$ \\ \hline
	  $R_\delta$ & $\{1\}$ 	 & $\{1,3\}$ & $\{1,3,4\}$ & $\{2,3,4\}$ & $\{3\}$   & $\{3,4\}$ & $\{1,2\}$ & $\{1,2,3\}$ & $\{4\}$   \\ \hline
	  $\tilde z^c_{R_\delta}$  & 1         & 3       	   & 6           & 6         & 1         & 3         & 3           & 6           & 1 
   \end{tabular}
	$ $\\ $ $\\ $ $\\
	$\Do_c = \{ 1,4\}$ and  $\Up_c=\{2,3\}$:\rule{10cm}{0cm}\\[2mm]
   \begin{tabular}{l|r|r|r|r|r|r|r|r|r}
  	  $\delta$   & $\{0,4\}$  & $\{2,4\}$ & $\{3,4\}$   & $\{0,5\}$ & $\{0,3\}$   & $\{1,2\}$   & $\{2,5\}$   & $\{1,3\}$ & $\{1,5\}$ \\ \hline
	  $R_\delta$ & $\{1\}$ 	  & $\{1,2\}$ & $\{1,2,3\}$ & $\{1,4\}$ & $\{1,3,4\}$ & $\{2,3,4\}$ & $\{1,2,4\}$ & $\{3,4\}$ & $\{4\}$   \\ \hline
	  $\tilde z^c_{R_\delta}$ & $1$         & $3$       & $6$         & $3$       & $6$       & $6$         & $6$       & $3$       & $1$\\
   \end{tabular}
	$ $\\ $ $\\
   \caption{The tables list~$R_\delta$ and~$\tilde z^c_I$ associated to the proper diagonal~$\delta$ of a 
 			labeled hexagon. The upper table corresponds to the associahedron shown on the left of 
			Figure~\ref{fig:a3_associahedra} ($\Do_c=\{1,3,4\}$ and $\Up_c=\{2\}$); the bottom one 
			corresponds to the one on the right ($\Do_c=\{1,4\}$ and $\Up_c=\{2,3\}$).}
	\label{tab:example_table}
\end{minipage}
\end{center}
\end{table}
\enlargethispage{1.5cm}
We derive values~$z_I$ for some subsets~$I\subset [n]$ using oriented proper diagonals of~$Q_c$ as follows. 
Orient each proper diagonal~$\delta$ from the smaller to the larger labeled end-point of~$\delta$, associate 
to~$\delta$ the set~$R_\delta$ that consists of all 
labels on the strict right-hand side of~$\delta$, and replace the elements~$0$ and~$n+1$ by the smaller 
respectively larger label of the end-points contained in~$\Up_c$ if possible. For each proper diagonal~$\delta$ 
we have~$R_\delta \subseteq [n]$ but obviously not every subset of~$[n]$ is of this type if $n>2$. Now set
\[
 	\tilde z_I^c := \begin{cases}
 			  \tfrac{|I|(|I|+1)}{2} & \text{ if }I=R_\delta\text{ for some proper diagonal~$\delta$,}\\
			   -\infty				& \text{ else,}
 			\end{cases}
\]
compare Table~\ref{tab:example_table} for the two associahedra~$\As^c_{3}$ depicted in Figure~\ref{fig:a3_associahedra}
that correspond to two different $c$-labelings of a hexagon. In~\cite{HohlwegLange}
it is shown that $P_n(\{\tilde z_I^c\})$ is in fact an associahedron of dimension~$(n-1)$ realised in~$\mathbb R^n$ 
for every choice of~$c$. In other words, to obtain these associahedra from the classical permutahedron, we make all
inequalities redundant that do not correspond to a proper diagonal of~$Q_c$. Of course, the right-hand sides 
$\tilde z_I^c = -\infty$ are not tight. Proposition~\ref{prop_tight_values_for_z_I} shows how we can compute the tight 
values for~$\tilde z_I^c$ using finite values~$\tilde z_I^c$ of facet-defining inequalities only. Throughout this 
manuscript and for any choice~$c$, the reader 
may refer to this set of tight values~$\{ \tilde z_I^c\}$ to illustrate the results. But we emphasise that this specific 
choice $\{\tilde z_I^c\}$ is only assumed for Statements~\ref{thm:second_y_I_thm}--\ref{cor:characterise_y_I=0}. All 
other results are valid for the larger class of $z_I$-coefficients where is polytope $P_n(\{z_I\})$ is an associahedron
with the same normal fan as some $\As_{n-1}^c$. Proposition~\ref{prop_tight_values_for_z_I} and Theorem~\ref{thm:first_y_I_thm} 
can be applied to this more general situation to obtain tight values for the redundant values~$z_I^c$ and to obtain the 
coefficients~$y_I$ of the Minkowski decomposition into faces of the standard simplex. 
\begin{figure}[b]
      \begin{center}
		 \begin{minipage}{0.95\linewidth}
         \begin{center}
			\begin{tikzpicture}[thick]
				\draw (0:0cm) -- (0:1cm);
				\draw (0:0cm) -- (120:2cm);
				\draw  (120:2cm) -- (60:2cm) -- (1.5cm,0.5*1.732cm) -- (1,0);
				\filldraw [black] (0:0cm) node[anchor=east] {$\left( \begin{smallmatrix} 3\\ 2\\ 1 \end{smallmatrix}\right)$} circle (2pt)
								  (0:1cm) node[anchor=west] {$\left( \begin{smallmatrix} 2\\ 3\\ 1 \end{smallmatrix}\right)$} circle (2pt)
								  (120:1cm) circle (2pt)
								  (120:2cm) node[anchor=east] {$\left( \begin{smallmatrix} 3\\ 0\\ 3 \end{smallmatrix}\right)$} circle (2pt)
								  (60:1cm) circle (2pt)
								  (60:2cm) node[anchor=west] {$\left( \begin{smallmatrix} 1\\ 2\\ 3 \end{smallmatrix}\right)$} circle (2pt)
								  (0cm,1.732cm) circle (2pt)
								  (1.5cm,0.5*1.732cm) node[anchor=west] {$\left( \begin{smallmatrix} 1\\ 3\\ 2 \end{smallmatrix}\right)$} circle (2pt);
			\end{tikzpicture}
			\qquad
			\begin{tikzpicture}[thick]
				\draw (120:1cm) -- (0:0cm) -- (60:1cm) -- (120:1cm);
				\filldraw [black] (0:0cm) node[anchor=east] {$\left( \begin{smallmatrix} 2/3\\ 2/3\\ -1/3 \end{smallmatrix}\right)$} circle (2pt)
								  (120:1cm) node[anchor=east] {$\left( \begin{smallmatrix} 2/3\\ -1/3\\ 2/3 \end{smallmatrix}\right)$} circle (2pt)
								  (60:1cm) node[anchor=west] {$\left( \begin{smallmatrix} -1/3\\ 2/3\\ 2/3 \end{smallmatrix}\right)$} circle (2pt);
			\end{tikzpicture}
			\quad
			\begin{tikzpicture}[thick]
				\draw [dotted](0:0cm) -- (0:1cm);
				\draw [dotted](0:0cm) -- (120:1cm);
				\draw (-0.5cm,0.5*1.732cm) -- (0.5cm,0.5*1.732cm) -- (0:1cm);
				\filldraw [black] (0:0cm) node[anchor=east] {$\left( \begin{smallmatrix} 7/3\\ 4/3\\ 4/3 \end{smallmatrix}\right)$} circle (2pt)
							  	  (0:1cm) node[anchor=west] {$\left( \begin{smallmatrix} 4/3\\ 7/3\\ 4/3 \end{smallmatrix}\right)$} circle (2pt)
								  (120:1cm) node[anchor=east] {$\left( \begin{smallmatrix} 7/3\\ 1/3\\ 7/3 \end{smallmatrix}\right)$} circle (2pt)
								  (60:1cm) node[anchor=west] {$\left( \begin{smallmatrix} 4/3\\ 4/3\\ 7/3 \end{smallmatrix}\right)$} circle (2pt);
		 \end{tikzpicture}
		 $ $\\[2mm]
		 \rule{2.5cm}{0cm}$\As_2^{c_2}$ \rule{1.3cm}{0cm} $=$ \rule{0.90cm}{0cm} $\Phi(\Delta_{\{1,2,3\}})$ \rule{0.5cm}{0cm} 
			$+$ \rule{0.3cm}{0cm} $\left(\begin{matrix} \Phi(\Delta_1)+\Phi(\Delta_2)+\Phi(\Delta_3)\\ 			
														+\Phi(\Delta_{\{1,2\}})+\Phi(\Delta_{\{2,3\}})\end{matrix}\right)$
	  \end{center}
	  \caption[]{The Minkowski decomposition of the $2$-dimensional assiciahedron~$\As_2^{c_1}$ into faces of the simplex $\Phi(\Delta_{[3]})$.}
      \label{fig:second_minkowski_decomp_As_c_2}
      \end{minipage}
      \end{center}
\end{figure}

It is known that realisations~$\As^{c_1}_{n-1}$ and~$\As^{c_2}_{n-1}$ can be linear isometric for certain choices~$c_1$ and~$c_2$ 
and values~$z_I$,~\cite{BergeronHohlwegLangeThomas}. While the two associahedra depicted in Figure~\ref{fig:a3_associahedra} are neither
linear isometric nor do they have the same normal fan, we remark that the associahedra~$\As^{c_1}_{2}$ and~$\As^{c_2}_{2}$ discussed 
in the previous section are linear isometric and the isometry is a point reflection~$\Phi$ in the hyperplane $\sum_{i\in[3]}x_i=6$. 
Although the $z_I$- and $y_I$-values differ for both realisations, they transform according to this isometry. If we consider a 
Minkowski decomposition of~$\As_2^{c_2}$ with respect to the faces of~$\Phi(\Delta_3)$, we obtain precisely the Minkowski coefficients 
of~$\As^{c_1}_2$ with respect to the faces of the standard simplex:
\[
	\As_2^{c_2} = 1\cdot \Phi(\Delta_{\{1\}}) + 1\cdot \Phi(\Delta_{\{2\}}) + 1\cdot \Phi(\Delta_{\{3\}}) 
				+ 1\cdot \Phi(\Delta_{\{1,2\}}) + 0\cdot \Phi(\Delta_{\{1,3\}}) + 1\cdot \Phi(\Delta_{\{2,3\}}) 
				+ 1\cdot \Phi(\Delta_{\{1,2,3\}}),
\]
see Figure~\ref{fig:second_minkowski_decomp_As_c_2} for an illustration. We can weaken this observation a little bit to 
obtain a statement about realisations with linear isomorphic normal fans. Such realisations have been discussed for example
by C.~Ceballos, F.~Santos and G.~M.~Ziegler~\cite{CeballosSantosZiegler}. 
Suppose that~$\Phi$ is a linear isomorphism that maps the normal fan of~$\As^{c_1}_{n-1}$ to the normal fan of~$\As^{c_2}_{n-1}$. 
Then~$\Phi$ induces a transformation between the index sets of the redundant/irredundant inequalities of~$\As^{c_1}_{n-1}$ 
to the redundant/irredundant inequalities of~$\As^{c_2}_{n-1}$. Of course, the right-hand sides of~$\As^{c_1}_{n-1}$ transform 
only into the right-hand sides of~$\As^{c_2}_{n-1}$ if~$\As^{c_2}_{n-1}=\Phi(\As^{c_1}_{n-1})$. Thus we have two Minkowski 
decompositions of~$\As^{c_1}_{n-1}$: one into faces of the standard simplex~$\Delta_n$ as described in Theorem~\ref{thm:first_y_I_thm} 
and another one into faces~$\overline{\Delta}_I$ of~$\Phi(\Delta_n)$. The combinatorial description of the coefficients~$\overline{y}_I$ 
for~$\As^{c_1}_{n-1}$ with respect to faces of~$\Phi(\Delta_n)$ is the same as the description of~$y_I$ for~$\As^{c_2}_{n-1}$ 
with respect to faces of~$\Delta_n$. Of course, to compute the coefficients~$\overline{y}_I$, the values for the right-hand sides 
have to be adjusted to the right-hand sides~$\overline z^{c_1}_I$ of~$\Phi(\As^{c_1}_{n-1})$. As a consequence, the 
combinatorial data that describes the simplification of the M\"obius inversion of Theorem~\ref{thm:first_y_I_thm} is 
already determined by the geometry of the normal fan of~$\As^c_n$ up to linear isomorphism. 

\medskip
We end this section relating~$\As^c_n$ to earlier work. Firstly, we indicate a connection to cambrian fans, 
generalised associahedra and cluster algebras and secondly to convex rank texts and semigraphoids in statistics.
Thirdly, we mention some earlier appearances of specific instances of~$\As^c_{n-1}$ in the literature.

S.~Fomin and A.~Zelevinsky introduced generalised associahedra in the context of cluster algebras of finite 
type,~\cite{FominZelevinsky_Y_systems}, and it is well-known that associahedra are generalised associahedra 
associated to cluster algebras of type~$A$. The construction of~\cite{HohlwegLange} was subsequently generalised 
by C.~Hohlweg, C.~Lange, and H.~Thomas to generalised associahedra,~\cite{HohlwegLangeThomas}, and depends also 
on a Coxeter element~$c$. The geometry of the normal fans of these realisations are determined by combinatorial 
properties of~$c$ and the normal fans are $c$-cambrian fans (introduced by N.~Reading and D.~Speyer in~\cite{ReadingSpeyer_CambrianFans}). 
Reading and Speyer conjectured the existence of a linear isomorphism between $c$-cambrian fans and $g$-vector 
fans associated to cluster algebras of finite type with acyclic initial seed (the notion of a $g$-vector fan
for cluster algebras was introduced by Fomin and Zelevinsky~\cite{FominZelevinsky_Cluster_IV}).  In~\cite{ReadingSpeyer_CombinatorialFramework}, Reading and Speyer describe and relate cambrian and $g$-vector fans in more detail and prove their conjecture 
up to an assumption of another conjecture of~\cite{FominZelevinsky_Cluster_IV}. S.-W.~Yang and A.~Zelevinsky 
gave an alternative proof of the conjecture of 
Reading and Speyer in~\cite{YangZelevinsky_ClusterAlgebrasOfFiniteType}. S.~Stella recently recovered the
the realizations of generalized associahedra for finite type of~\cite{HohlwegLangeThomas} and describes the 
relationship to cluster algebras in detail,~\cite{Stella_PolyhedralModels}.

Generalised permutahedra and therefore the associahedra~$\As^c_{n-1}$ are closely related to the framework of convex 
rank tests and semigraphoids from statistics as discussed by J.~Morton, L.~Pachter, A.~Shiu, B.~Sturmfels, and O.~Wienand~\cite{MortonPachterShiuEtAl}.
The semigraphoid axiom characterises the collection of edges of a permutahedron that can be contracted simultaneously 
to obtain a generalised permutahedron. The authors also study submodular rank tests, its subclass of Minkowski
sum of simplices tests and graphical rank tests. The latter one relates to graph associahedra of M.~Carr and
S.~Devadoss~\cite{CarrDevadoss}. Among the associahedra studied in this manuscript, Loday's realisation fits to 
Minkowski sum of simplices and graphical rank tests. 

Some instances of~$\As^c_{n-1}$ have been studied earlier. For example, the realisations of J.-L.~Loday,~\cite{Loday},
and of G.~Rote, F.~Santos, and I.~Streinu,~\cite{RoteSantosStreinu}, related to one-dimensional point configurations,
are affine equivalent to $\As_{n-1}^c$ if $\Up_c=\varnothing$ or $\Up_c=[n]\setminus\{1,n\}$. For $\Up_c = \varnothing$, the 
Minkowski decomposition into faces of a standard simplex is described by Postnikov in~\cite{Postnikov}. Moreover, G.~Rote, 
F.~Santos, and I.~Streinu point out in Section~$5.3$ that their realisation is not affine equivalent to the realisation 
of F.~Chapoton, S.~Fomin, and A.~Zelevinsky,~\cite{ChapotonFominZelevinsky}. It is not difficult to show that the realisation 
described in~\cite{ChapotonFominZelevinsky} is affine equivalent to $\As_{3}^c$ if $\Up_c=\{2\}$ or $\Up_c=\{3\}$.

\section{Tight values for all $z_I^c$ for $\As_{n-1}^c$} \label{sec_tight_values_for_z_I}

Since the facet-defining inequalities for~$\As^c_{n-1}$ correspond to proper diagonals of~$Q_c$, we know 
precisely the irredundant inequalities for the generalised permutahedron~$P_n(\{\tilde z_I^c\})$. In this 
section, we determine tight values~$\tilde z^c_I$ for all $I\subseteq [n]$ corresponding to redundant 
inequalities in order to be able to compute the coefficients~$y_I$ of the Minkowski decomposition 
of~$\As_{n-1}^c$ as described by Proposition~\ref{prop:ardila}. The concept of an up and down interval 
decomposition induced by the partitioning $\Do_c\sqcup\Up_c$ (or, equivalently, induced by~$c$) of a 
given interval $I\subset[n]$ is a key concept that we introduce first, it allows us to describe any $I\subseteq [n]$ 
in terms of unions and intersections of sets~$R_\delta$ for certain proper diagonals determined by 
this decomposition (or, equivalently, as unions of set differences of certain sets~$R_\delta$ and their complements).

\begin{defn}[up and down intervals]$ $\\
   Let~$\Do_c=\{d_1=1<d_2<\cdots<d_\ell=n\}$ and~$\Up_c=\{u_1<u_2<\cdots<u_m\}$ be the
   partition of~$[n]$ induced by a Coxeter element~$c$.
   \begin{compactenum}[(a)]
   	  \item A set~$S \subseteq [n]$ is a non-empty interval of~$[n]$ if $S = \{ r, r+1, \cdots, s\}$ 
 			for some $0 < r \leq s < n$. We write $S$ as closed interval $[r,s]$ (end-points included) 
			or as open interval $(r-1,s+1)$ (end-points excluded). An empty interval is an open
			interval~$(k,k+1)$ for some $1\leq k < n$.
	  \item A non-empty open down interval is a set~$S \subseteq \Do_c$ such that 
			$S = \{ d_r<d_{r+1}<\cdots<d_s\}$ for some $1\leq r \leq s \leq \ell$. We
			write~$S$ as open down interval~$(d_{r-1},d_{s+1})_{\Do_c}$ where we allow~$d_{r-1}=0$ 
			and~$d_{s+1}=n+1$, i.e. $d_{r-1},d_{s+1}\in \overline{\Do}_c$. For $1\leq r\leq \ell-1$, 
			we also have the empty down interval~$(d_r,d_{r+1})_{\Do_c}$.   
	  \item A closed up interval is a non-empty set~$S \subseteq \Up_c$ such that 
			$S = \{ u_r<u_{r+1}<\cdots<u_s\}$ for some 
			$1\leq r \leq s \leq \ell$. We
			write~$[u_r,u_s]_{\Up_c}$.
   \end{compactenum}
\end{defn}
We often omit the words {\em open} and {\em closed} when we consider down and up intervals.
There will be no ambiguity, because we are not going to deal with closed down intervals 
or open up intervals. Up intervals are always non-empty, while down intervals may be empty.
It will be useful to distinguish the empty down intervals $(d_r,d_{r+1})_{\Do_c}$ 
and $(d_s,d_{s+1})_{\Do_c}$ if~$r\neq s$ although they are equal as sets.

\medskip
It might be helpful to read the following definition of the up and down interval decomposition
in combination with the following Examples~\ref{ex:up_down_eaxample} and~\ref{expl:types_of_diagonals}.
 
\begin{defn}[up and down interval decomposition] \label{defn:up_down_decomp}$ $\\
	Let~$\Do_c=\{d_1=1<d_2<\cdots<d_\ell=n\}$ and~$\Up_c=\{u_1<u_2<\cdots<u_m\}$ be the
	partition of~$[n]$ induced by a Coxeter element~$c$ and $I\subset [n]$ be non-empty.
	The up and down interval decomposition of type $(v,w)$ of~$I$ is a partition of~$I$ 
	into disjoint up and down intervals~$I^\Up_1, \cdots, I^\Up_w$ and $I^\Do_1,\cdots,I^\Do_v$ 
	obtained by the following procedure.
	\begin{compactenum}[1.]
		\item Suppose there are $\tilde v$ non-empty inclusion maximal down intervals of~$I$
		      \label{item:defn_i}
			  denoted by $\tilde I^\Do_k=(\tilde a_k,\tilde b_k)_{\Do_c}$, 
			  $1\leq k\leq \tilde v$, with $\tilde b_k \leq \tilde a_{k+1}$ for 
			  $1\leq k < \tilde v$. Consider also all empty down intervals
			  $E^{\Do}_i=(d_{r_i},d_{r_i+1})_{\Do_c}$ with 
			  $\tilde b_k\leq d_{r_i}<d_{r_i+1}\leq \tilde a_{k+1}$
			  for $0\leq k \leq \tilde v$ where $\tilde b_0=1$ and $\tilde a_{\tilde v+1}=n$.
			  Denote the open intervals $(\tilde a_i,\tilde b_i)$ and $(d_{r_i},d_{r_i+1})$ 
			  of~$[n]$ by~$\tilde I_i$ and~$E_i$ respectively.
		\item Consider all inclusion maximal up intervals of~$I$ contained in some interval~$\tilde I_i$ 
			  or~$E_i$ obtained in Step $1$ and denote these up intervals by 
			  \[
				I^\Up_1=[\alpha_1,\beta_1]_{\Up_c},\cdots,I^\Up_w=[\alpha_w,\beta_w]_{\Up_c}.
			  \]
			  Without loss of generality, we assume $\alpha_i \leq \beta_i < \alpha_{i+1}$.
			  \label{item:defn_ii}
		\item A down interval~$I^\Do_i = (a_i,b_i)_{\Do_c}$, $1\leq i \leq v$, is a down interval 
			  obtained in Step~$1$ that is either a non-empty down interval~$\tilde I^\Do_k$ or an 
			  empty down interval~$E^{\Do}_k$ with the additional property that there is some up
			  interval~$I^\Up_j$ obtained in Step~$2$ such that $I^\Up_j \subseteq E_k$. Without 
			  loss of generality, we assume $b_i \leq a_{i+1}$ for $1\leq i < v$.
			  \label{item:defn_iii}
	\end{compactenum} 
\end{defn}

\begin{expl}\label{ex:up_down_eaxample}
	$ $\\
	We describe the up and down interval decomposition for three subsets of $[4]$ which is
	partitioned into $\Do_c=\{1,3,4\}$ and $\Up_c=\{2\}$ and encourage the reader to sketch the steps.
	\begin{compactenum}[i)]
		\item Consider $J_1=\{2,3\}$.\\
		 	The only non-empty inclusion maximal down interval of $J_1$ is $\tilde I^\Do_1=(1,4)_{\Do_c}=\{3\}$;
			there are no empty down intervals $E^{\Do}_i$ to be considered. As inclusion maximal up intervals of $J_1$ 
			contained in $\tilde I_1=(1,4)=\{2,3\}$, we identify $I^\Up_1=[2,2]_{\Up_c}=\{2\}$. It follows that the
			up and down interval decomposition of~$J_1$ is $(1,4)_{\Do_c} \sqcup [2,2]_{\Up_c}$. Its type is $(1,1)$
		\item Consider $J_2=\{2\}$.\\
			There is no non-empty inclusion maximal down interval of $J_2$ to be considered, but there is one empty down 
			interval $E^{\Do}_1=(1,3)_{\Do_c}$ such that $E_1=(1,3)=\{2\}$ contains one inclusion maximal up interval 
			$I^\Up_1=[2,2]_{\Up_c}=\{2\}$ of $J_2$. It follows that the up and down interval decomposition of~$J_2$ 
			is $(1,3)_{\Do_c}\sqcup [2,2]_{\Up_c}$. Its type is $(1,1)$.
		\item Consider $J_3=\{2,4\}$.\\
			The only non-empty inclusion maximal down interval of $J_3$ is $\tilde I^\Do_1=(3,5)_{\Do_c}=\{4\}$; there
			is one empty down interval $E^{\Do}_1=(1,3)_{\Do_c}$ such that $E_1$ contains an inclusion maximal up interval 
			of $J_3$, this is the up interval $I^\Up_1=[2,2]_{\Up_c}=\{2\}$. There is no non-empty
			inclusion maximal up interval contained in $\tilde I^\Do_1$. It follows that the up and down interval decomposition
			of~$J_3$ is $\left( (1,3)_{\Do_c}\sqcup [2,2]_{\Up_c}\right) \sqcup \left( (3,5)_{\Do_c}\right)$.
			Its type is $(2,1)$.
	\end{compactenum}
\end{expl}

\begin{defn}[nested up and down interval decomposition, nested components]$ $\\
	Let~$\Do_c=\{d_1=1<d_2<\cdots<d_\ell=n\}$ and~$\Up_c=\{u_1<u_2<\cdots<u_m\}$ be the
	partition of~$[n]$ induced by a Coxeter element~$c$ and $I\subset [n]$ be non-empty.
	\begin{compactenum}[(a)]
		\item The up and down interval decomposition of~$I$ is nested if its type is~$(1,w)$. 
		\item A nested component of~$I$ is an inclusion-maximal subset~$J$ of~$I$ such that 
			the up and down decomposition of~$J$ is nested.
	\end{compactenum}
\end{defn}

The definition of a nested up and down interval decomposition can be rephrased as follows: all up intervals are contained in 
the interval~$(a_1,b_1)$ of~$[n]$ obtained from the unique (empty or non-empty) down interval~$I^\Do_1 = (a_1,b_1)_{\Do_c}$.
The following example describes the up and down interval decompositions of $I=R_\delta$ for all proper diagonals~$\delta$
of~$Q_c$. The situation is illustrated in Figure~\ref{fig:example_R_delta_of_diagonal}. As a consequence, we observe that 
the up and down interval decomposition for $R_\delta$ is always nested if $\delta$ is a proper diagonal.

\begin{expl}\label{expl:types_of_diagonals}$ $\\
	Let~$\Do_c=\{d_1=1<d_2<\cdots<d_\ell=n\}$ and~$\Up_c=\{u_1<u_2<\cdots<u_m\}$ be the
	partition of~$[n]$ induced by a Coxeter element~$c$. The proper diagonals~$\delta=\{a,b\}$, 
	$a<b$, of the $c$-labeled polygon~$Q_c$ are in bijection to certain non-empty proper
	subsets~$I\subset [n]$ that have an up and down interval decomposition of type~$(1,0)$,~$(1,1)$, 
	or~$(1,2)$. More precisely, we have 
	\begin{compactenum}[(a)]
		\item $R_\delta=(a,b)_{\Do_c}$ iff $R_\delta$ has an up and down decomposition of type~$(1,0)$.
		\item $R_\delta=(0,b)_{\Do_c}\cup[u_1,a]_{\Up_c}$ or $R_\delta=(a,n+1)_{\Do_c}\cup[b,u_m]_{\Up_c}$ 
			  iff $R_\delta$ has a decomposition of type~$(1,1)$, compare Figure~\ref{fig:example_R_delta_of_diagonal}
			  for an illustration of these two cases.
		\item $R_\delta=(0,n+1)_{\Do_c}\cup[u_1,a]_{\Up_c}\cup[b,u_m]_{Up_c}$ iff $R_\delta$ has an up and 
			  down decomposition of type~$(1,2)$.
	\end{compactenum}
\end{expl}

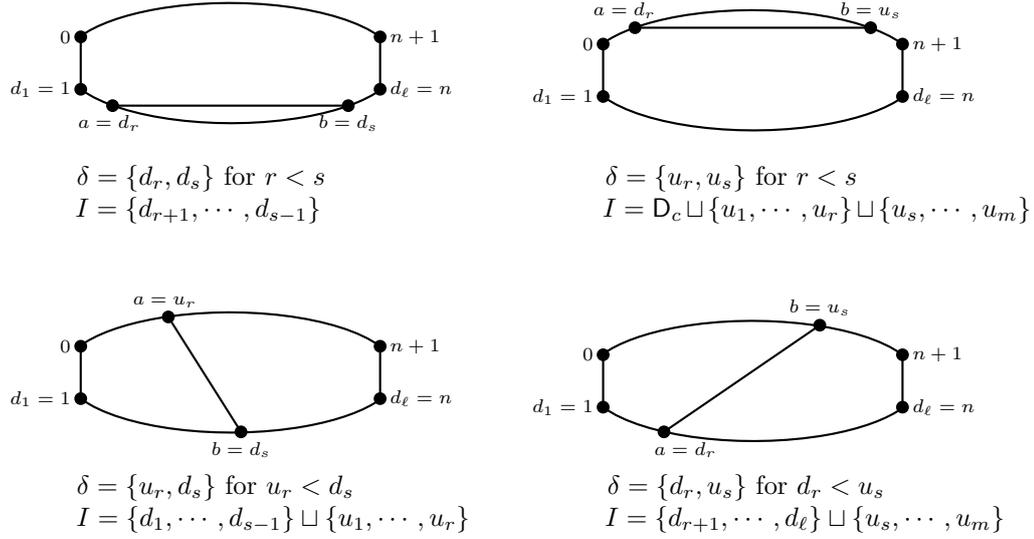
\begin{figure}
      \begin{center}
      \begin{minipage}{0.95\linewidth}
		     \begin{center}
				\begin{tikzpicture}[thick]
					\filldraw [black] (-10:2cm) circle (2pt)
									  (10:2cm) circle (2pt)
								  	  (170:2cm) circle (2pt)
									  (190:2cm) circle (2pt);
					\draw (-10:2cm) node[anchor=west] {\scriptsize$d_\ell=n$} -- (10:2cm) node[anchor=west] {\scriptsize$n+1$}; 
					\draw (170:2cm) node[anchor=east] {\scriptsize$0$} -- (190:2cm) node[anchor=east] {\scriptsize$d_1=1$}; 
					\draw (10:2cm) arc (26:154:2.2cm and 0.8cm);
					\draw (190:2cm) arc (206:334:2.2cm and 0.8cm);
					\filldraw [black] (340:1.65cm) node[anchor=north] {\scriptsize$b=d_{s}$} circle (2pt)
									  (200:1.65cm) node[anchor=north] {\scriptsize$a=d_{r}\ $} circle (2pt);
					\draw (200:1.65cm) -- (340:1.65cm);
			 \end{tikzpicture}
			 \qquad
			 \begin{tikzpicture}[thick]
				\filldraw [black] (-10:2cm) circle (2pt)
								  (10:2cm) circle (2pt)
								  (170:2cm) circle (2pt)
								  (190:2cm) circle (2pt);
				\draw (-10:2cm) node[anchor=west] {\scriptsize$d_\ell=n$} -- (10:2cm) node[anchor=west] {\scriptsize$n+1$}; 
				\draw (170:2cm) node[anchor=east] {\scriptsize$0$} -- (190:2cm) node[anchor=east] {\scriptsize$d_1=1$}; 
				\draw (10:2cm) arc (26:154:2.2cm and 0.8cm);
				\draw (190:2cm) arc (206:334:2.2cm and 0.8cm);
				\filldraw [black] (20:1.65cm) node[anchor=south] {\scriptsize$b=u_s$} circle (2pt)
								  (160:1.65cm) node[anchor=south] {\scriptsize$a=d_r\quad$} circle (2pt);
				\draw (20:1.65cm) -- (160:1.65cm);
			 \end{tikzpicture}
			\end{center}
			 $ $\\[-2mm]
			 \begin{minipage}{\linewidth}
			 	\rule{1.62cm}{0cm}$\delta=\{d_r,d_s\}$ for $r<s$ \rule{3.45cm}{0cm}  $\delta=\{u_r,u_s\}$ for $r<s$\rule{0.3cm}{0cm}\\
			 	\rule{1.5cm}{0cm} $I=\{d_{r+1},\cdots,d_{s-1}\}$ \rule{3.5cm}{0cm} $I=\Do_c\sqcup\{u_1,\cdots,u_r\}\sqcup\{u_s,\cdots,u_m\}$
			 \end{minipage}
	         $ $\\[2mm]
         	 \begin{center}
			 \begin{tikzpicture}[thick]
				\filldraw [black] (-10:2cm) circle (2pt)
								  (10:2cm) circle (2pt)
							  	  (170:2cm) circle (2pt)
								  (190:2cm) circle (2pt);
				\draw (-10:2cm) node[anchor=west] {\scriptsize$d_\ell=n$} -- (10:2cm) node[anchor=west] {\scriptsize$n+1$}; 
				\draw (170:2cm) node[anchor=east] {\scriptsize$0$} -- (190:2cm) node[anchor=east] {\scriptsize$d_1=1$}; 
				\draw (10:2cm) arc (26:154:2.2cm and 0.8cm);
				\draw (190:2cm) arc (206:334:2.2cm and 0.8cm);
				\filldraw [black] (280:0.8cm) node[anchor=north] {\scriptsize$b=d_{s}$} circle (2pt)
								  (138:1.1cm) node[anchor=south] {\scriptsize$a=u_{r}\ $} circle (2pt);
				\draw (138:1.1cm) -- (280:0.8cm);
		 \end{tikzpicture}
		 \qquad
		 \begin{tikzpicture}[thick]
			\filldraw [black] (-10:2cm) circle (2pt)
							  (10:2cm) circle (2pt)
							  (170:2cm) circle (2pt)
							  (190:2cm) circle (2pt);
			\draw (-10:2cm) node[anchor=west] {\scriptsize$d_\ell=n$} -- (10:2cm) node[anchor=west] {\scriptsize$n+1$}; 
			\draw (170:2cm) node[anchor=east] {\scriptsize$0$} -- (190:2cm) node[anchor=east] {\scriptsize$d_1=1$}; 
			\draw (10:2cm) arc (26:154:2.2cm and 0.8cm);
			\draw (190:2cm) arc (206:334:2.2cm and 0.8cm);
			\filldraw [black] (40:1.15cm) node[anchor=south] {\scriptsize$b=u_s$} circle (2pt)
							  (210:1.35cm) node[anchor=north] {\scriptsize$\qquad a=d_r$} circle (2pt);
			\draw (210:1.35cm) -- (40:1.15cm);
		 \end{tikzpicture}
	  	 \end{center}
		 \begin{minipage}{\linewidth}
			\rule{1.62cm}{0cm}$\delta=\{u_r,d_s\}$ for $u_r<d_s$ \rule{3.1cm}{0cm}  $\delta=\{d_r,u_s\}$ for $d_r<u_s$\rule{0.3cm}{0cm}\\
			\rule{1.5cm}{0cm} $I=\{d_1,\cdots,d_{s-1}\}\sqcup\{u_1, \cdots,u_r\}$ \rule{1.55cm}{0cm} $I=\{d_{r+1},\cdots,d_\ell\}\sqcup\{u_s,\cdots,u_m\}$
	 	 \end{minipage}
	  \caption[]{The four possible situations for a diagonal~$\delta=\{a,b\}$ of Example~\ref{expl:types_of_diagonals}.}
      \label{fig:example_R_delta_of_diagonal}
      \end{minipage}
      \end{center}
\end{figure}

\medskip
\noindent
To simplify notation, we extend the definition of $R_\delta$ to the non-proper diagonals $\delta=\{0,u_1\}$ and $\delta=\{u_m,n+1\}$
by defining $R_{\{0,u_1\}}=R_{\{u_m,n+1\}}=[n]$. An example of the diagonals $\delta_{i,j}$ associated to an up and down interval 
decomposition defined in the next Lemma is discussed and illustrated in Example~\ref{expl:associated_diagonals_to decomp} and 
Figure~\ref{fig:associated_diagonals_delta_ij}.
\begin{lem}$ $\\ \label{lem:nested_diag_decomp} 
	Given the partition~$[n]=\Do_c \sqcup \Up_c$ induced by a Coxeter element~$c$. Let $I$ be a	non-empty proper subset 
	of~$[n]$ with up and down interval decomposition of type~$(v,w)$ and nested components of type~$(1,w_1), \cdots, (1,w_v)$. 
	For $1\leq i \leq v$ and $1\leq j \leq w_i$, denote by~$[\alpha_{i,j},\beta_{i,j}]_{\Up_c}$ the inclusion maximal up 
	intervals contained	in the down interval~$(a_i,b_i)_{\Do_c}$ where $\beta_{i,j}<\alpha_{i,j+1}$ and $b_i\leq a_{i+1}$.
	
	If $w_i=0$ then associate to the nested component $(1,w_i)$ the diagonal $\delta_{i,1}=\{a_i,b_i\}$. If $w_i>0$
	then associate to the nested component $(1,w_i)$ the diagonals
	\[
		\delta_{i,1}   := \{a_i,\alpha_{i,1}\}, \qquad
		\delta_{i,j}   := \{ \beta_{i,j-1},\alpha_{i,j}\} \text{ for $1<j\leq w_i$},\text{ and }\qquad
		\delta_{i,w_i+1} := \{ \beta_{i,w_i},b \}.
	\]
	Then the diagonals $\delta_{i,j}$ are non-crossing and
	\[
				I = \bigcup_{i\in[v]}\  \bigcap_{j\in[w_i+1]}R_{\delta_{i,j}}
				   = \bigcup_{i\in[v]}
						\left (R_{\delta_{i,w_i+1}}
								\setminus
								\Bigl ( \bigcup_{j\in [w_i]} [n]\setminus R_{\delta_{i,j}}
								\Bigr)
						\right ).
	\]
\end{lem}
\begin{proof}
	It follows from the definition of nested components that $\delta_{i,j}$ and $\delta_{i^\prime,j^\prime}$ are non-crossing
	if $i\neq i^\prime$. That $\delta_{i,j}$ and $\delta_{i,j^\prime}$ are non-crossing within a nested component is implied by 
	$\beta_{i,j}<\alpha_{i,j+1}$.
	
	To see the identities on~$I$, we first remark that $I = \bigcap_{j\in[w_1+1]}R_{\delta_{1,j}}$ follows directly from the the 
	up and down interval decomposition of~$I$ and the definition of $R_\delta$ if $I$ has only one nested component. If~$I$ consists 
	of more than one nested component, we obtain the claim since it holds for each nested component separately. The second identity 
	is a simple reformulation of the first. This is easily seen in case of just one nested component: instead of intersecting
	the sets $R_\delta$, we choose $\delta=\delta_{1,w_1+1}$ and remove the complements $[n]\setminus R_{\delta_{1,j}}$, $1\leq j\leq w_1$
	from $R_\delta$. This yields $\bigcap_{j\in[w_1+1]}R_{\delta_{i,j}}$.
\end{proof}

\begin{expl}\label{expl:associated_diagonals_to decomp}$ $\\
	We briefly discuss the diagonals associated to the up and down interval decomposition for the three subsets $J_1=\{2,3\}$, 
	$J_2=\{2\}$ and $J_3=\{2,4\}$ of $[4]$ partitioned by $\Do_c=\{1,3,4\}$ and $\Up_c=\{2\}$. These examples are illustrated 
	in Figure~\ref{fig:associated_diagonals_delta_ij}.
	\begin{compactenum}[i)]
		\item We computed $(1,4)_{\Do_c} \sqcup [2,2]_{\Up_c}$ as up and down interval decomposition for $J_1$ .
			We therefore have the associated diagonals $\delta_{1,1}=\{1,2\}$ and $\delta_{1,2}=\{2,4\}$.
		\item We computed $(1,3)_{\Do_c} \sqcup [2,2]_{\Up_c}$ as up and down interval decomposition for $J_2=\{2\}$.
			The associated diagonals are $\delta_{1,1}=\{1,2\}$ and $\delta_{1,2}=\{2,3\}$. 
		\item We computed $\left( (1,3)_{\Do_c}\sqcup [2,2]_{\Up_c}\right) \sqcup \left( (3,5)_{\Do_c}\right)$ as up and down 
			interval decomposition for $J_3=\{2,4\}$.
			The associated diagonals are $\delta_{1,1}=\{1,2\}$, $\delta_{1,2}=\{2,3\}$ and $\delta_{2,1}=\{3,5\}$.
	\end{compactenum}
\end{expl}

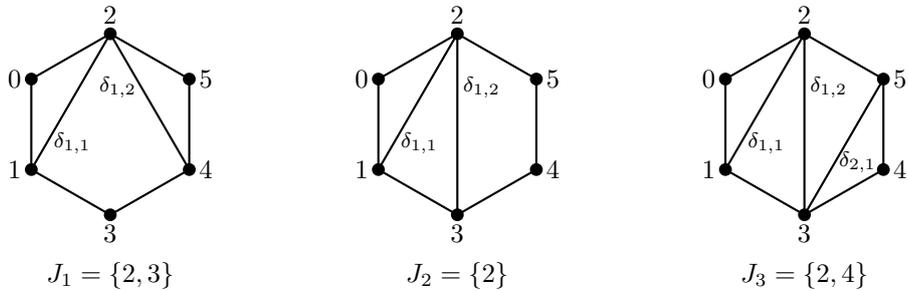
\begin{figure}[b]
      \begin{center}
      \begin{minipage}{0.9\linewidth}
		    \begin{center}
				\begin{tikzpicture}[thick]
					\draw (30:1.2cm) -- (90:1.2cm) -- (150:1.2cm) -- (210:1.2cm) -- (270:1.2cm) -- (330:1.2cm) -- (30:1.2cm);
					\filldraw [black] (30:1.2cm)   node[anchor=west] {$5$} circle (2pt)
									  (90:1.2cm)  node[anchor=south] {$2$} circle (2pt)
									  (150:1.2cm) node[anchor=east] {$0$} circle (2pt)
									  (210:1.2cm) node[anchor=east] {$1$} circle (2pt)
									  (270:1.2cm) node[anchor=north] {$3$} circle (2pt)
									  (330:1.2cm) node[anchor=west] {$4$} circle (2pt);
					\draw (210:1.2cm) -- (90:1.2cm) -- (330:1.2cm);
					\draw (-.5cm,-0.25cm) node{{\footnotesize $\delta_{1,1}$}};
					\draw (.1cm,0.49cm) node{{\footnotesize $\delta_{1,2}$}};
					\draw (0cm,-2cm) node{$J_1=\{2,3\}$};
				\end{tikzpicture}
	$\qquad\qquad$
				\begin{tikzpicture}[thick]
					\draw (30:1.2cm) -- (90:1.2cm) -- (150:1.2cm) -- (210:1.2cm) -- (270:1.2cm) -- (330:1.2cm) -- (30:1.2cm);
					\filldraw [black] (30:1.2cm)   node[anchor=west] {$5$} circle (2pt)
									  (90:1.2cm)  node[anchor=south] {$2$} circle (2pt)
									  (150:1.2cm) node[anchor=east] {$0$} circle (2pt)
									  (210:1.2cm) node[anchor=east] {$1$} circle (2pt)
									  (270:1.2cm) node[anchor=north] {$3$} circle (2pt)
									  (330:1.2cm) node[anchor=west] {$4$} circle (2pt);
					\draw (210:1.2cm) -- (90:1.2cm) -- (270:1.2cm);
					\draw (-.5cm,-0.25cm) node{{\footnotesize $\delta_{1,1}$}};
					\draw (.32cm,0.49cm) node{{\footnotesize $\delta_{1,2}$}};
					\draw (0cm,-2cm) node{$J_2=\{2\}$};
				\end{tikzpicture}
	$\qquad\qquad$
				\begin{tikzpicture}[thick]
					\draw (30:1.2cm) -- (90:1.2cm) -- (150:1.2cm) -- (210:1.2cm) -- (270:1.2cm) -- (330:1.2cm) -- (30:1.2cm);
					\filldraw [black] (30:1.2cm)   node[anchor=west] {$5$} circle (2pt)
									  (90:1.2cm)  node[anchor=south] {$2$} circle (2pt)
									  (150:1.2cm) node[anchor=east] {$0$} circle (2pt)
									  (210:1.2cm) node[anchor=east] {$1$} circle (2pt)
									  (270:1.2cm) node[anchor=north] {$3$} circle (2pt)
									  (330:1.2cm) node[anchor=west] {$4$} circle (2pt);
					\draw (210:1.2cm) -- (90:1.2cm) -- (270:1.2cm) -- (30:1.2cm);
					\draw (-.5cm,-0.25cm) node{{\footnotesize $\delta_{1,1}$}};
					\draw (.32cm,0.49cm) node{{\footnotesize $\delta_{1,2}$}};
					\draw (.7cm,-0.5cm) node{{\footnotesize $\delta_{2,1}$}};
					\draw (0cm,-2cm) node{$J_3=\{2,4\}$};
				\end{tikzpicture}
			\end{center}
     \end{minipage}
		  \caption[]{The associated diagonals~$\delta_{i,j}$ for the three examples considered in Example~\ref{expl:associated_diagonals_to decomp}.}
	      \label{fig:associated_diagonals_delta_ij}
     \end{center}
\end{figure}

The final proposition of this section resolves the quest for tight values~$z_{I}^c$ of all redundant inequalities 
of an associahedron that has the normal fan of~$\As_{n-1}^c$. If we denote this associahedron by $P_n(\{\tilde z^c_I\})$, 
then the inequalities that correspond to an index set~$I=R_\delta$ for some proper diagonal of~$Q_c$ are precisely the facet 
defining inequalities and all other inequalities are redundant.

\begin{prop}$ $\\
	Given the partition~$[n]=\Do_c \sqcup \Up_c$ induced by a Coxeter element~$c$. Let~$I$ be a non-empty proper subset 
	of~$[n]$ with up and down interval decomposition of type~$(v,w)$ and nested components of type~$(1,w_1), \cdots, (1,w_v)$. 
	For $1\leq i \leq v$ and $1\leq j \leq w_i$, denote by~$[\alpha_{i,j},\beta_{i,j}]_{\Up_c}$ the inclusion maximal up 
	intervals contained	in the down interval~$(a_i,b_i)_{\Do_c}$ where $\beta_{i,j}<\alpha_{i,j+1}$ and $b_i\leq a_{i+1}$. 
	For non-empty $I\subseteq [n]$ we set
	\[
	 	z_I^c := \sum_{i\in[v]} \left( \sum_{j\in [w_i+1]} \tilde z_{R_{\delta_{i,j}}}^c
											   -w_i\tilde z^c_{[n]} 
								\right).
	\]
	Then $P(\{z_I^c\})=P(\{\tilde z_I^c\})$ and all~$z_I^c$ are tight.
	\label{prop_tight_values_for_z_I}
\end{prop}
\begin{proof}
	The verification of the inequality is a straightforward calculation:
	\begin{align*}
		\sum_{i\in I}x_i 
   		   &= \sum_{i\in [v]}\ \sum_{k\in\bigcap_{j\in[w_i+1]}R_{\delta_{i,j}}}x_k\\
		   &= \sum_{i\in[v]}
			  \left( \sum_{k\in[v]}x_k -\sum_{j\in[w_{i}+1]}\ \sum_{k\in [n]\setminus R_{\delta_{i,j}}}x_k\right)\\
		   &\geq \sum_{i\in[v]} 
				     \left( \tilde z_{R_{[n]}}^c
							+\sum_{j\in [w_i+1]} \left(\tilde z_{R_{\delta_{i,j}}}^c - \tilde z^c_{[n]}\right)
					 \right).
	\end{align*}
	The first equality is an application of Lemma~\ref{lem:nested_diag_decomp} and the second equality 
	is a simple reformulation. The inequality holds, since $\sum_{i\in R_\delta}x_i \geq \tilde z_{R_\delta}^c$ is
	equivalent to $-\sum_{i \in [n]\setminus R_{\delta}}x_i \geq \tilde z_{R_{\delta}}^c - z_{[n]}$ 
	for every proper diagonal~$\delta$. 
\end{proof}

\begin{defn}$ $\\
	Let~$I$ be a non-empty proper subset of~$[n]$ with up and down interval decomposition of type~$(v,w)$ and nested 
	components of type~$(1,w_1), \cdots, (1,w_v)$. Following notation of Proposition~\ref{prop_tight_values_for_z_I}, 
	we associate diagonals $\delta_{i,j}$ for $1\leq i \leq v$ and $1\leq j \leq w_i$. 
	
	The subset $\mathcal D_I$ of proper diagonals of $\set{\delta_{i,j}}{1\leq i \leq v\text{ and }1\leq j}$
	is called set of proper diagonals associated to~$I$. Similarly, we say that $\delta\in\mathcal D_I$
	is a proper diagonal associated to~$I$. 
\end{defn}

We make a few remarks. First, if a non-proper diagonal $\delta=\{0,u_1\}$ or $\delta=\{u_m,n+1\}$ occurs as a diagonal 
associated to the first or last nested component, the formula for $z_I^c$ in Proposition~\ref{prop_tight_values_for_z_I} 
can be simplified by cancelation of the corresponding terms $\tilde z^c_{[n]}$. Second, for any proper
diagonal $\delta$ of $Q_c$, we obtain $z^c_{R_\delta}=\tilde z^c_{R_\delta}$. An finally, we can characterise the 
face of~$P(\{\tilde z^c_I\})$ that minimises the linear functional $\sum_{i\in I} x_i$ for a given non-empty and proper 
subset $I\subset [n]$.

\begin{cor}$ $\\
	Associate the linear functional $\varphi_I(x)=\sum_{i\in I} x_i$ to a non-empty proper subset~$I\subset[n]$ and 
	denote the facet of~$P(\{\tilde z^c_I\})$ that is supported by $\sum_{i\in R_\delta} x_i=\tilde z^c_{R_\delta}$ 
	for the proper diagonal $\delta$ by $F_{R_\delta}$. Then the intersection $\bigcap_{\delta\in\mathcal D_I}F_{R_\delta}$ 
	is the minimizing face of~$P(\{\tilde z^c_I\})$ for $\varphi_I$.
\end{cor}

\section{Main results and examples} \label{sec:main_results_and_exapmples}

Substitution of Proposition~\ref{prop_tight_values_for_z_I} into Proposition~\ref{prop:ardila} provides a way to compute all 
Minkowski coefficients~$y_I$ since all tight values $z_I^c$ for $\As^c_{n-1}=P_n(\{ z^c_I\})$ are known:
\begin{equation}\label{horrible_equation}
	y_I = \sum_{J\subseteq I}(-1)^{|I\setminus J|}z^c_J
	    = \sum_{J\subseteq I}(-1)^{|I\setminus J|}
				\sum_{i\in[v_J]}\left (  
									\sum_{j\in [w_i+1]} \tilde z_{R_{\delta^J_{i,j}}}^c - w_i \tilde z^c_{[n]}
								\right).
\end{equation}
The goal of this section is to provide two simpler formulae for $y_I$. The first one, given in Theorem~\ref{thm:first_y_I_thm}, 
simplifies Formula~(\ref{horrible_equation}) to at most four non-zero summands for each~$I\subseteq [n]$. The second one, 
stated in Theorem~\ref{thm:second_y_I_thm}, is only valid if the right-hand sides of the facet-defining inequalities 
satisfy $z^c_I=\tfrac{|I|(|I|+1)}{2}$. The values~$y_I$ are then described as a (signed) product of two numbers that measure 
certain paths of~$Q_c$. Theorem~\ref{thm:second_y_I_thm} can be seen as a new aspect to relate combinatorics of the labeled
$n$-gon~$Q_c$ to a construction of~$\As^c_{n-1}$: the coefficients for the Minkowski decomposition into faces of the standard 
simplex can be obtained from the combinatorics of~$Q_c$. Two other relations of the combinatorics of~$Q_c$ to the geometry 
of~$\As_{n-1}^c$ were known before. It is possible to extract the coordinates of the vertices,~\cite{Loday, HohlwegLange}, 
but it is also possible to determine the the facet normals and the right-hand sides for their 
inequalities,~\cite{HohlwegLange}. 

\medskip
From now on, we use the following notation and make some general assumptions unless explicitly mentioned 
otherwise. Let $[n]=\Do_c \sqcup \Up_c$ be the partition of~$[n]$ induced by some fixed Coxeter element~$c$ 
with $\Do_c=\{d_1=1<d_2<\cdots<d_\ell=n\}$ and $\Up_c=\{u_1<\cdots<u_m\}$. A non-empty subset $I\subseteq [n]$ 
with up and  down interval decomposition of type $(v,w)$ has nested components $(1,w_i)$, $1\leq i\leq v$, 
such that the inclusion maximal up intervals~$[\alpha_{i,j},\beta_{i,j}]_{\Up_c}$ contained in the down 
interval~$(a_i,b_i)_{\Do_c}$ satisfy $\beta_{i,j}<\alpha_{i,j+1}$ and $b_i\leq a_{i+1}$. For nested~$I$, 
that is, if $v=1$, we simplify notation and drop one subscript: the up and down interval decomposition 
is $(a,b)_{\Do_c} \cup \bigcup_{j=1}^w [\alpha_j,\beta_j]_{\Up_c}$ where $\alpha_j < \beta_j \leq \alpha_{j+1}$ 
as before. Nevertheless, we do not drop an index for the associated diagonals~$\delta_{ij}$ introduced in
Lemma~\ref{lem:nested_diag_decomp}, we continue to denote them by~$\delta_{i,j}$ or~$\delta_{1,j}$ to avoid
a conflict with the diagonals $\delta_1$, $\delta_2$, $\delta_3$ and $\delta_4$ defined next. To that respect,
we define $\gamma$ (respectively~$\Gamma$) to denote the smallest (respectively largest) element of a nested 
set~$I$ and associate the following four diagonals of the $c$-labeled $(n+2)$-gon~$Q_c$ to a nested set~$I$:
\[
	\delta_1=\{a,b\}, \qquad \delta_2=\{a,\Gamma\}, \qquad \delta_3=\{\gamma,b\}, \qquad\text{and}\qquad \delta_4=\{\gamma,\Gamma\}.
\]  
In general, not all diagonals $\delta_i$ will be proper diagonals, but it will be useful to consider the 
subset~$\mathscr D_I$ of $\{\delta_1,\delta_2,\delta_3,\delta_4\}$ that consists of proper diagonals only. 
We emphasize that the diagonals~$\delta_i$ should be distinguished from the diagonals~$\delta_{i,j}$ defined
in Lemma~\ref{lem:nested_diag_decomp} and the set~$\mathcal D_I$ should be distinguished from~$\mathscr D_I$.

\begin{expl}\label{ex:for_diagonals}
	$ $\\
	We discuss the four diagonals~$\delta_1$, $\delta_2$, $\delta_3$ and $\delta_4$ associated to three subsets of $[4]$ 
	which is partitioned into $\Do_c=\{1,3,4\}$ and $\Up_c=\{2\}$. These associated set $\mathscr D_I$ are illustrated in 	
	Figure~\ref{fig:associated_diagonals_delta_i}.
	\begin{compactenum}[i)]
		\item The up and down interval decomposition of~$J_1=\{2,3\}$ is $(1,4)_{\Do_c} \sqcup [2,2]_{\Up_c}$. Moreover, we
			have $\gamma=2$ and $\Gamma=3$. It follows that 
			\[
				\delta_1=\{1,4\}, \qquad \delta_2=\{1,3\}, \qquad \delta_3=\{2,4\} \qquad\text{and}\qquad \delta_4=\{2,3\}.
			\]  
			In this situation, all diagonals $\delta_i$ except diagonal $\delta_2=\{1,3\}$ are proper diagonals. 
			Therefore, $\mathscr D_I=\{\delta_1, \delta_3, \delta_4\}$
		\item The up and down interval decomposition of~$J_2=\{2\}$ is $(1,3)_{\Do_c}\sqcup [2,2]_{\Up_c}$. Moreover, we have
			$\gamma=\Gamma=2$. This implies
			\[
				\delta_1=\{1,3\}, \qquad \delta_2=\{1,2\}, \qquad \delta_3=\{2,3\} \qquad\text{and}\qquad \delta_4=\{2,2\}.
			\]  
			In this situation, the diagonals~$\delta_1$ and~$\delta_4$ are not proper while the diagonals~$\delta_2$ and $\delta_3$ 
			are proper. Hence, $\mathscr D_I=\{\delta_2, \delta_3\}$.
		\item The set $J_3=\{2,4\}$ is not nested since its up and down interval decomposition is of type~$(2,1)$. We do not associate
			diagonals $\delta_i$ to $J_3$, the set $\mathscr D_I$ is empty.
	\end{compactenum}
\end{expl}

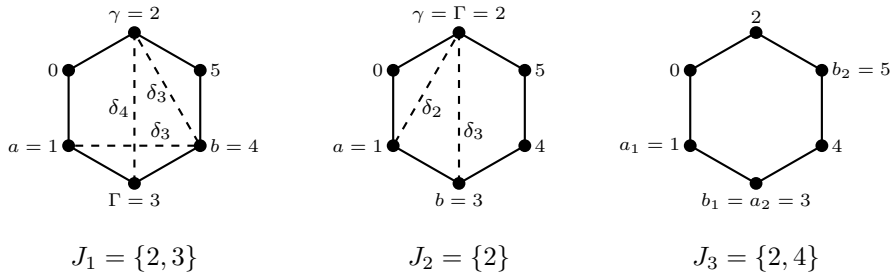
\begin{figure}[b]
      \begin{center}
      \begin{minipage}{0.9\linewidth}
		    \begin{center}
				\begin{tikzpicture}[thick]
					\draw (30:1cm) -- (90:1cm) -- (150:1cm) -- (210:1cm) -- (270:1cm) -- (330:1cm) -- (30:1cm);
					\filldraw [black] (30:1cm)   node[anchor=west] {\scriptsize$5$} circle (2pt)
									  (90:1cm)  node[anchor=south] {\scriptsize$\gamma=2$} circle (2pt)
									  (150:1cm) node[anchor=east] {\scriptsize$0$} circle (2pt)
									  (210:1cm) node[anchor=east] {\scriptsize$a=1$} circle (2pt)
									  (270:1cm) node[anchor=north] {\scriptsize$\Gamma=3$} circle (2pt)
									  (330:1cm) node[anchor=west] {\scriptsize$b=4$} circle (2pt);
					\draw [dashed] (270:1cm) -- (90:1cm) -- (330:1cm) -- (210:1cm);
					\draw (-.2cm,0cm) node{{\footnotesize $\delta_{4}$}};
					\draw (.3cm,0.2cm) node{{\footnotesize $\delta_{3}$}};
					\draw (.35cm,-0.3cm) node{{\footnotesize $\delta_{3}$}};
					\draw (0cm,-2cm) node{$J_1=\{2,3\}$};
				\end{tikzpicture}
				\hspace{0.5cm}
				\begin{tikzpicture}[thick]
					\draw (30:1cm) -- (90:1cm) -- (150:1cm) -- (210:1cm) -- (270:1cm) -- (330:1cm) -- (30:1cm);
					\filldraw [black] (30:1cm)  node[anchor=west] {\scriptsize$5$} circle (2pt)
									  (90:1cm)  node[anchor=south] {\scriptsize$\gamma=\Gamma=2$} circle (2pt)
									  (150:1cm) node[anchor=east] {\scriptsize$0$} circle (2pt)
									  (210:1cm) node[anchor=east] {\scriptsize$a=1$} circle (2pt)
									  (270:1cm) node[anchor=north] {\scriptsize$b=3$} circle (2pt)
									  (330:1cm) node[anchor=west] {\scriptsize$4$} circle (2pt);
					\draw [dashed] (270:1cm) -- (90:1cm) -- (210:1cm);
					\draw (-.35cm,0cm) node{{\footnotesize $\delta_{2}$}};
					\draw (.2cm,-0.3cm) node{{\footnotesize $\delta_{3}$}};
					\draw (0cm,-2cm) node{$J_2=\{2\}$};
				\end{tikzpicture}
				\hspace{0.5cm}
				\begin{tikzpicture}[thick]
					\draw (30:1cm) -- (90:1cm) -- (150:1cm) -- (210:1cm) -- (270:1cm) -- (330:1cm) -- (30:1cm);
					\filldraw [black] (30:1cm)   node[anchor=west] {\scriptsize$b_2=5$} circle (2pt)
									  (90:1cm)  node[anchor=south] {\scriptsize$2$} circle (2pt)
									  (150:1cm) node[anchor=east] {\scriptsize$0$} circle (2pt)
									  (210:1cm) node[anchor=east] {\scriptsize$a_1=1$} circle (2pt)
									  (270:1cm) node[anchor=north] {\scriptsize$b_1=a_2=3$} circle (2pt)
									  (330:1cm) node[anchor=west] {\scriptsize$4$} circle (2pt);
					\draw (0cm,-2cm) node{$J_3=\{2,4\}$};
				\end{tikzpicture}
			\end{center}
     \end{minipage}
		  \caption[]{The diagonals of~$\mathscr D_{J}$ (the proper diagonals among the associated diagonals~$\delta_{i}$) for the three 
		 			examples of Example~\ref{ex:for_diagonals}.}
	      \label{fig:associated_diagonals_delta_i}
     \end{center}
\end{figure}

We now extend our definition of~$R_\delta$ and~$z_{R_\delta}$ to all non-proper and degenerate diagonals~$\delta$. 
If~$\delta=\{0,n+1\}$ and~$\Up_c=\varnothing$ we set $R_\delta:=[n]$ and $z_{R_\delta}^c=z^c_{[n]}$. 
Otherwise, if~$\delta=\{x,y\}$ is not a proper diagonal (different from~$\delta=\{0,n+1\}$ and~$\Up_c=\varnothing$), we set:
\[
	R_\delta := \begin{cases}
					\varnothing & \text{if $x,y\in\overline\Do_c$}\\
					[n]			& \text{otherwise,}
				\end{cases}
\qquad\text{and}\qquad
	z_{R_\delta}^c := \begin{cases}
						0			& \text{if $R_\delta=\varnothing$}\\
						z^c_{[n]}	& \text{if $R_\delta=[n].$}
					\end{cases}
\]
The main result, Theorem~\ref{thm:first_y_I_thm}, actually combines two statements. Firstly, there is a more efficient way 
to compute the coefficients of the Minkowski decomposition of an associahedron~$\As_{n-1}^c=P(\{z^c_I\})$ compared to the 
alternating sum proposed by Proposition~\ref{prop:ardila}. Secondly, the terms~$z^c_I$ for redundant inequalities that are 
needed to compute~$y_I$ are combinatorially characterised and depend on the choice of~$c$ or equivalently on the normal fan 
of~$\As_{n-1}^c$. Of course, their precise values depend on the values~$z^c_I$ of inequalities that are facet-defining.

\begin{thm}\label{thm:first_y_I_thm}$ $\\
	Let $I$ be non-empty subset of~$[n]$. Then the Minkowski coefficient~$y_I$ of~$\As_{n-1}^c=P(\{z^c_I\})$ is 
	\[
		y_I = \begin{cases}
				(-1)^{|I\setminus R_{\delta_1}|}
				\left(
					z_{R_{\delta_1}}^c-z_{R_{\delta_2}}^c-z_{R_{\delta_3}}^c+z_{R_{\delta_4}}^c
				\right)
																			  & \text{if $v=1$,}\\
				0															  & \text{otherwise.}
			  \end{cases}
	\]
\end{thm}
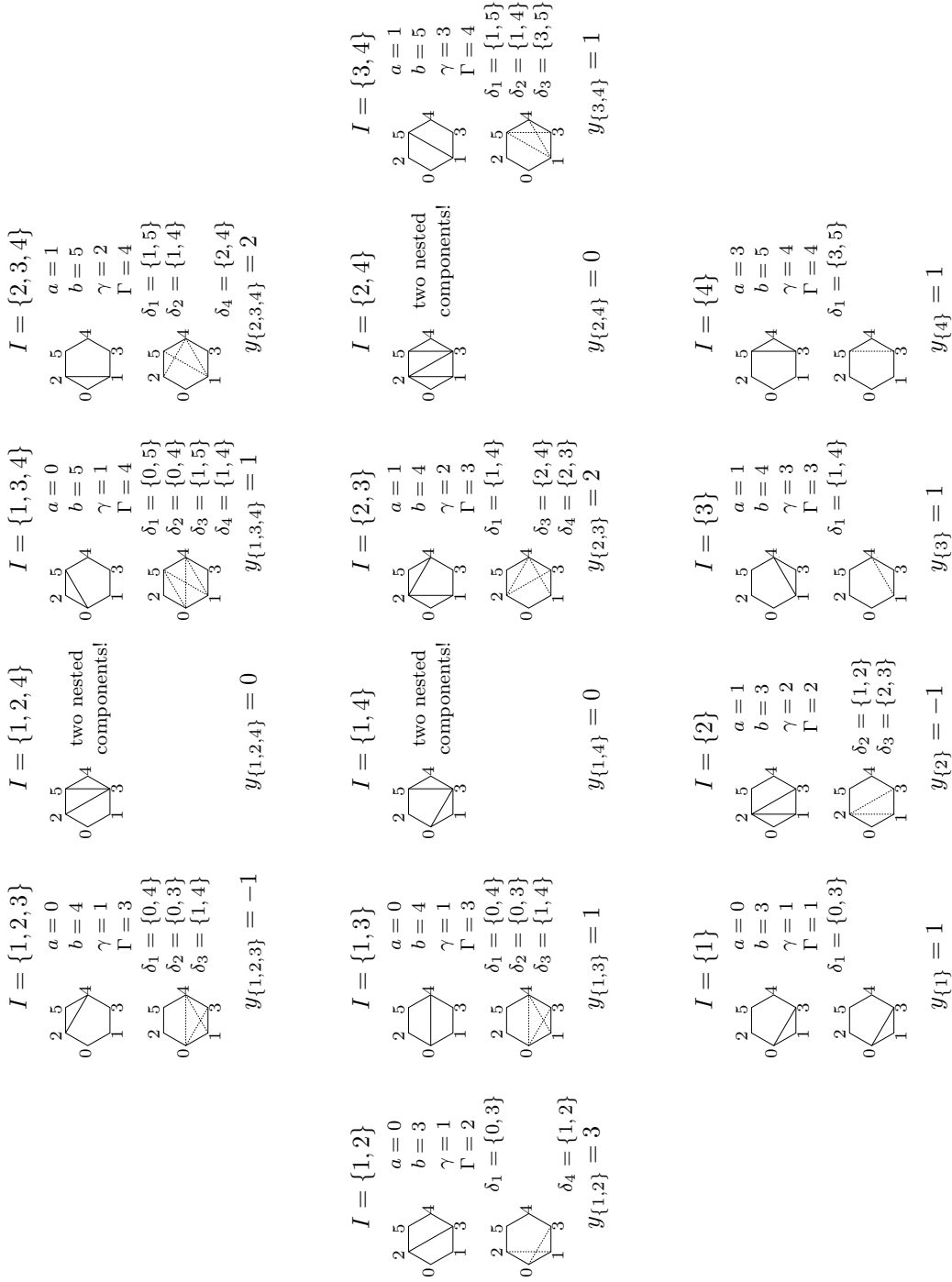
\begin{figure}
      \begin{center}
      \begin{minipage}{\linewidth}
         \begin{center}
		 \vspace{-1.5cm}
         \begin{overpic}[height=0.9\linewidth,angle=90]{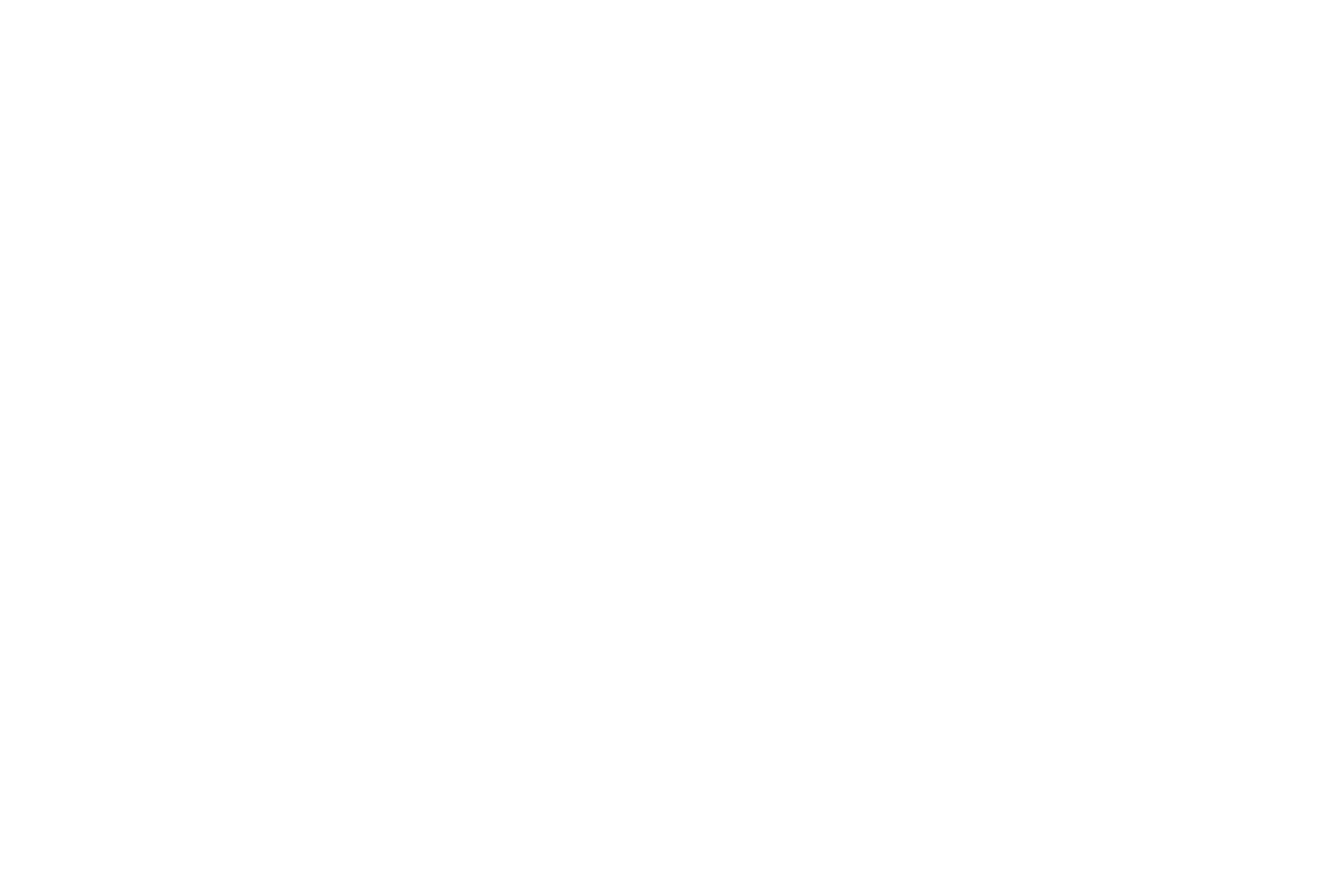}
			\put(-20,90){
				\begin{overpic}[width=\linewidth]{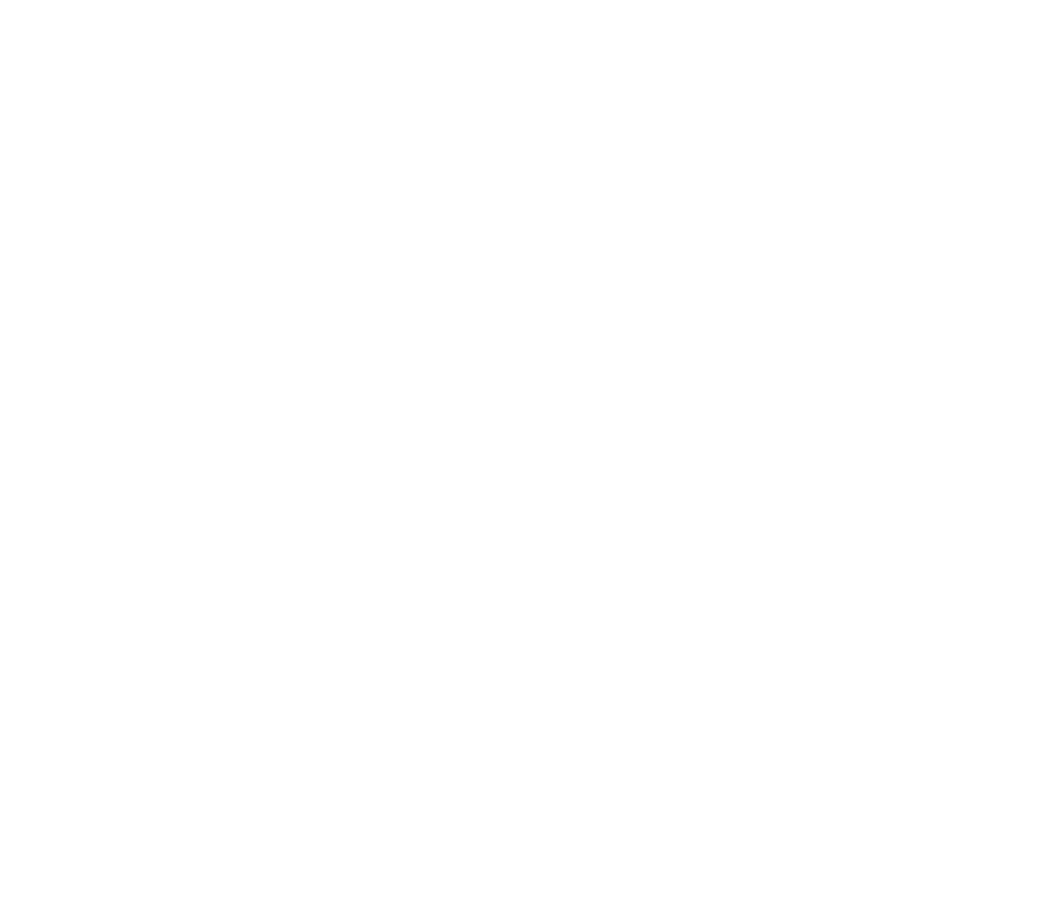}
				\put(0,20){\begin{sideways}$I=\{1,2,3\}$\end{sideways}}
				\put(20,5){
					\begin{overpic}[width=0.05\linewidth,angle=90]{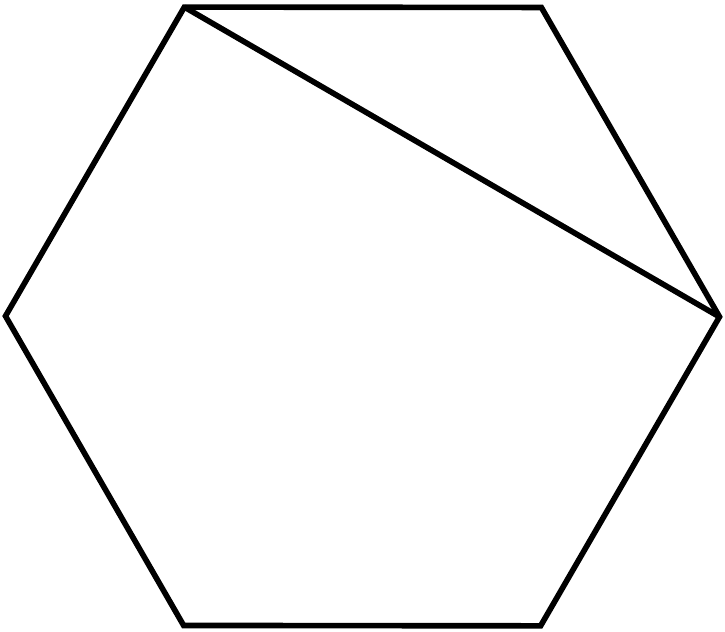}
           				\put(6,-5){\scriptsize \begin{sideways}$0$\end{sideways}}
           				\put(19,3){\scriptsize \begin{sideways}$1$\end{sideways}}
           				\put(-6,2){\scriptsize \begin{sideways}$2$\end{sideways}}
           				\put(19,14){\scriptsize \begin{sideways}$3$\end{sideways}}
           				\put(6,21){\scriptsize \begin{sideways}$4$\end{sideways}}
           				\put(-6,13){\scriptsize \begin{sideways}$5$\end{sideways}}
           				\put(-10,39){\footnotesize \begin{sideways}$\begin{matrix} a=0 \\ b=4 \\ \gamma=1 \\ \Gamma=3 \end{matrix}$\end{sideways}}
					\end{overpic}
					}
				\put(60,5){
					\begin{overpic}[width=0.05\linewidth,angle=90]{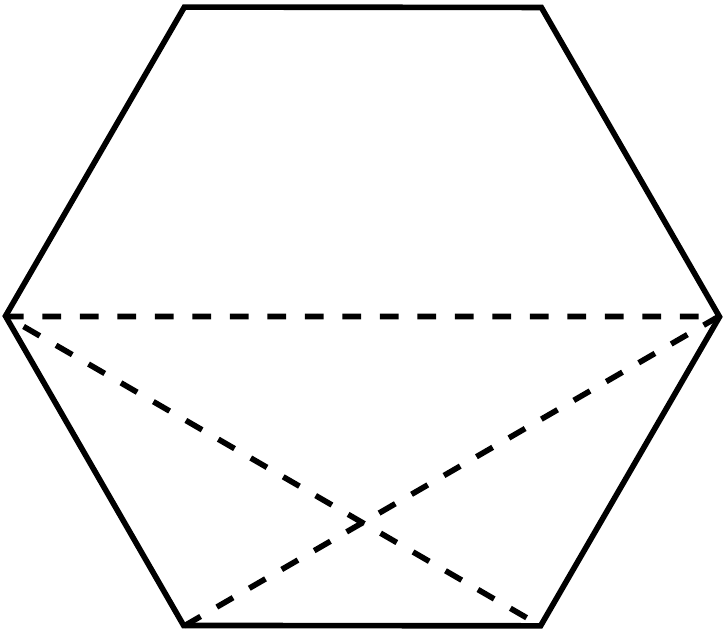}
           				\put(6,-5){\scriptsize \begin{sideways}$0$\end{sideways}}
           				\put(19,2){\scriptsize \begin{sideways}$1$\end{sideways}}
           				\put(-6,3){\scriptsize \begin{sideways}$2$\end{sideways}}
           				\put(19,14){\scriptsize \begin{sideways}$3$\end{sideways}}
           				\put(6,21){\scriptsize \begin{sideways}$4$\end{sideways}}
           				\put(-6,13){\scriptsize \begin{sideways}$5$\end{sideways}}
           				\put(-10,30){\footnotesize \begin{sideways}$\begin{matrix} \delta_1=\{0,4\} \\ 
																 \delta_2=\{0,3\} \\ 
																 \delta_3=\{1,4\} \\ 
																  
												  \end{matrix}$\end{sideways}}
					\end{overpic}
					}
				\put(95,20){\begin{sideways}$y_{\{1,2,3\}}=-1$\end{sideways}}
				\end{overpic}
				}
			\put(-20,180){
				\begin{overpic}[width=\linewidth]{pics/6times8_white_rectangl}
				\put(0,20){\begin{sideways}$I=\{1,2,4\}$\end{sideways}}
				\put(20,5){
					\begin{overpic}[width=0.05\linewidth,angle=90]{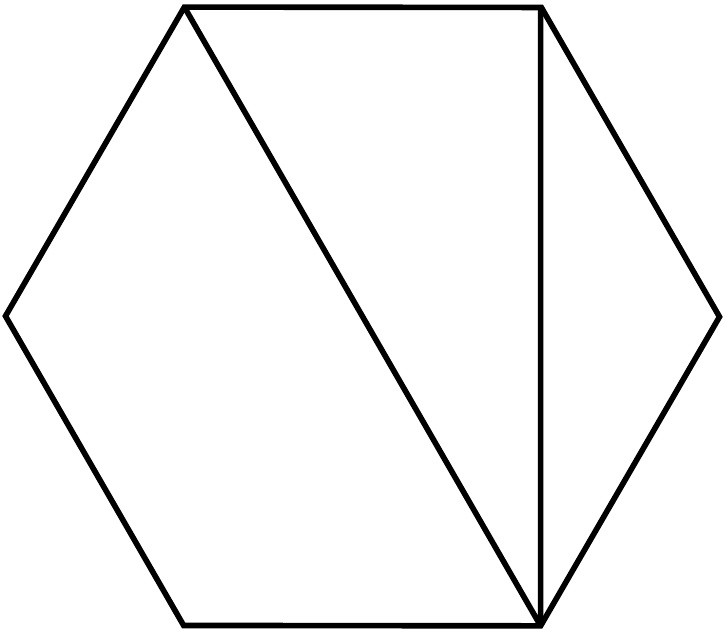}
           				\put(6,-5){\scriptsize \begin{sideways}$0$\end{sideways}}
           				\put(19,3){\scriptsize \begin{sideways}$1$\end{sideways}}
           				\put(-6,2){\scriptsize \begin{sideways}$2$\end{sideways}}
           				\put(19,14){\scriptsize \begin{sideways}$3$\end{sideways}}
           				\put(6,21){\scriptsize \begin{sideways}$4$\end{sideways}}
           				\put(-6,13){\scriptsize \begin{sideways}$5$\end{sideways}}
           				\put(-10,30){\footnotesize \begin{sideways}$\begin{matrix}  \\  \text{two nested}\\  \text{components!}\\  \end{matrix}$\end{sideways}}
					\end{overpic}
					}
				\put(95,20){\begin{sideways}$y_{\{1,2,4\}}=0$\end{sideways}}
				\end{overpic}
				}
			\put(-20,270){
				\begin{overpic}[width=\linewidth]{pics/6times8_white_rectangl}
				\put(0,20){\begin{sideways}$I=\{1,3,4\}$\end{sideways}}
				\put(20,5){
					\begin{overpic}[width=0.05\linewidth,angle=90]{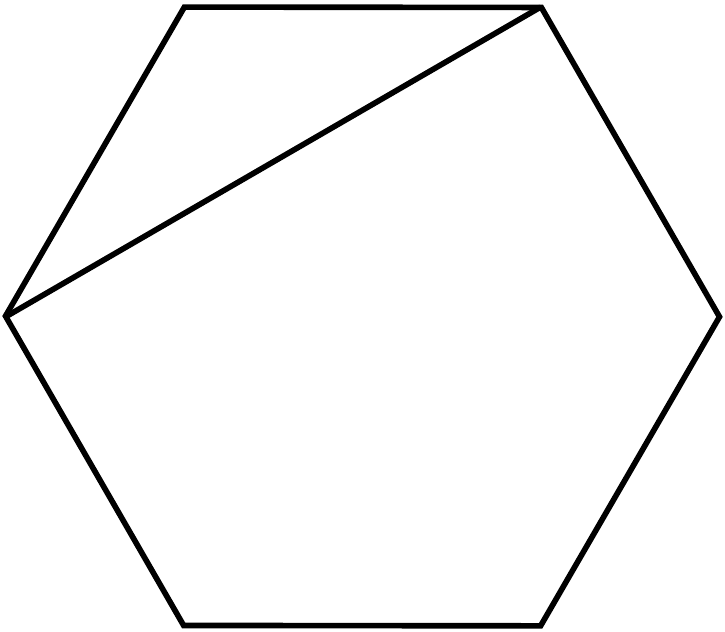}
           				\put(6,-5){\scriptsize \begin{sideways}$0$\end{sideways}}
           				\put(19,3){\scriptsize \begin{sideways}$1$\end{sideways}}
           				\put(-6,2){\scriptsize \begin{sideways}$2$\end{sideways}}
           				\put(19,14){\scriptsize \begin{sideways}$3$\end{sideways}}
           				\put(6,21){\scriptsize \begin{sideways}$4$\end{sideways}}
           				\put(-6,13){\scriptsize \begin{sideways}$5$\end{sideways}}
           				\put(-10,39){\footnotesize \begin{sideways}$\begin{matrix} a=0 \\ b=5 \\ \gamma=1 \\ \Gamma=4 \end{matrix}$\end{sideways}}
					\end{overpic}
					}
				\put(60,5){
					\begin{overpic}[width=0.05\linewidth,angle=90]{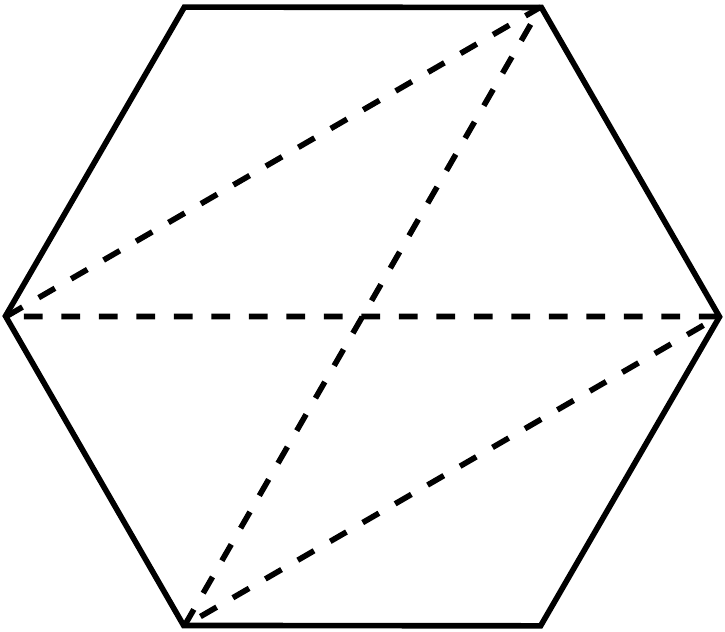}
           				\put(6,-5){\scriptsize \begin{sideways}$0$\end{sideways}}
           				\put(19,2){\scriptsize \begin{sideways}$1$\end{sideways}}
           				\put(-6,3){\scriptsize \begin{sideways}$2$\end{sideways}}
           				\put(19,14){\scriptsize \begin{sideways}$3$\end{sideways}}
           				\put(6,21){\scriptsize \begin{sideways}$4$\end{sideways}}
           				\put(-6,13){\scriptsize \begin{sideways}$5$\end{sideways}}
           				\put(-10,30){\footnotesize \begin{sideways}$\begin{matrix} \delta_1=\{0,5\} \\ 
																 \delta_2=\{0,4\} \\ 
																 \delta_3=\{1,5\} \\ 
																 \delta_4=\{1,4\} 
												  \end{matrix}$\end{sideways}}
					\end{overpic}
					}
				\put(95,20){\begin{sideways}$y_{\{1,3,4\}}=1$\end{sideways}}
				\end{overpic}
				}
			\put(-20,360){
				\begin{overpic}[width=\linewidth]{pics/6times8_white_rectangl}
				\put(0,20){\begin{sideways}$I=\{2,3,4\}$\end{sideways}}
				\put(20,5){
					\begin{overpic}[width=0.05\linewidth,angle=90]{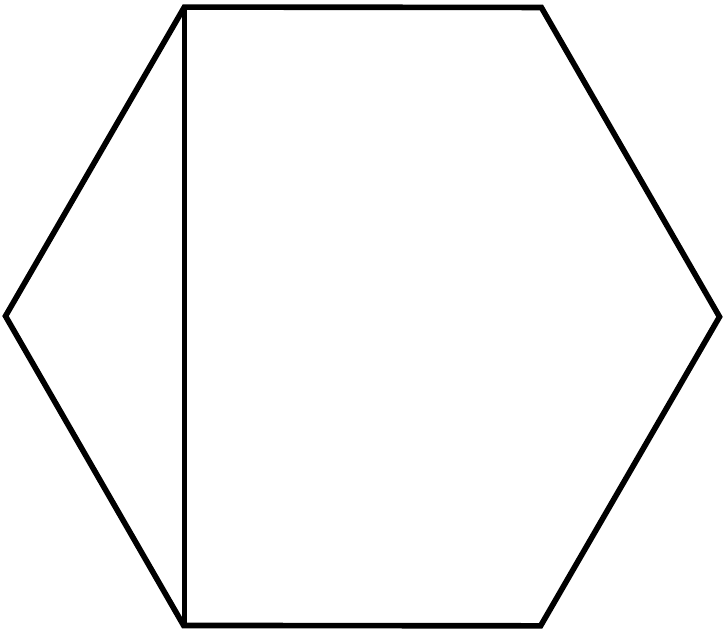}
           				\put(6,-5){\scriptsize \begin{sideways}$0$\end{sideways}}
           				\put(19,3){\scriptsize \begin{sideways}$1$\end{sideways}}
           				\put(-6,2){\scriptsize \begin{sideways}$2$\end{sideways}}
           				\put(19,14){\scriptsize \begin{sideways}$3$\end{sideways}}
           				\put(6,21){\scriptsize \begin{sideways}$4$\end{sideways}}
           				\put(-6,13){\scriptsize \begin{sideways}$5$\end{sideways}}
           				\put(-10,39){\footnotesize \begin{sideways}$\begin{matrix} a=1 \\ b=5 \\ \gamma=2 \\ \Gamma=4 \end{matrix}$\end{sideways}}
					\end{overpic}
					}
				\put(60,5){
					\begin{overpic}[width=0.05\linewidth,angle=90]{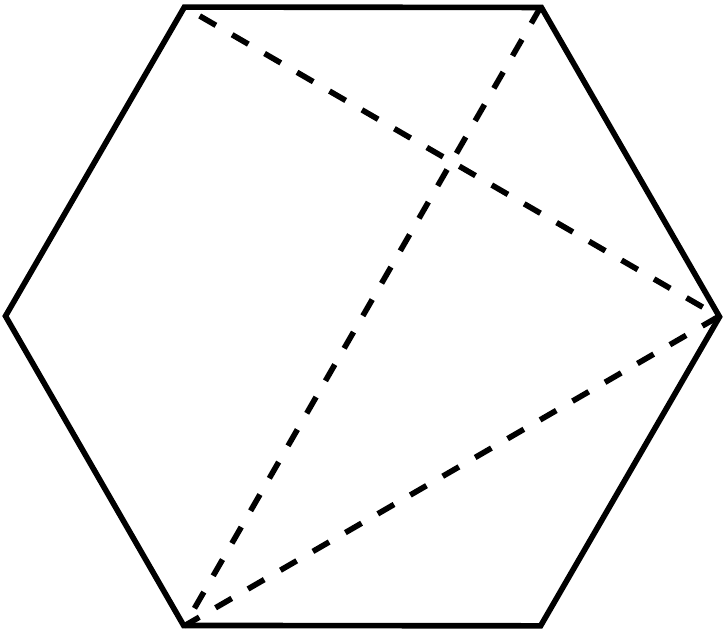}
           				\put(6,-5){\scriptsize \begin{sideways}$0$\end{sideways}}
           				\put(19,2){\scriptsize \begin{sideways}$1$\end{sideways}}
           				\put(-6,3){\scriptsize \begin{sideways}$2$\end{sideways}}
           				\put(19,14){\scriptsize \begin{sideways}$3$\end{sideways}}
           				\put(6,21){\scriptsize \begin{sideways}$4$\end{sideways}}
           				\put(-6,13){\scriptsize \begin{sideways}$5$\end{sideways}}
           				\put(-10,30){\footnotesize \begin{sideways}$\begin{matrix} \delta_1=\{1,5\} \\ 
																 \delta_2=\{1,4\} \\ 
																  \\ 
																 \delta_4=\{2,4\} 
												  \end{matrix}$\end{sideways}}
					\end{overpic}
					}
				\put(95,20){\begin{sideways}$y_{\{2,3,4\}}=2$\end{sideways}}
				\end{overpic}
				}
%
			\put(120,0){
				\begin{overpic}[width=\linewidth]{pics/6times8_white_rectangl}
				\put(00,20){\begin{sideways}$I=\{1,2\}$\end{sideways}}
				\put(20,5){
					\begin{overpic}[width=0.05\linewidth,angle=90]{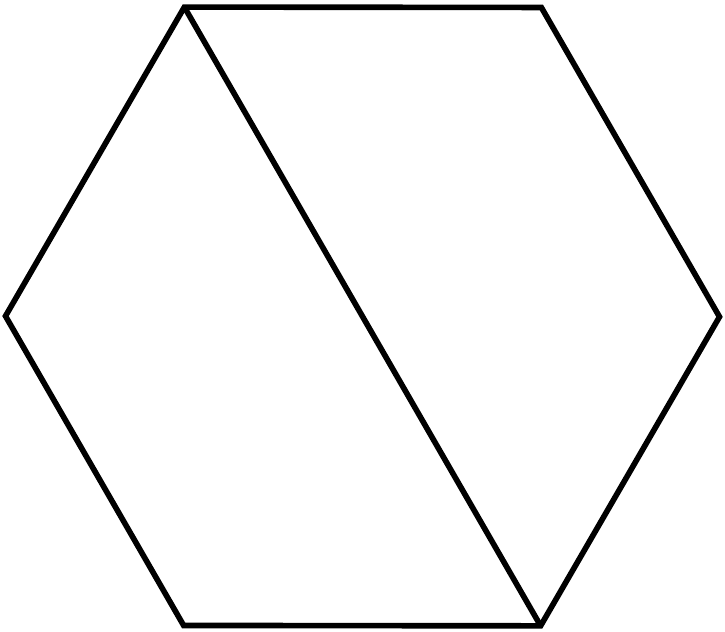}
						\put(6,-5){\scriptsize \begin{sideways}$0$\end{sideways}}
						\put(19,3){\scriptsize \begin{sideways}$1$\end{sideways}}
						\put(-6,2){\scriptsize \begin{sideways}$2$\end{sideways}}
						\put(19,14){\scriptsize \begin{sideways}$3$\end{sideways}}
						\put(6,21){\scriptsize \begin{sideways}$4$\end{sideways}}
						\put(-6,13){\scriptsize \begin{sideways}$5$\end{sideways}}
						\put(-10,39){\footnotesize \begin{sideways}$\begin{matrix} a=0 \\ b=3 \\ \gamma=1 \\ \Gamma=2 \end{matrix}$\end{sideways}}
					\end{overpic}
					}
				\put(60,5){
					\begin{overpic}[width=0.05\linewidth,angle=90]{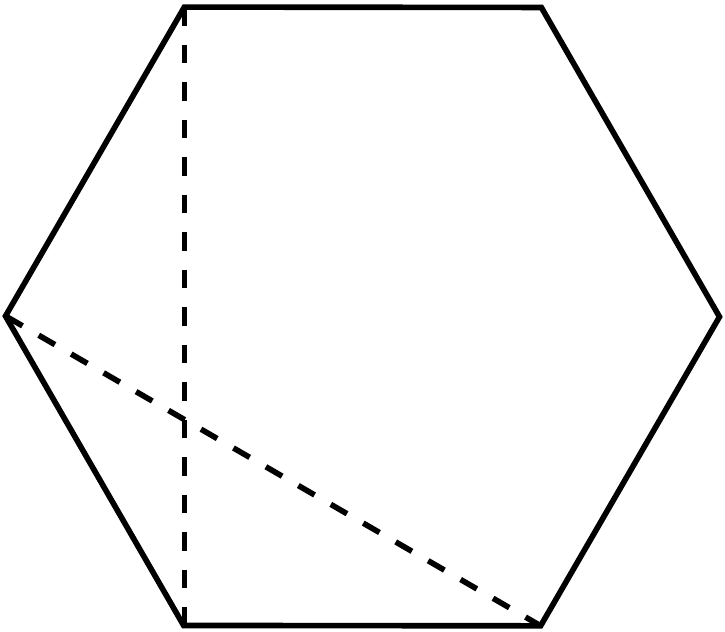}
						\put(6,-5){\scriptsize \begin{sideways}$0$\end{sideways}}
						\put(19,2){\scriptsize \begin{sideways}$1$\end{sideways}}
						\put(-6,3){\scriptsize \begin{sideways}$2$\end{sideways}}
						\put(19,14){\scriptsize \begin{sideways}$3$\end{sideways}}
						\put(6,21){\scriptsize \begin{sideways}$4$\end{sideways}}
						\put(-6,13){\scriptsize \begin{sideways}$5$\end{sideways}}
						\put(-10,30){\footnotesize \begin{sideways}$\begin{matrix} \delta_1=\{0,3\} \\ 
																  \\ 
																  \\ 
																 \delta_4=\{1,2\} 
												  \end{matrix}$\end{sideways}}
					\end{overpic}
					}
				\put(95,20){\begin{sideways}$y_{\{1,2\}}=3$\end{sideways}}
				\end{overpic}
				}
			\put(120,90){
				\begin{overpic}[width=\linewidth]{pics/6times8_white_rectangl}
				\put(00,20){\begin{sideways}$I=\{1,3\}$\end{sideways}}
				\put(20,5){
					\begin{overpic}[width=0.05\linewidth,angle=90]{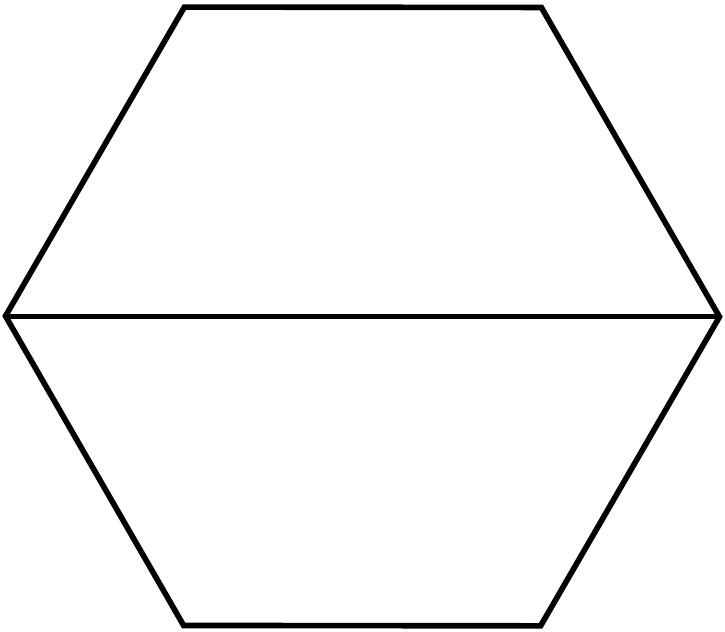}
						\put(6,-5){\scriptsize \begin{sideways}$0$\end{sideways}}
						\put(19,3){\scriptsize \begin{sideways}$1$\end{sideways}}
						\put(-6,2){\scriptsize \begin{sideways}$2$\end{sideways}}
						\put(19,14){\scriptsize \begin{sideways}$3$\end{sideways}}
						\put(6,21){\scriptsize \begin{sideways}$4$\end{sideways}}
						\put(-6,13){\scriptsize \begin{sideways}$5$\end{sideways}}
						\put(-10,39){\footnotesize \begin{sideways}$\begin{matrix} a=0 \\ b=4 \\ \gamma=1 \\ \Gamma=3 \end{matrix}$\end{sideways}}
					\end{overpic}
					}
				\put(60,5){
					\begin{overpic}[width=0.05\linewidth,angle=90]{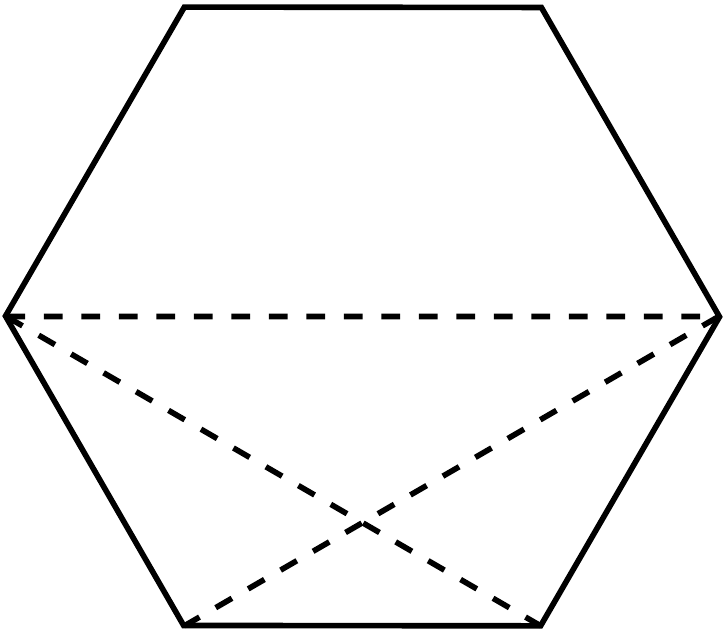}
						\put(6,-5){\scriptsize \begin{sideways}$0$\end{sideways}}
						\put(19,2){\scriptsize \begin{sideways}$1$\end{sideways}}
						\put(-6,3){\scriptsize \begin{sideways}$2$\end{sideways}}
						\put(19,14){\scriptsize \begin{sideways}$3$\end{sideways}}
						\put(6,21){\scriptsize \begin{sideways}$4$\end{sideways}}
						\put(-6,13){\scriptsize \begin{sideways}$5$\end{sideways}}
						\put(-10,30){\footnotesize \begin{sideways}$\begin{matrix} \delta_1=\{0,4\} \\ 
																 \delta_2=\{0,3\} \\ 
																 \delta_3=\{1,4\} \\ 
													  
												  \end{matrix}$\end{sideways}}
					\end{overpic}
					}
				\put(95,20){\begin{sideways}$y_{\{1,3\}}=1$\end{sideways}}
				\end{overpic}
				}
			\put(120,180){
				\begin{overpic}[width=\linewidth]{pics/6times8_white_rectangl}
				\put(00,20){\begin{sideways}$I=\{1,4\}$\end{sideways}}
				\put(20,5){
					\begin{overpic}[width=0.05\linewidth,angle=90]{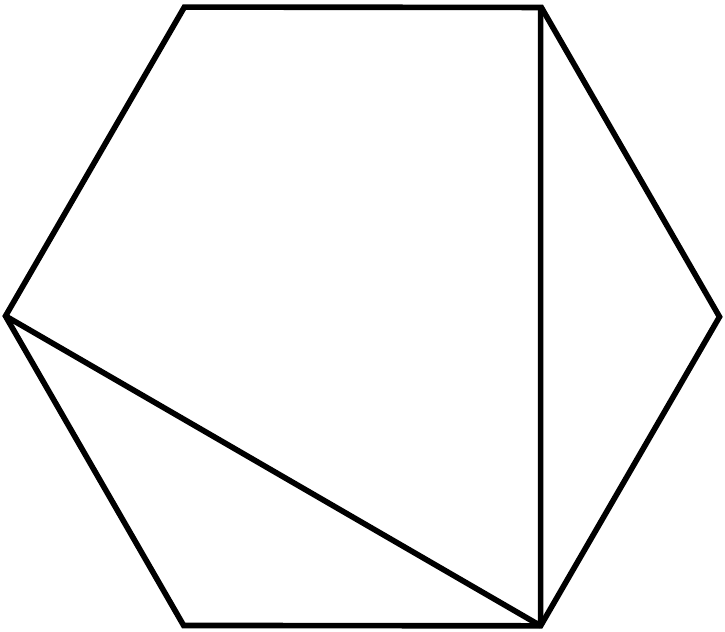}
						\put(6,-5){\scriptsize \begin{sideways}$0$\end{sideways}}
						\put(19,3){\scriptsize \begin{sideways}$1$\end{sideways}}
						\put(-6,2){\scriptsize \begin{sideways}$2$\end{sideways}}
						\put(19,14){\scriptsize \begin{sideways}$3$\end{sideways}}
						\put(6,21){\scriptsize \begin{sideways}$4$\end{sideways}}
						\put(-6,13){\scriptsize \begin{sideways}$5$\end{sideways}}
						\put(-10,30){\footnotesize \begin{sideways}$\begin{matrix}  \\  \text{two nested}\\  \text{components!}\\  \end{matrix}$\end{sideways}}
					\end{overpic}
					}
				\put(95,20){\begin{sideways}$y_{\{1,4\}}=0$\end{sideways}}
				\end{overpic}
				}
			\put(120,270){
				\begin{overpic}[width=\linewidth]{pics/6times8_white_rectangl}
				\put(00,20){\begin{sideways}$I=\{2,3\}$\end{sideways}}
				\put(20,5){
					\begin{overpic}[width=0.05\linewidth,angle=90]{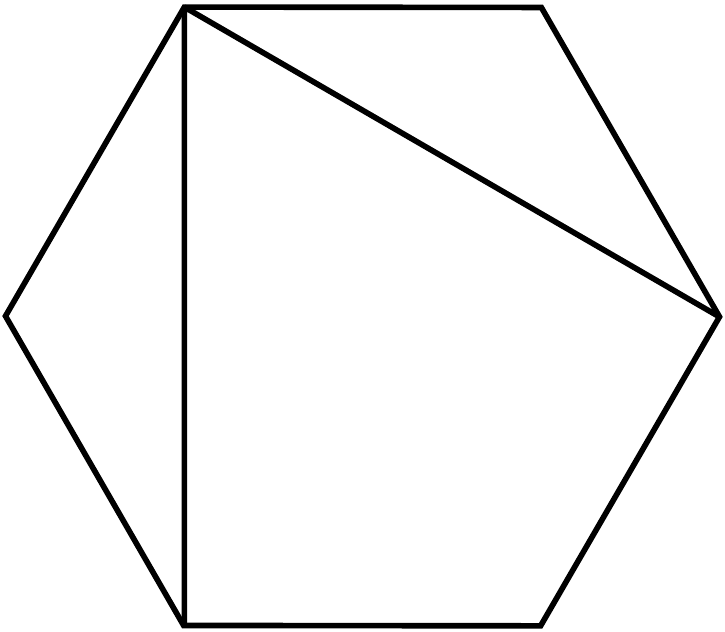}
						\put(6,-5){\scriptsize \begin{sideways}$0$\end{sideways}}
						\put(19,3){\scriptsize \begin{sideways}$1$\end{sideways}}
						\put(-6,2){\scriptsize \begin{sideways}$2$\end{sideways}}
						\put(19,14){\scriptsize \begin{sideways}$3$\end{sideways}}
						\put(6,21){\scriptsize \begin{sideways}$4$\end{sideways}}
						\put(-6,13){\scriptsize \begin{sideways}$5$\end{sideways}}
						\put(-10,39){\footnotesize \begin{sideways}$\begin{matrix} a=1 \\ b=4 \\ \gamma=2 \\ \Gamma=3 \end{matrix}$\end{sideways}}
					\end{overpic}
					}
				\put(60,5){
					\begin{overpic}[width=0.05\linewidth,angle=90]{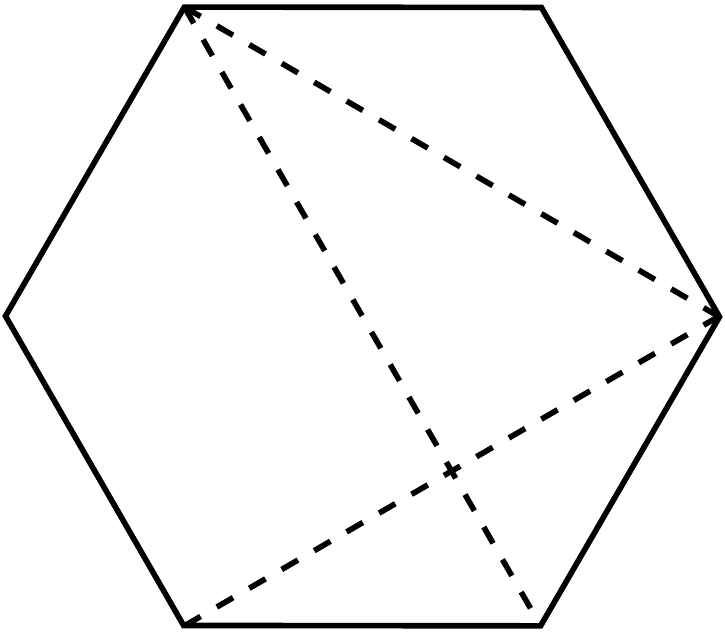}
						\put(6,-5){\scriptsize \begin{sideways}$0$\end{sideways}}
						\put(19,2){\scriptsize \begin{sideways}$1$\end{sideways}}
						\put(-6,3){\scriptsize \begin{sideways}$2$\end{sideways}}
						\put(19,14){\scriptsize \begin{sideways}$3$\end{sideways}}
						\put(6,21){\scriptsize \begin{sideways}$4$\end{sideways}}
						\put(-6,13){\scriptsize \begin{sideways}$5$\end{sideways}}
						\put(-10,30){\footnotesize \begin{sideways}$\begin{matrix} \delta_1=\{1,4\} \\ 
																  \\ 
																 \delta_3=\{2,4\} \\ 
																 \delta_4=\{2,3\} 
												  \end{matrix}$\end{sideways}}
					\end{overpic}
					}
				\put(95,20){\begin{sideways}$y_{\{2,3\}}=2$\end{sideways}}
				\end{overpic}
				}
			\put(120,360){
				\begin{overpic}[width=\linewidth]{pics/6times8_white_rectangl}
				\put(00,20){\begin{sideways}$I=\{2,4\}$\end{sideways}}
				\put(20,5){
					\begin{overpic}[width=0.05\linewidth,angle=90]{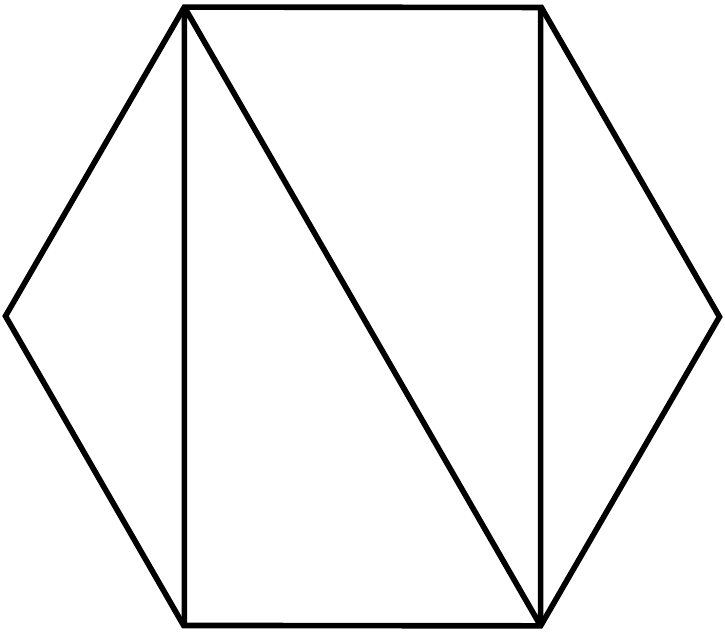}
						\put(6,-5){\scriptsize \begin{sideways}$0$\end{sideways}}
						\put(19,3){\scriptsize \begin{sideways}$1$\end{sideways}}
						\put(-6,2){\scriptsize \begin{sideways}$2$\end{sideways}}
						\put(19,14){\scriptsize \begin{sideways}$3$\end{sideways}}
						\put(6,21){\scriptsize \begin{sideways}$4$\end{sideways}}
						\put(-6,13){\scriptsize \begin{sideways}$5$\end{sideways}}
						\put(-10,30){\footnotesize \begin{sideways}$\begin{matrix}  \\  \text{two nested}\\  \text{components!}\\  \end{matrix}$\end{sideways}}
					\end{overpic}
					}
				\put(95,20){\begin{sideways}$y_{\{2,4\}}=0$\end{sideways}}
				\end{overpic}
				}
			\put(120,450){
				\begin{overpic}[width=\linewidth]{pics/6times8_white_rectangl}
				\put(00,20){\begin{sideways}$I=\{3,4\}$\end{sideways}}
				\put(20,5){
					\begin{overpic}[width=0.05\linewidth,angle=90]{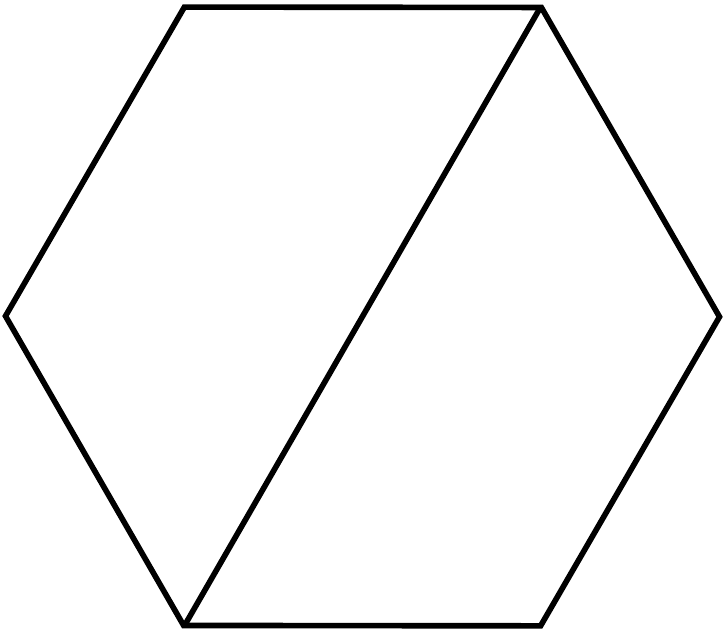}
						\put(6,-5){\scriptsize \begin{sideways}$0$\end{sideways}}
						\put(19,3){\scriptsize \begin{sideways}$1$\end{sideways}}
						\put(-6,2){\scriptsize \begin{sideways}$2$\end{sideways}}
						\put(19,14){\scriptsize \begin{sideways}$3$\end{sideways}}
						\put(6,21){\scriptsize \begin{sideways}$4$\end{sideways}}
						\put(-6,13){\scriptsize \begin{sideways}$5$\end{sideways}}
						\put(-10,39){\footnotesize \begin{sideways}$\begin{matrix} a=1 \\ b=5 \\ \gamma=3 \\ \Gamma=4 \end{matrix}$\end{sideways}}
					\end{overpic}
					}
				\put(60,5){
					\begin{overpic}[width=0.05\linewidth,angle=90]{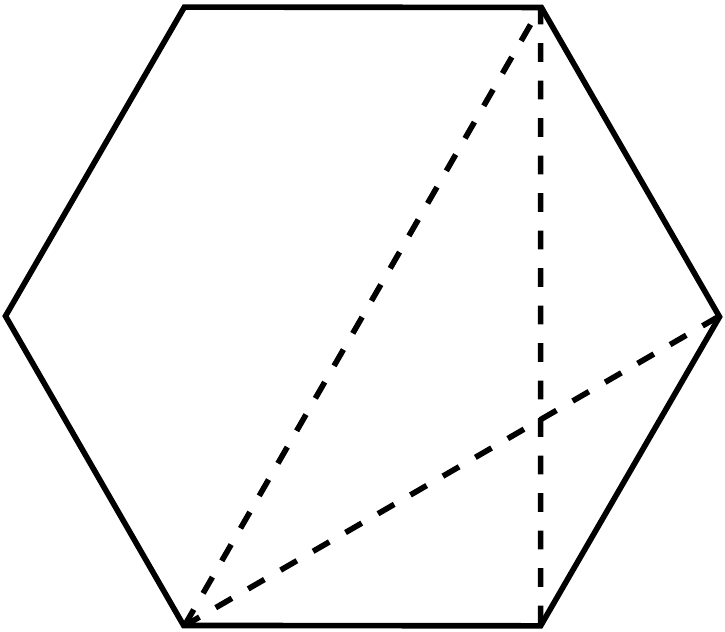}
						\put(6,-5){\scriptsize \begin{sideways}$0$\end{sideways}}
						\put(19,2){\scriptsize \begin{sideways}$1$\end{sideways}}
						\put(-6,3){\scriptsize \begin{sideways}$2$\end{sideways}}
						\put(19,14){\scriptsize \begin{sideways}$3$\end{sideways}}
						\put(6,21){\scriptsize \begin{sideways}$4$\end{sideways}}
						\put(-6,13){\scriptsize \begin{sideways}$5$\end{sideways}}
						\put(-10,30){\footnotesize \begin{sideways}$\begin{matrix} \delta_1=\{1,5\} \\ 
																 \delta_2=\{1,4\} \\ 
																 \delta_3=\{3,5\} \\ 
													  
												  \end{matrix}$\end{sideways}}
					\end{overpic}
					}
				\put(95,20){\begin{sideways}$y_{\{3,4\}}=1$\end{sideways}}
				\end{overpic}
				}
%
%
			\put(260,90){
				\begin{overpic}[width=\linewidth]{pics/6times8_white_rectangl}
				\put(00,20){\begin{sideways}$I=\{1\}$\end{sideways}}
				\put(20,5){
					\begin{overpic}[width=0.05\linewidth,angle=90]{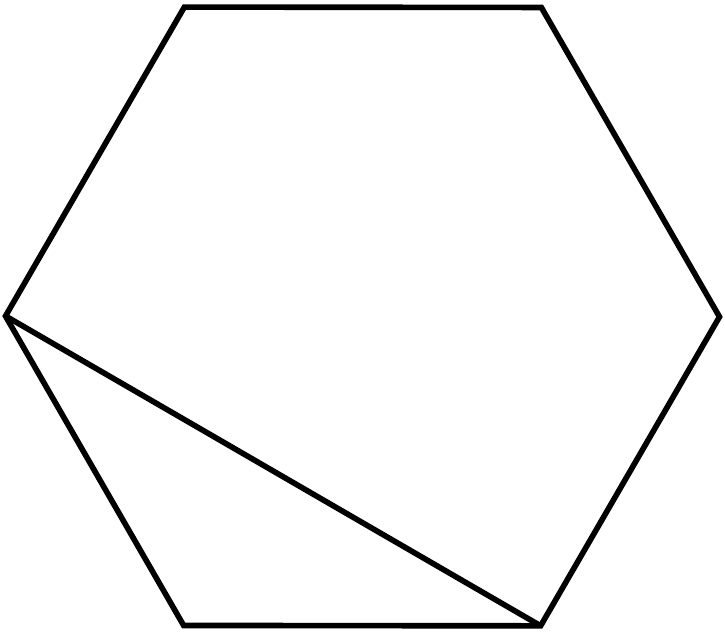}
						\put(6,-5){\scriptsize \begin{sideways}$0$\end{sideways}}
						\put(19,3){\scriptsize \begin{sideways}$1$\end{sideways}}
						\put(-6,2){\scriptsize \begin{sideways}$2$\end{sideways}}
						\put(19,14){\scriptsize \begin{sideways}$3$\end{sideways}}
						\put(6,21){\scriptsize \begin{sideways}$4$\end{sideways}}
						\put(-6,13){\scriptsize \begin{sideways}$5$\end{sideways}}
						\put(-10,39){\footnotesize \begin{sideways}$\begin{matrix} a=0 \\ b=3 \\ \gamma=1 \\ \Gamma=1 \end{matrix}$\end{sideways}}
					\end{overpic}
					}
				\put(60,5){
					\begin{overpic}[width=0.05\linewidth,angle=90]{pics/D_1_hexagon}
						\put(6,-5){\scriptsize \begin{sideways}$0$\end{sideways}}
						\put(19,2){\scriptsize \begin{sideways}$1$\end{sideways}}
						\put(-6,3){\scriptsize \begin{sideways}$2$\end{sideways}}
						\put(19,14){\scriptsize \begin{sideways}$3$\end{sideways}}
						\put(6,21){\scriptsize \begin{sideways}$4$\end{sideways}}
						\put(-6,13){\scriptsize \begin{sideways}$5$\end{sideways}}
						\put(-10,30){\footnotesize \begin{sideways}$\begin{matrix} \delta_1=\{0,3\} \\ 
																 \\ 
																 \\ 
																  
												  \end{matrix}$\end{sideways}}
					\end{overpic}
					}
				\put(95,20){\begin{sideways}$y_{\{1\}}=1$\end{sideways}}
				\end{overpic}
				}
		\put(260,180){
			\begin{overpic}[width=\linewidth]{pics/6times8_white_rectangl}
			\put(00,20){\begin{sideways}$I=\{2\}$\end{sideways}}
			\put(20,5){
				\begin{overpic}[width=0.05\linewidth,angle=90]{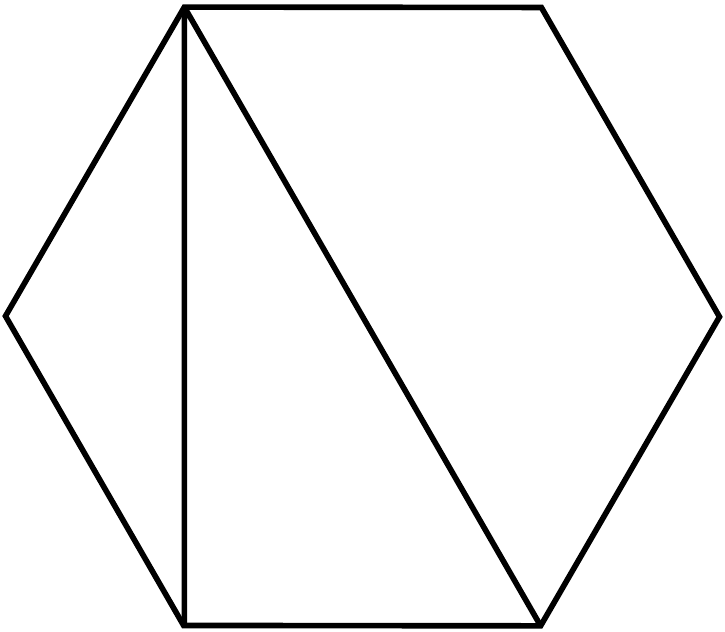}
					\put(6,-5){\scriptsize \begin{sideways}$0$\end{sideways}}
					\put(19,3){\scriptsize \begin{sideways}$1$\end{sideways}}
					\put(-6,2){\scriptsize \begin{sideways}$2$\end{sideways}}
					\put(19,14){\scriptsize \begin{sideways}$3$\end{sideways}}
					\put(6,21){\scriptsize \begin{sideways}$4$\end{sideways}}
					\put(-6,13){\scriptsize \begin{sideways}$5$\end{sideways}}
					\put(-10,39){\footnotesize \begin{sideways}$\begin{matrix} a=1 \\ b=3 \\ \gamma=2 \\ \Gamma=2 \end{matrix}$\end{sideways}}
				\end{overpic}
				}
			\put(60,5){
				\begin{overpic}[width=0.05\linewidth,angle=90]{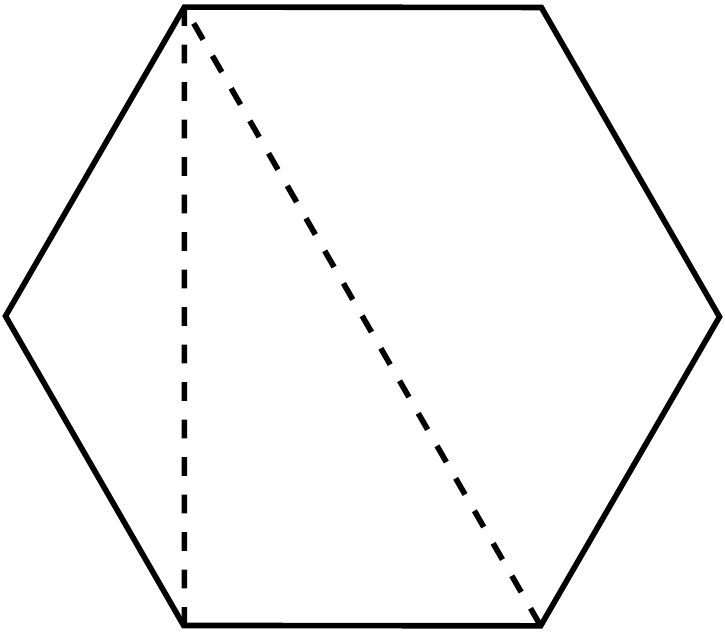}
					\put(6,-5){\scriptsize \begin{sideways}$0$\end{sideways}}
					\put(19,2){\scriptsize \begin{sideways}$1$\end{sideways}}
					\put(-6,3){\scriptsize \begin{sideways}$2$\end{sideways}}
					\put(19,14){\scriptsize \begin{sideways}$3$\end{sideways}}
					\put(6,21){\scriptsize \begin{sideways}$4$\end{sideways}}
					\put(-6,13){\scriptsize \begin{sideways}$5$\end{sideways}}
					\put(-10,30){\footnotesize \begin{sideways}$\begin{matrix} \\ 
															 \delta_2=\{1,2\} \\ 
															 \delta_3=\{2,3\} \\ 
												  
											  \end{matrix}$\end{sideways}}
				\end{overpic}
				}
			\put(95,20){\begin{sideways}$y_{\{2\}}=-1$\end{sideways}}
			\end{overpic}
			}
		\put(260,270){
			\begin{overpic}[width=\linewidth]{pics/6times8_white_rectangl}
			\put(00,20){\begin{sideways}$I=\{3\}$\end{sideways}}
			\put(20,5){
				\begin{overpic}[width=0.05\linewidth,angle=90]{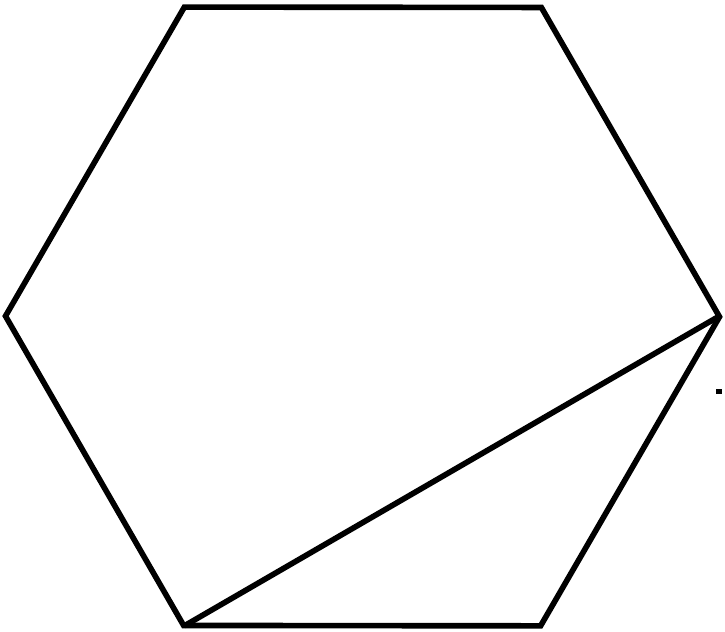}
					\put(6,-5){\scriptsize \begin{sideways}$0$\end{sideways}}
					\put(19,3){\scriptsize \begin{sideways}$1$\end{sideways}}
					\put(-6,2){\scriptsize \begin{sideways}$2$\end{sideways}}
					\put(19,14){\scriptsize \begin{sideways}$3$\end{sideways}}
					\put(6,21){\scriptsize \begin{sideways}$4$\end{sideways}}
					\put(-6,13){\scriptsize \begin{sideways}$5$\end{sideways}}
					\put(-10,39){\footnotesize \begin{sideways}$\begin{matrix} a=1 \\ b=4 \\ \gamma=3 \\ \Gamma=3 \end{matrix}$\end{sideways}}
				\end{overpic}
				}
			\put(60,5){
				\begin{overpic}[width=0.05\linewidth,angle=90]{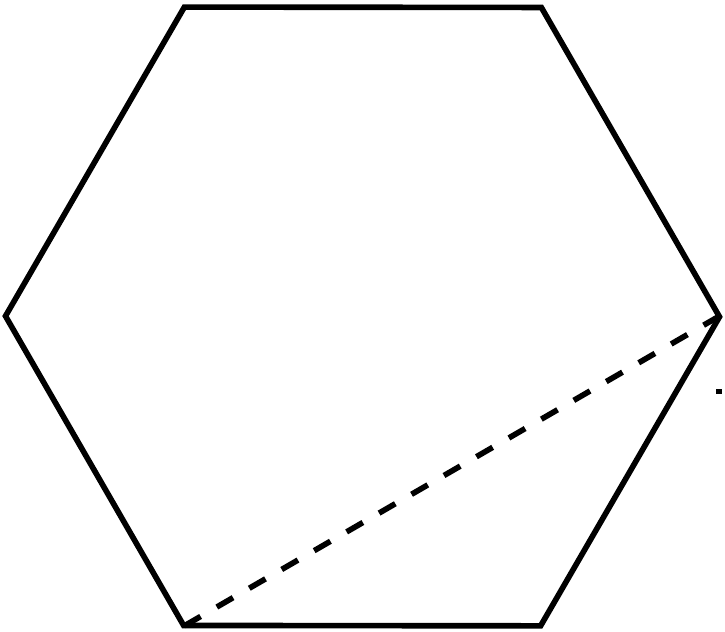}
					\put(6,-5){\scriptsize \begin{sideways}$0$\end{sideways}}
					\put(19,2){\scriptsize \begin{sideways}$1$\end{sideways}}
					\put(-6,3){\scriptsize \begin{sideways}$2$\end{sideways}}
					\put(19,14){\scriptsize \begin{sideways}$3$\end{sideways}}
					\put(6,21){\scriptsize \begin{sideways}$4$\end{sideways}}
					\put(-6,13){\scriptsize \begin{sideways}$5$\end{sideways}}
					\put(-10,30){\footnotesize \begin{sideways}$\begin{matrix} \delta_1=\{1,4\} \\ 
															  \\ 
															  \\ 
												  
											  \end{matrix}$\end{sideways}}
				\end{overpic}
				}
			\put(95,20){\begin{sideways}$y_{\{3\}}=1$\end{sideways}}
			\end{overpic}
			}
		\put(260,360){
			\begin{overpic}[width=\linewidth]{pics/6times8_white_rectangl}
			\put(00,20){\begin{sideways}$I=\{4\}$\end{sideways}}
			\put(20,5){
				\begin{overpic}[width=0.05\linewidth,angle=90]{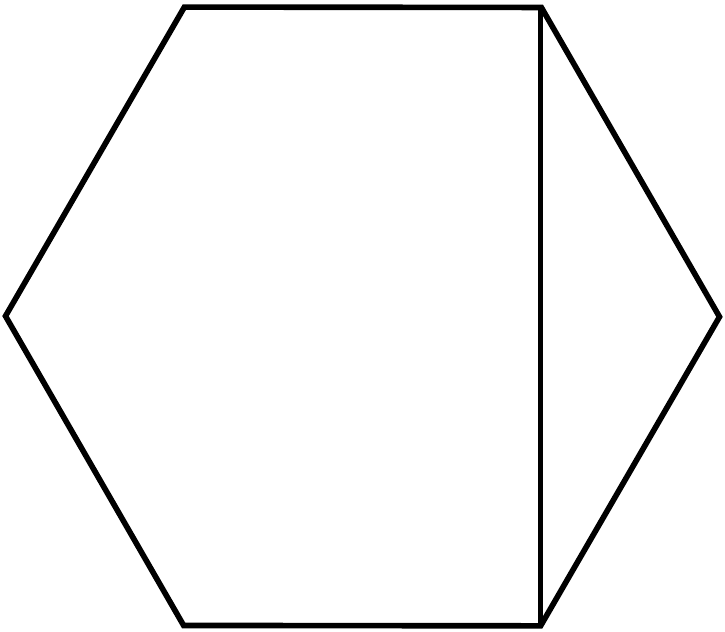}
					\put(6,-5){\scriptsize \begin{sideways}$0$\end{sideways}}
					\put(19,3){\scriptsize \begin{sideways}$1$\end{sideways}}
					\put(-6,2){\scriptsize \begin{sideways}$2$\end{sideways}}
					\put(19,14){\scriptsize \begin{sideways}$3$\end{sideways}}
					\put(6,21){\scriptsize \begin{sideways}$4$\end{sideways}}
					\put(-6,13){\scriptsize \begin{sideways}$5$\end{sideways}}
					\put(-10,39){\footnotesize \begin{sideways}$\begin{matrix} a=3 \\ b=5 \\ \gamma=4 \\ \Gamma=4 \end{matrix}$\end{sideways}}
				\end{overpic}
				}
			\put(60,5){
				\begin{overpic}[width=0.05\linewidth,angle=90]{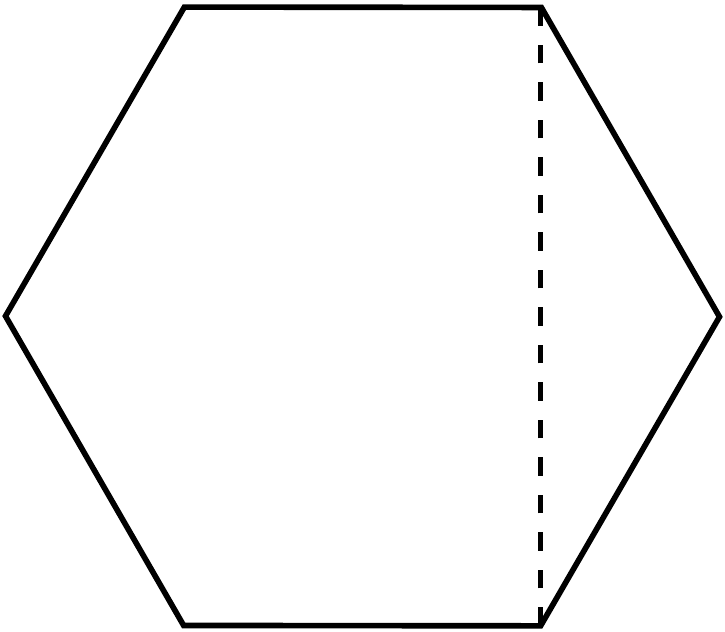}
					\put(6,-5){\scriptsize \begin{sideways}$0$\end{sideways}}
					\put(19,2){\scriptsize \begin{sideways}$1$\end{sideways}}
					\put(-6,3){\scriptsize \begin{sideways}$2$\end{sideways}}
					\put(19,14){\scriptsize \begin{sideways}$3$\end{sideways}}
					\put(6,21){\scriptsize \begin{sideways}$4$\end{sideways}}
					\put(-6,13){\scriptsize \begin{sideways}$5$\end{sideways}}
					\put(-10,30){\footnotesize \begin{sideways}$\begin{matrix} \delta_1=\{3,5\} \\ 
															  \\ 
															  \\ 
												  
											  \end{matrix}$\end{sideways}}
				\end{overpic}
				}
			\put(95,20){\begin{sideways}$y_{\{4\}}=1$\end{sideways}}
			\end{overpic}
			}
	     \end{overpic}
         \end{center}
         \caption[]{Details for the computation of the Minkowski coefficients~$y_I$ of $\As_{3}^c$ in case $\Do_c=\{1,3,4\}$ 
					and $\Up_c=\{2\}$. The first line states~$I$, the second line's first column pictures the non-crossing proper
					diagonals associated the up and down interval decomposition of~$I$, while the second column gives the values 
					for $a$, $b$, $\gamma$, and $\Gamma$ if there is only one nested component. The third line's first column 
					illustrates the proper diagonals of $\mathscr D_I\subseteq \{\delta_1, \delta_2, \delta_3, \delta_4\}$, the second 
					column specifies their end-points. We give the value for~$y_I$, where $y_{\{1,2,3,4\}}=-1$ since $|\Up_c|=1$ is omitted.}
         \label{fig:example_compute_y_I}
      \end{minipage}
      \end{center}
\end{figure}
\noindent
We prove Theorem~\ref{thm:first_y_I_thm} in Section~\ref{sec:proof_main_thm}. An example illustrating the theorem 
for the left associahedron~$\As^c_3$ of Figure~\ref{fig:a3_associahedra} ($\Do_c=\{1,3,4\}$ and $\Up_c=\{2\}$) is 
given in Figure~\ref{fig:example_compute_y_I} where we also explicitly compute the $y_I$-values for this realisation 
with $z^c_I=\frac{|I|(|I|+1)}{2}$ for the facet-defining inequalities. 

\medskip
For the rest of this section, we specialise to 
realisations with this specific choice of $z_I$-values. We obtain a nice combinatorial interpretation the coefficients~$y_I$ 
in Theorem~\ref{thm:second_y_I_thm} and characterise the vanishing~$y_I$-values in Corollary~\ref{cor:characterise_y_I=0}.

\medskip
If $I$ has a nested up and down interval decomposition, the \emph{signed lengths}~$K_\gamma$ and~$K_\Gamma$ of~$I$ are 
integers defined as follows. $|K_\Gamma|$ is the number of edges of the path in~$\partial Q$ connecting~$b$ and~$\Gamma$ 
that does not use the vertex labeled~$a$ and $K_\Gamma$ is negative if and only if $\Gamma \in \Do_c$. Similarly, 
$|K_\gamma|$ is the length of path in $\partial Q$ connecting~$a$ and~$\gamma$ not using label~$b$ and $K_\gamma$ is 
negative if and only if $\gamma \in \Do_c$. Equivalently, we have that~$K_\gamma$ (respectively~$K_\Gamma$) is a 
positive integer if and only if~$\gamma\in\Up_c$ (respectively~$\Gamma\in\Up_c$) and that~$K_\gamma=-1$ 
(respectively~$K_\Gamma=-1$) if and only if~$\gamma\in\Do_c$ (respectively~$\Gamma\in\Do_c$). We can now express the 
coefficients~$y_I$ of~$\As_{n-1}^c$ in terms of~$K_\gamma$ and~$K_\Gamma$. The following theorem is an easy 
consequence of Theorem~\ref{thm:first_y_I_thm}.
\begin{thm}$ $\\
	Let~$K_\Gamma$ and~$K_\gamma$ be the signed lengths of~$I$ as defined above if $I\subseteq [n]$ has a nested up and down interval 
	decomposition of type~$(1,k)$. Then	the Minkowski coefficient~$y_I$ of~$\As_{n-1}^c$ is 
	\[
	y_I = \begin{cases}
			(-1)^{|I\setminus (a,b)_{\Do}|}K_\gamma K_\Gamma			& \text{if $I\neq \{u_s\}\subseteq \Up_c$ and $v=1$,}\\
			(n+1) - K_\gamma K_\Gamma								 	& \text{if $I= \{u_s\}\subseteq \Up_c$,}\\
			0															& \text{if $v\geq 2$.}
		  \end{cases}
	\]\label{thm:second_y_I_thm}
\end{thm}
\begin{proof}
	By Theorem~\ref{thm:first_y_I_thm}, the claim is trivial if $I$ has up and down interval decomposition of type $v>1$. We 
	therefore assume $v=1$, set~$K:=|R_{\delta_1}|$, and observe $K_\Gamma:=|R_{\delta_2}|-|R_{\delta_1}|$ and 
	$K_\gamma:=|R_{\delta_3}|-|R_{\delta_1}|$. Thus\\[2mm]
	\centerline{$|R_{\delta_4}|=\begin{cases}
						K+K_\gamma+K_\Gamma     &\text{if $I \neq \{u_s\}$,}\\
						K+K_\gamma+K_\Gamma-1=n &\text{if $I = \{u_s\}$,}
					\end{cases}$}\\[2mm]
	as well as
	\begin{align*}
		z_{R_{\delta_1}}^c &= \frac{K(K+1)}{2},\\
		z_{R_{\delta_2}}^c &= \frac{(K+K_\Gamma)(K+K_\Gamma+1)}{2},\\
		z_{R_{\delta_3}}^c &= \frac{(K+K_\gamma)(K+K_\gamma+1)}{2}, and\\
		z_{R_{\delta_4}}^c &=
			\begin{cases}
				\frac{(K+K_\Gamma+K_\gamma)(K+K_\Gamma+K_\gamma+1)}{2}			& \text{if $I\neq\{ u_s\}$,}\\
				\frac{(K+K_\Gamma+K_\gamma)(K+K_\Gamma+K_\gamma+1)}{2}-(n+1)	& \text{if $I=\{u_s\}$.}
			\end{cases}
	\end{align*}
	A direct computation shows
	\[
		z_{R_{\delta_1}}^c - z_{R_{\delta_2}}^c - z_{R_{\delta_3}}^c + \frac{(K+K_\Gamma+K_\gamma)(K+K_\Gamma+K_\gamma+1)}{2}
		= K_\Gamma K_\gamma. 
	\]
 	The claim is now an immediate consequence of Theorem~\ref{thm:first_y_I_thm}.
\end{proof}

\begin{cor}\label{cor:sign_full_set}
	For $n\geq 2$ and any choice $\Do_c\sqcup\Up_c$, we have $y_{[n]}=(-1)^{|\Up_c|}$
\end{cor}
\begin{proof}
	The claim follows directly either from Theorem~\ref{thm:first_y_I_thm} or from Theorem~\ref{thm:second_y_I_thm}. 
	To obtain the claim from Theorem~\ref{thm:first_y_I_thm}, observe that $[n]\setminus R_{\delta_1}=\Up_c$ and 
	$z_{R_{\delta_1}}^c-z_{R_{\delta_2}}^c-z_{R_{\delta_3}}^c+z_{R_{\delta_4}}^c=1$. To obtain the claim from 
	Theorem~\ref{thm:second_y_I_thm}, we remark that $[n] \setminus R_{\delta_1} = I \setminus (a,b)_{\Do}$ and
	$K_\gamma =K_\Gamma=-1$ since $a=0$, $b=n+1$, $\gamma=1$, and $\Gamma=n$.
\end{proof}
\begin{cor}\label{cor:characterise_y_I=0}
	Let~$n \geq 2$ and $\Do_c\sqcup\Up_c$ be a partition induced by some Coxeter element~$c$.
	Then $y_I=0$ if and only if $I$ has an up and down decomposition of type $(v_I,w_I)$ with $v_I>1$ or $n=3$ and $I=\Up_c=\{2\}$.
\end{cor}
\begin{proof}
	Since~$K_\gamma$ and~$K_\Gamma$ are non-zero, Theorem~\ref{thm:second_y_I_thm} implies that $y_I\neq 0$ if $I\neq\{u_s\}\subseteq \Up_c$. 
	So we assume that $I=\{u_s\}\subseteq \Up_c$. It now suffices to prove that $y_I= 0$ if and only if $n=3$. 
	
	If $n=2$ then $I=\{u_s\}\subseteq \Up_c$ is impossible, so we have $n\geq 3$. From $R_{\delta_2}\cup R_{\delta_3}=[n]$ 
	and $R_{\delta_2}\cap R_{\delta_3}=\{ u_s\}$ we conclude $K_\gamma +K_\Gamma = n+1$. On the other hand, Theorem~\ref{thm:second_y_I_thm}
	implies that $y_I=0$ is equivalent to $K_\Gamma K_\gamma=n+1$. By substitution we have\\[2mm]
	\centerline{$K_{\Gamma}^2 -(n+1)K_{\Gamma}+(n+1)=0$} 
	and solving for $K_\Gamma$ gives \\[2mm]
	\centerline{
		$K_{\Gamma,1/2}=-\frac{-(n+1)}{2}\pm\sqrt{\frac{(n+1)^2}{4}-(n+1)} = \frac{(n+1)\pm\sqrt{n^2-2n-3}}{2}$.
	}\\[2mm]
	Since~$K_\Gamma$ is a positive integer, we conclude that $\sqrt{n^2-2n-3}$ is a positive integer. In particular,
	$n^2-2n-3=(n+1)(n-3)$ must be a square. For $n=3$, we conclude~$K_\Gamma=2$, that is $I=\Up_c=\{2\}$. For $n>3$ 
	we derive the contradiction $(n+1)=r^2(n-3)$ or $(n-3)=r^2(n-1)$ for some positive integer~$r$. 
\end{proof}

\medskip
We now illustrate Theorem~\ref{thm:second_y_I_thm} by recomputing the $y_I$-values for~$\As^{c_1}_2$ and~$\As^{c_2}_2$
mentioned in the introduction. For $n=3$, there are two possible partitions of~$\{1,2,3\}$ that correspond \
to the two Coxeter elements of~$\Sigma_3$: either $\Do_{c_1}=\{1,2,3\}$ and $\Up_{c_1}=\varnothing$ or $\Do_{c_2}=\{1,3\}$ 
and $\Up_{c_2}=\{2\}$. 

\begin{expl} \label{expl:As^c1_2}
	Consider $\Do_{c_1}=\{1,2,3\}$ and $\Up_{c_1}=\varnothing$ which yields Loday's realisation. 
\begin{compactenum}[(1)] 
	\item We have $y_I=1$ for~$I=\{ i\}$ and $1\leq i\leq 3$.\\
		  The up and down interval decomposition of $\{ i\}$ is $(i-1,i+1)_{\Do}$ and $\gamma=\Gamma=i$.
		  It follows that $K_\gamma=K_\Gamma=-1$ and $I\setminus (a,b)_{\Do}=\varnothing$. Thus $y_I=1$.
	\item We have $y_I=1$ for~$I=\{ i,i+1\}$ and $1\leq i\leq 2$.\\
		  Then $I=(i-1,i+2)_{\Do}$, $\gamma=i$, and $\Gamma=i+1$. It follows that
		  $K_\gamma=K_\Gamma=-1$ and $I\setminus (a,b)_{\Do}=\varnothing$. Thus $y_I=1$.
	\item We have $y_I=0$ for~$I=\{1,3\}$.\\
		  Then $I=(0,2)_{\Do}\sqcup(2,4)_{\Do}$, so $I$ is of type~$(2,0)$ and~$y_I=0$ by
		  Corollary~\ref{cor:characterise_y_I=0}.
	\item We have $y_I=1$ for~$I=\{1,2,3\}$.\\
		  Then $I=(0,4)_{\Do}$, $\gamma=1$ and $\Gamma=3$ implies $K_\gamma=K_\Gamma=-1$ and 
		  $I\setminus (a,b)_{\Do}=\varnothing$. Thus $y_I=1$. Of course, we could also use 
		  Corollary~\ref{cor:sign_full_set} instead.
\end{compactenum}
Altogether we have~$y_I\in\{0,1\}$ and~$\As^{c_1}_2$ is a Minkowski sum of faces of the standard simplex:
\[
	\As_2^{c_1} = 1\cdot \Delta_{\{1\}} + 1\cdot \Delta_{\{2\}} + 1\cdot \Delta_{\{3\}} 
				+ 1\cdot \Delta_{\{1,2\}} + 0\cdot \Delta_{\{1,3\}} + 1\cdot \Delta_{\{2,3\}} 
				+ 1\cdot \Delta_{\{1,2,3\}},
\]
recall Figure~\ref{fig:minkowski_decomp_As_c_1} for a visualisation of this equation of polytopes. 
\end{expl}

\begin{expl} \label{expl:As^c2_2}
	Consider $\Do_{c_2}=\{1,3\}$ and $\Up_{c_2}=\{2\}$. The associahedron~$\As_{2}^{c_2}$ is isometric 
	to~$\As_{2}^{c_1}$,~\cite{BergeronHohlwegLangeThomas}, but it is not the Minkowski sum of faces of 
	a standard simplex as we show now.
\begin{compactenum}[(1)] 
	\item We have $y_I=1$ for~$I=\{1\}$ and~$I=\{3\}$.\\
		  The up and down interval decomposition is $(0,3)_{\Do}$ and $(1,4)_{\Do}$ respectively. Therefore we 
		  have $\gamma=\Gamma=1$ and $\gamma=\Gamma=3$ respectively. It follows $K_\gamma=K_\Gamma=-1$ 
		  and~$I\setminus (a,b)_{\Do}=\varnothing$.
	\item We have $y_I=0$ for~$I=\{2\}$.\\
		  The up and down interval decomposition is $(1,3)_{\Do} \sqcup [2,2]_{\Up}$, so $I$ is of type~$(1,1)$. 
		  We have $\gamma=\Gamma=2$ which implies $K_\gamma=K_\Gamma=2$. Since 
		  $n=3$, we conclude $y_I=(3+1)-2\cdot 2=0$. Of course, we could have used Corollary~\ref{cor:characterise_y_I=0}.
	\item We have $y_I=2$ for~$I=\{i,i+1\}$ and $1\leq i\leq 2$.\\
		  Then $I=(i-1,i+2)_{\Do}$, $\gamma=i$, and $\Gamma=i+1$, that is, $K_\gamma=-1$, $K_\Gamma=2$. Moreover, 
		  $I\setminus (a,b)_{\Do}=\{2\}$ if $I=\{1,2\}$ and $K_\gamma=2$, $K_\Gamma=-1$, and 
		  $I\setminus (a,b)_{\Do}=\{2\}$ if $I=\{2,3\}$.
	\item We have $y_I=1$ for~$I=\{1,3\}$.\\
		  Then $I=(0,4)_{\Do}$, $\gamma=1$, and $\Gamma=3$. It follows that $K_\gamma=K_\Gamma=-1$ and 
		  $I\setminus (a,b)_{\Do}=\varnothing$. 
	\item We have $y_I=-1$ for~$I=\{ 1,2,3\}$.\\
		  Then $I=(0,4)_{\Do} \sqcup [2,2]_{\Up}$ with $\gamma=1$ and $\Gamma=3$. It follows that $K_\gamma=K_\Gamma=-1$ and
		  $I\setminus (a,b)_{\Do}=\{ 2\}$. Again, we could have used Corollary~\ref{cor:sign_full_set} instead.
\end{compactenum}
Thus, we obtain the following Minkowski decomposition into dilated faces of the standard simplex:
\[
	\As_2^{c_2} = 1\cdot \Delta_{\{1\}} + 0\cdot \Delta_{\{2\}} + 1\cdot \Delta_{\{3\}} 
				+ 2\cdot \Delta_{\{1,2\}} + 1\cdot \Delta_{\{1,3\}} + 2\cdot \Delta_{\{2,3\}} 
				+ (-1)\cdot \Delta_{\{1,2,3\}},
\]
recall that an illustration of this decomposition is given in Figure~\ref{fig:minkowski_decomp_As_c_2}.
\end{expl}

\section{A remark on cyclohedra} \label{sec_cyclohedron}

We now show that Proposition~\ref{prop:ardila} does not hold if we consider a polytope obtained by 
`moving some inequalities of the permutahedron past vertices'. The example is a cyclohedron which 
is an associahedron associated to a Coxeter group of type~$B$. A Minkowski decomposition of `generalised 
permutahedra of type~$B$' (similar to Proposition~\ref{prop:ardila} for generalised permutahedra) 
is not known. 

The canonical embedding of the hyperoctahedral group~$W_n$ in the symmetric group~$S_{2n}$ induces 
realisations~$\Cy_n^c$ of cyclohedra (also known as Bott-Taubes polytopes or type~$B$ generalised 
associahedra,~\cite{BottTaubes,ChapotonFominZelevinsky,Simion}) using realisations~$\As^c_{2n-1}$ 
for certain {\em symmetric} choices~$c$. To obtain realisations of cyclohedra, we follow~\cite{HohlwegLange} 
and intersect~$\As_{2n-1}^c$ with `type~$B$ hyperplanes' $x_i+x_{2n+1-i}=2n-1$ for $1\leq i < n$.
A 2-dimensional cyclohedron~$\Cy_{2}^c$ obtained from~$\As_{3}^c$ (with up set $\Up_c=\{2\}$) by intersection 
with $x_1+x_4=5$ is shown in Figure~\ref{fig:b2_cyclohedra} (the hyperplane $x_2+x_3=5$ is implicitly used since
$\As_{3}^c$ is contained in $x_1+x_2+x_3+x_4=10$). A similar construction does not yield a cyclohedron 
if one starts with the other associahedron of Figure~\ref{fig:a3_associahedra} where $\Up_c=\{2,3\}$.
The tight right-hand sides of this realisation of the cyclohedron are obviously the tight right-hand 
sides of~$\As_{3}^c$ except $z^c_{\{1,4\}}=z^c_{\{2,3\}}=5$. The inequalities $x_1+x_4\geq 2$ and $x_2+x_3\geq 2$ 
are redundant for $\As_{3}^c$ and altering the level sets for these inequalities from~$2$ (for~$\As_{3}^c$) 
to~$5$ (for~$\Cy_2^c$) means that we move past the four vertices~$A$,~$B$,~$C$, and~$D$, so the realisation 
of the cyclohedron is not in the deformation cone of the classical permutahedron. We now show by example that
Proposition~\ref{prop:ardila} does not hold in this situation. To this respect, we list the function $z_I$ of tight
right hand-sides for all inequalities of the permutahedron (that is, facet-defining or not for the cyclohedron)
and its M\"obius inverse $y_I$, both defined on the boolean lattice: \\[2mm]

	\centerline{              $\begin{matrix} z_{\{1,2,3,4\}}=10\\ y_{\{1,2,3,4\}}=5 \end{matrix}$}$ $\\[2mm]
	\centerline{              $\begin{matrix} z_{\{1,2,3\}}=6\\ y_{\{1,2,3\}}=-4 \end{matrix}$
		        \rule{1cm}{0cm} $\begin{matrix} z_{\{1,2,4\}}=4\\ y_{\{1,2,4\}}=-3 \end{matrix}$
			    \rule{1cm}{0cm} $\begin{matrix} z_{\{1,3,4\}}=6\\ y_{\{1,3,4\}}=-2 \end{matrix}$
				\rule{1cm}{0cm} $\begin{matrix} z_{\{2,3,4\}}=6\\ y_{\{2,3,4\}}=-1 \end{matrix}$}$ $\\[4mm]
	\centerline{                $\begin{matrix} z_{\{1,2\}}=3\\ y_{\{1,2\}}=3 \end{matrix}$
		        \rule{1cm}{0cm} $\begin{matrix} z_{\{1,3\}}=3\\ y_{\{1,3\}}=1 \end{matrix}$
			    \rule{1cm}{0cm} $\begin{matrix} z_{\{1,4\}}=5\\ y_{\{1,4\}}=3 \end{matrix}$
				\rule{1cm}{0cm} $\begin{matrix} z_{\{2,3\}}=5\\ y_{\{2,3\}}=5 \end{matrix}$
				\rule{1cm}{0cm} $\begin{matrix} z_{\{2,4\}}=0\\ y_{\{2,4\}}=0 \end{matrix}$
				\rule{1cm}{0cm} $\begin{matrix} z_{\{3,4\}}=3\\ y_{\{3,4\}}=1 \end{matrix}$}$ $\\[4mm]
	\centerline{                $\begin{matrix} z_{\{1\}}=1  \\ y_{\{1\}}=1 \end{matrix}$
		        \rule{1cm}{0cm} $\begin{matrix} z_{\{2\}}=-1 \\ y_{\{2\}}=-1 \end{matrix}$
			    \rule{1cm}{0cm} $\begin{matrix} z_{\{3\}}=1  \\ y_{\{3\}}=1  \end{matrix}$
			    \rule{1cm}{0cm} $\begin{matrix} z_{\{4\}}=1  \\ y_{\{4\}}=1  \end{matrix}$.}$ $\\[4mm]
In other words, if Proposition~\ref{prop:ardila} were true for `generalised permutahedra not in the deformation cone
\begin{figure}
	\vspace{-0.5cm}
    \begin{center}
      \begin{minipage}{0.95\linewidth}
         \begin{center}
         	\begin{overpic}
            	[width=0.6\linewidth]{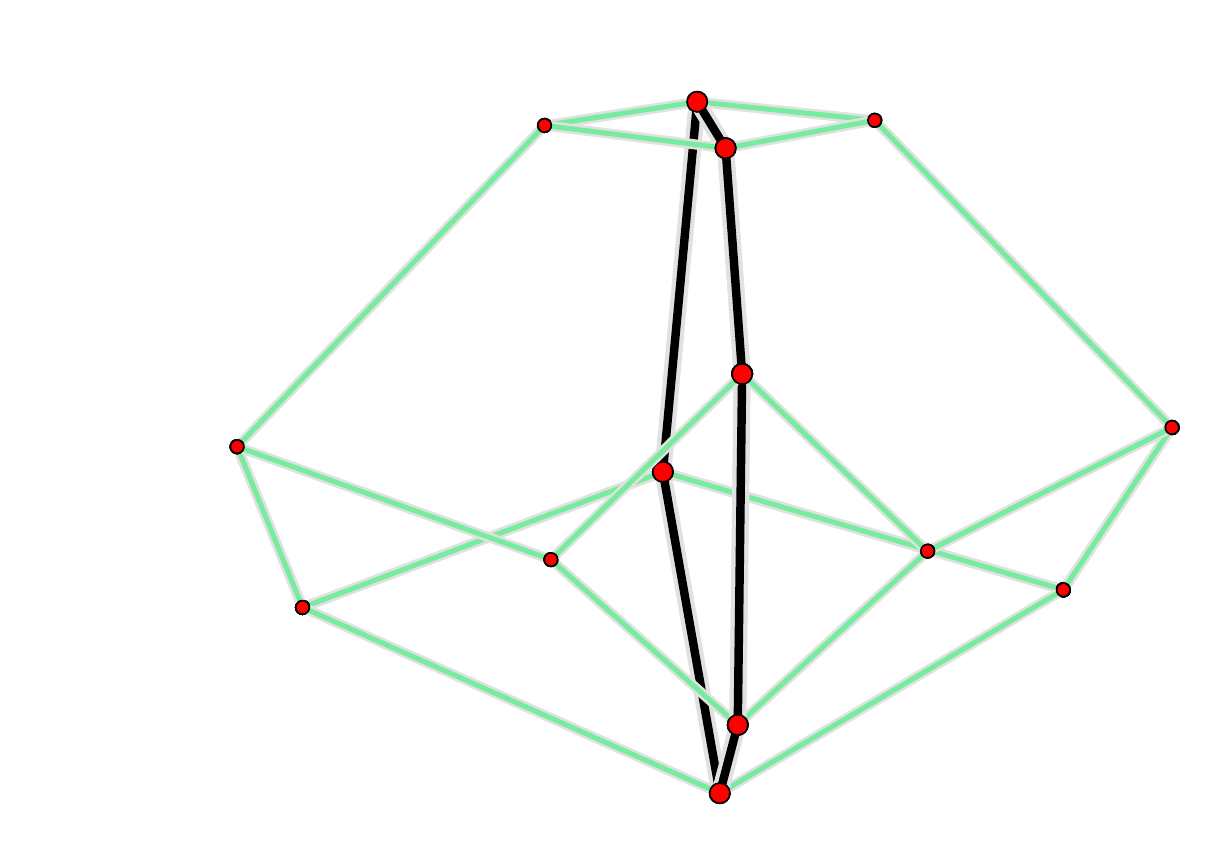}
            	\put(95,143){$A$}
            	\put(176,143){$B$}
            	\put(96,50){$C$}
            	\put(180,50){$D$}
         	\end{overpic}
        \end{center}
		\vspace{-0.5cm}
		\caption[]{A $2$-dimensional cyclohedron~$\Cy_2$ (black) obtained from~$\As^c_3$.}
    	\label{fig:b2_cyclohedra}
    \end{minipage}
   \end{center}
\end{figure}
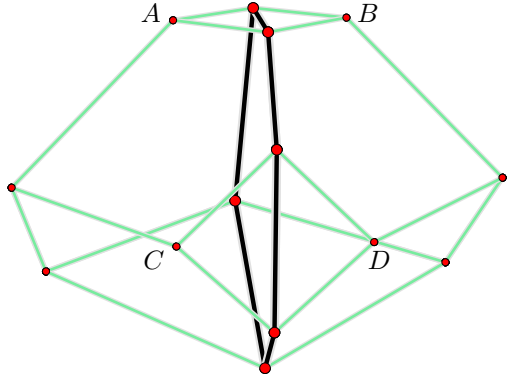
of the classical permutahedron', then the following equation of polytopes has to hold:
\begin{align*}
	\Cy_{2}^c +&\left(\Delta_2+4\Delta_{123}+3\Delta_{124}+2\Delta_{134}+\Delta_{234}\right)\\
	  		   &\qquad= \Delta_1+\Delta_3+\Delta_4+3\Delta_{12}+\Delta_{13}+3\Delta_{14}+5\Delta_{23}+\Delta_{34}+5\Delta_{1234}.
\end{align*}
One way to see that this equation does not hold is to compute the number of vertices of the polytope on the left-hand 
side ($27$ vertices) and on the right-hand side ($20$ vertices) using for example {\tt polymake},~\cite{polymake}.

\section{A proof of Theorem~\ref{thm:first_y_I_thm}} \label{sec:proof_main_thm}
This section is devoted to the proof of Theorem~\ref{thm:first_y_I_thm} under the assumption that 
Lemma~\ref{lem:possible_D_I} holds; Lemma~\ref{lem:possible_D_I} is proved in Section~\ref{sec_characterisation_of_possible_D_I}. The strategy to prove Theorem~\ref{thm:first_y_I_thm} 
is as follows.

First, we prove Proposition~\ref{prop:y_I_for_non-degenerate_D_I} which weakens Theorem~\ref{thm:first_y_I_thm} 
in two senses: we restrict to $I\subset[n]$ with a nested decomposition and we restrict to the situation 
where $\mathscr D_I=\{\delta_1,\delta_2,\delta_3,\delta_4\}$, that is, where all four diagonals $\delta_i$ are 
proper. That the statement of Proposition~\ref{prop:y_I_for_non-degenerate_D_I} is actually the statement of 
Theorem~\ref{thm:first_y_I_thm} weakened by these additional assumptions follows from Corollary~\ref{cor:the_case_v=1}.

Lemma~\ref{lem:possible_D_I} states precisely which subsets of~$\{\delta_1,\delta_2,\delta_3,\delta_4\}$ are 
sets~$\mathscr D_I$ for some~$I\subset [n]$ with a nested up and down interval decomposition. Lemma~\ref{lem:other_special_D_I}
then expresses the Minkowski coefficients $y_I$ using these sets~$\mathscr D_I$ if~$I\subset [n]$ has a nested 
up and down interval decomposition and $|\mathscr D_I|<4$. 
Lemma~\ref{lem:which_R_delta_i} and Lemma~\ref{lem_relate_signs} then
imply the claim of Theorem~\ref{thm:first_y_I_thm} when $I\subset [n]$ has a nested decomposition 
and not all $\delta_i$ are proper. Finally, Lemma~\ref{lem_y_I_for_multiple_components} covers the cases $I\subset [n]$ 
where $I$ does not have a nested decomposition and Lemma~\ref{lem:y_I_for_[n]} settles~$I=[n]$.

\medskip
\noindent
It will be convenient to rewrite Equation~(\ref{horrible_equation}) that was obtained at the beginning of Section~\ref{sec:main_results_and_exapmples} by combination of Proposition~\ref{prop:ardila} and Proposition~\ref{prop_tight_values_for_z_I}:
\begin{align*}
	y_I &= \sum_{J\subseteq I}(-1)^{|I\setminus J|}
			\sum_{i\in[v_J]}\left (  
								\sum_{j\in [w_i+1]} \tilde z_{R_{\delta^J_{i,j}}}^c - w_i \tilde z^c_{[n]}
							\right)\\
		&= \sum_{J\subseteq I}(-1)^{|I\setminus J|}
				\sum_{i\in[v_J]}\left ( \tilde z_{R_{\delta^J_{i,m^J_i}}}^c 
										+ \sum_{j\in [m^J_i-1]} 
													\bigl(\tilde z_{R_{\delta^J_{i,j}}}^c - \tilde z^c_{[n]}\bigr)
								\right)
\end{align*}
where $m^J_i$ is either $w^J_i$ or $w^J_i+1$ in order to simplify the involved sum.

\medskip
Suppose now that the proper diagonal $\delta$ occurs in the right-hand side of this rewritten formula for $y_I$, that
is, $\delta$ is one of the associated diagonals $\delta^J_{i,j}$ for some $J\subseteq I$. We now distinguish whether~$\delta$ occurs as a single summand $\tilde z_{R_{\delta^J_{i,m^J_i}}}^c$ or as a compound summand
$(\tilde z_{R_{\delta^J_{i,j}}}^c - \tilde z^c_{[n]})$.

\noindent
We now make the following definition.
\begin{defn}
	Let $I\subset [n]$ be non-empty. 
	\begin{compactenum}
		\item A proper diagonal $\delta$ (associated to~$J\subseteq I$) is of type~$\tilde z_{R_\delta}^c$ (in 
			the expression for~$y_I$), if there exist an index~$i\in [v_J]$ such that $\delta=\delta^J_{i,m^J_i}$. 
		\item A proper diagonal $\delta$ (associated to~$J\subseteq I$) is of type 
			$\bigl(\tilde z_{R_\delta}^c - \tilde z^c_{[n]}\bigr)$ (in the expression for~$y_I$), if there exist
			indices~$i\in [v_J]$ and~$j\in [m^J_i-1]$ such that $\delta=\delta^J_{i,j}$
	\end{compactenum}
	
\end{defn}

A geometric interpretation of these notions is the following.
The proper diagonal~$\delta$ (associated to~$J\subseteq I$) is of type~$\tilde z_{R_\delta}^c$ (in the expression for~$y_I$), if $\delta$ is the `rightmost' proper diagonal associated to a nested component of $J$.  
Similarly, the proper diagonal~$\delta$ (associated to~$J\subseteq I$) is of type 
$\bigl(\tilde z_{R_\delta}^c - \tilde z^c_{[n]}\bigr)$ (in the expression for~$y_I$), if $\delta$ is a proper diagonal 
associated to a nested component of $J$, but it is not the rightmost one. 

\begin{prop}\label{prop:y_I_for_non-degenerate_D_I}$ $\\
	Let $I$ be a non-empty proper subset of~$[n]$ with up and down interval decomposition of type $(1,w)$ 
	and $\mathscr D_I=\{\delta_1,\delta_2,\delta_3,\delta_4\}$.
	Then the Minkowski coefficient $y_I$ of an~$P(\{\tilde z^c_I\})$ with the normal fan of~$\As^c_{n-1}$ is given by
	\[
		y_I = \sum_{\delta\in\mathscr D_I}(-1)^{|I\setminus R_{\delta}|}\tilde z_{R_{\delta}}^c.
	\]
\end{prop}
\noindent
The proof is not difficult but long and convoluted, so we first outline the proof. The goal is to simplify the 
rewritten Equation~(\ref{horrible_equation}) for~$y_I$ stated above. To that respect, we first study the potential contribution 
of a proper diagonal~$\delta$ that occurs in the sum on the right-hand side. Given such a diagonal~$\delta$, we 
study which sets $S\subseteq I$ satisfy $\delta\in\mathcal D_S$ in order to collect all terms that involve~$z_{R_\delta}^c$. 
We will show that the corresponding sum vanishes often. This result is obtained by a case study that depends 
on the type of the up and down interval decomposition of~$R_\delta$. Since the up and down interval decomposition
of $R_\delta$ is of type $(1,0)$, $(1,1)$ or $(1,2)$ for any proper diagonal~$\delta$, we study these cases in detail. 
After the necessary information is deduced 
for every possible diagonal~$\delta$, we further simplify the formula for~$y_I$
by another case study that distinguishes whether $\gamma$ or $\Gamma$ is element of~$\Do_c$ or~$\Up_c$.

\begin{proof}
	By assumption, the set~$I\subset [n]$ has an up and down interval decomposition of type~$(1,w)$, that is, 
	$I=(a,b)_{\Do_c}\sqcup \bigsqcup_{j=1}^w[\alpha_j,\beta_j]_{\Up_c}$. Let~$\delta$ be some diagonal $\delta^J_{i,j}$ 
	that occurs on the right-hand side of the equation for~$y_I$. In other words, $\delta$ is a proper and non-degenerate 
	diagonal~$\delta^J_{i,j}$ associated to the up and down interval decomposition of type $(v^J,w^J)$ for some $J\subseteq I$. 
	By Example~\ref{expl:types_of_diagonals}, the up and down interval decomposition of $R_{\delta}$ is either of 
	type $(1,0)$, $(1,1)$ or $(1,2)$. A good understanding which sets $S\subseteq I$ (besides $J$)
	satisfy $\delta\in\mathcal D_S$ is essential for the simplification. The complete proof is basically a case 
	study of these three cases. 
	\begin{compactenum}[1.]
	\item $R_\delta$ has up and down decomposition of type~$(1,0)$, see Figure~\ref{fig:example_cases_1}.
	\label{case:one}\\
		Then $R_{\delta}=(\tilde a, \tilde b)_{\Do_c}\subseteq (a,b)_{\Do_c}$ and we may consider $J=R_\delta\subseteq I$
		as witness for the occurrence of $\delta$ in the right-hand side of~(\ref{horrible_equation}). Let
		$S\subseteq I$ be a set with $\delta\in \mathcal D_S$. Then	$J=(\tilde a,\tilde b)_{\Do_c}$ is necessarily 
		a nested component of type~$(1,0)$ of~$S$ and all other nested components are subsets of~$(a,\tilde a)\cap I$ 
		and~$(\tilde b,b)\cap I$. It follows that~$S\subseteq I$ satisfies $\delta\in\mathcal D_S$ if and only if
		\[
			R_{\delta}\subseteq S \subseteq R_{\delta}\cup \bigl((a,\tilde a)\cap I\bigr)\cup \bigl((\tilde b,b)\cap I\bigr).
		\]
		We now collect all terms for $\tilde z_{R_\delta}^c$ in the expression for $y_I$. Since $\delta$ is a 
		proper diagonal, we have $\tilde z_{R_\delta}^c\neq 0$ and the resulting alternating sum vanishes if 
		and only if there is more than one term of this type, that is, if and only if
		$\bigl((a,\tilde a)\cap I\bigr)\cup \bigl((\tilde b,b)\cap I\bigr)\neq \varnothing$. If 
		$\bigl((a,\tilde a)\cap I\bigr)\cup \bigl((\tilde b,b)\cap I\bigr)= \varnothing$, 
		we obtain $(-1)^{|I\setminus R_\delta|}\tilde z_{R_\delta}^c$ as contribution for~$y_I$. 
		
		For later use in this proof, we note that 
		$\bigl((a,\tilde a)\cap I\bigr)\cup \bigl((\tilde b,b)\cap I\bigr)= \varnothing$ guarantees 
		$\delta \in \mathscr D_I$. The diagonal~$\delta_1$ is always of type~$(1,0)$ if the up and down decomposition
		of~$R_\delta$ is of type~$(1,0)$. Similarly, we have $\delta_2\in\mathscr D_I$ is of type~$(1,0)$ if additionally 
		$\Gamma\in\Do_c$, $\delta_3\in\mathscr D_I$ is of type~$(1,0)$ if additionally $\gamma\in \Do_c$, 
		and $\delta_4\in\mathscr D_I$ is of type~$(1,0)$ if additionally $\gamma,\Gamma\in\Do_c$.

	\item $R_\delta$ has up and down decomposition of type~$(1,1)$. \\
		In contrast to Case~\ref{case:one}, $R_\delta\subseteq I$ is not true in general any more. We distinguish two cases,
		either $\delta=\{\tilde \beta, \tilde b\}$ with $\tilde \beta<\tilde b$, $\tilde \beta\in\Up_c$ and $\tilde b\in\Do_c$
		or $\delta=\{\tilde a, \tilde \alpha\}$ with $\tilde a<\tilde \alpha$, $\tilde a\in\Do_c$ and $\tilde \alpha\in\Up_c$.
		\begin{compactenum}[a.]
			\item $\delta=\{\tilde \beta, \tilde b\}$, see Figure~\ref{fig:example_cases_2a}\\
				\label{case:two_a}
				Observe first that $R_\delta=(0,\tilde b)_{\Do_c}\cup[u_1,\tilde\beta]_{\Up_c}$ 
				with $\tilde\beta<\tilde b\leq b$. 
				Since we assume that $\delta$ appears in the right-hand side of~(\ref{horrible_equation}), we have 
				$\tilde \beta\in I$ and may consider $J=R_\delta\cap I$. 
			
				If $S\subseteq I$ is a subset with $\delta\in \mathcal D_S$ then $\delta$ must be the `rightmost' diagonal 
				of one nested component for~$S$. This means that the diagonal~$\delta$ associated to~$S$ is never of type 
				$(\tilde z_{R_{\delta}}^c-\tilde z^c_{[n]})$ in the expression for~$y_I$. Similarly to Case~\ref{case:one}, 
				we conclude that the terms~$\tilde z_{R_\delta}^c$ cancel if and only if 
				\[
			  	  \bigl((a,\tilde \beta)\cap I\bigr)\cup\bigl((\tilde b, b)\cap I\bigr)\neq \varnothing
			  	  \qquad\text{or}\qquad 
			  	  \tilde z_{R_\delta}^c=0.
				\]
				Again, $\tilde z_{R_\delta}^c\neq 0$ since $\delta$ is a proper diagonal and the terms for 
				$\tilde z_{R_\delta}^c$ do not cancel if and only if there is only one subset $S\subseteq I$ 
				with $\delta\in \mathcal D_S$, that is, if $((a,\tilde \beta)\cap I)\cup((\tilde b, b)\cap I)= \varnothing$.
			
				For later use in ths proof, we mention the two possible scenarios 
				if $((a,\tilde \beta)\cap I)\cup((\tilde b, b)\cap I)= \varnothing$. Firstly, if $\gamma\in\Up_c$ and
				$\Gamma\in\Do_c$, then $\delta\in\{\delta_3,\delta_4\}$ and the contribution of $\delta_3$ and~$\delta_4$ 
				to~$y_I$ is
				\[
				  (-1)^{|I\setminus R_{\delta_3}|}\tilde z_{R_{\delta_3}}^c 
				  \qquad\text{and}\qquad
				  (-1)^{|I\setminus R_{\delta_4}|}\tilde z_{R_{\delta_4}}^c.
				\]
				Secondly, if $\gamma,\Gamma\in \Up_c$, then $\delta=\delta_3$ and the contribution to $y_I$ is 
				$(-1)^{|I\setminus R_{\delta_3}|}\tilde z_{R_{\delta_3}}^c$.

				\begin{figure}[b]
				      \begin{center}
						\rule{0mm}{5mm}
						
				      \begin{minipage}{0.95\linewidth}
				         \begin{center}
				         \begin{overpic}
				            [width=0.4\linewidth]{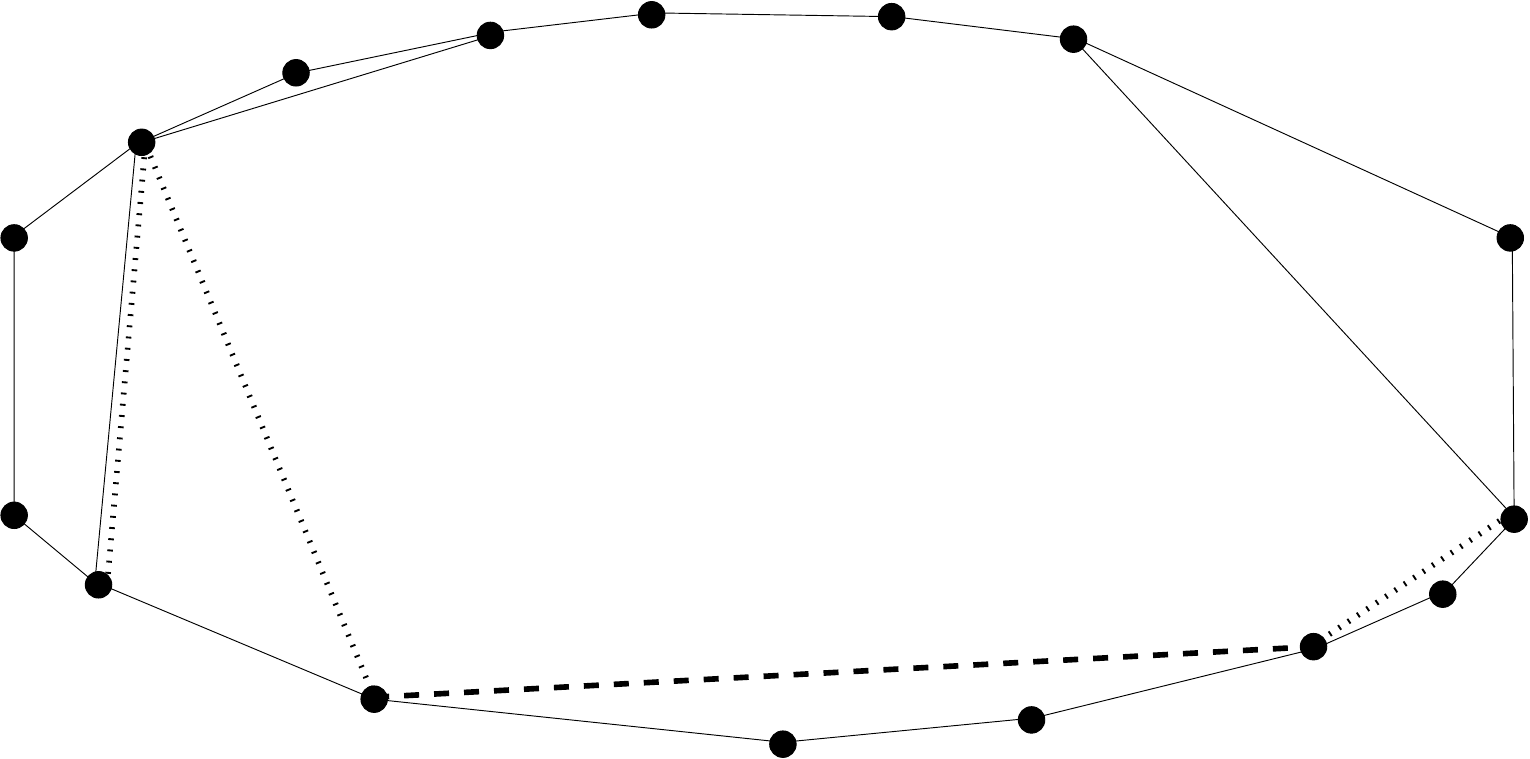}
				            \put(-8,53){$0$}
				            \put(-8,20){$1$}
				            \put(7,9){$2$}
				            \put(14,67){$3$}
				            \put(23,73){$4$}
				            \put(35,-3){$5$}
				            \put(45,78){$6$}
				            \put(62,80){$7$}
				            \put(80,-10){$8$}
				            \put(95,80){$9$}
				            \put(110,-5){$10$}
				            \put(115,77){$11$}
				            \put(138,2){$12$}
				            \put(152,10){$13$}
				            \put(163,20){$14$}
				            \put(163,53){$15$}
				            \put(90,15){$\delta$}

					     \end{overpic}
				         \end{center}
				         \caption[]{Let $I=\{3,5,6,7,8,9,10,11,12,13\}$ with $\gamma=3$ and $\Gamma=13$.
				 					Its up and down decomposition is $I=(2,14)_{\Do_c}\cup[3,3]_{\Up_c}\cup[6,11]_{\Up_c}$.\\
									\rule{2mm}{0cm} The diagonal $\delta=\{5,12\}$ appears in the right hand side for $y_I$ since 
									$(5,12)_{\Do_c}\subseteq (2,14)_{\Do_c}$ and the up and down interval decomposition 
									of $R_\delta$ has type $(1,0)$. Since $(2,5)\cap I = \{3\}$ and $(12,14)\cap I=\{13\}$, 
									$\delta$ is associated to $S\in\left\{\{8,10\}, \{3,8,10\},\{8,10,13\}, \{3,8,10,13\}\right\}$, 
									the diagonals associated to the up and down interval decompositions of~$S$ form a subset of
									the dashed diagonals. The contribution of~$\delta$ to~$y_I$ vanishes.\\
									\rule{2mm}{0cm} The only diagonals associated to some $J\subseteq I$ 
									with up and down interval decomposition of type~$(1,0)$ and non-vanishing contribution to~$y_I$ 
									are diagonals associated to only 
									one subset $J\subseteq I$, i.e. $\delta_1=\{2,14\}$ and $\delta_2=\{2,13\}$ in this example. }
				         \label{fig:example_cases_1}
				      \end{minipage}
						$ $
						
						\rule{0mm}{9mm}
				      \end{center}
				\end{figure}

				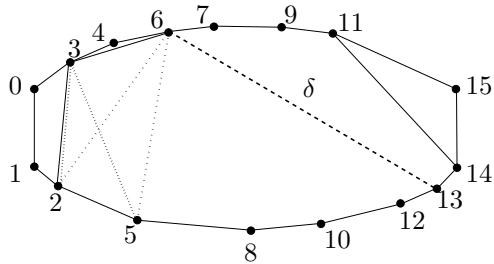
\begin{figure}[b]
				      \begin{center}
				      \begin{minipage}{0.95\linewidth}
				         \begin{center}
				         \begin{overpic}
				            [width=0.4\linewidth]{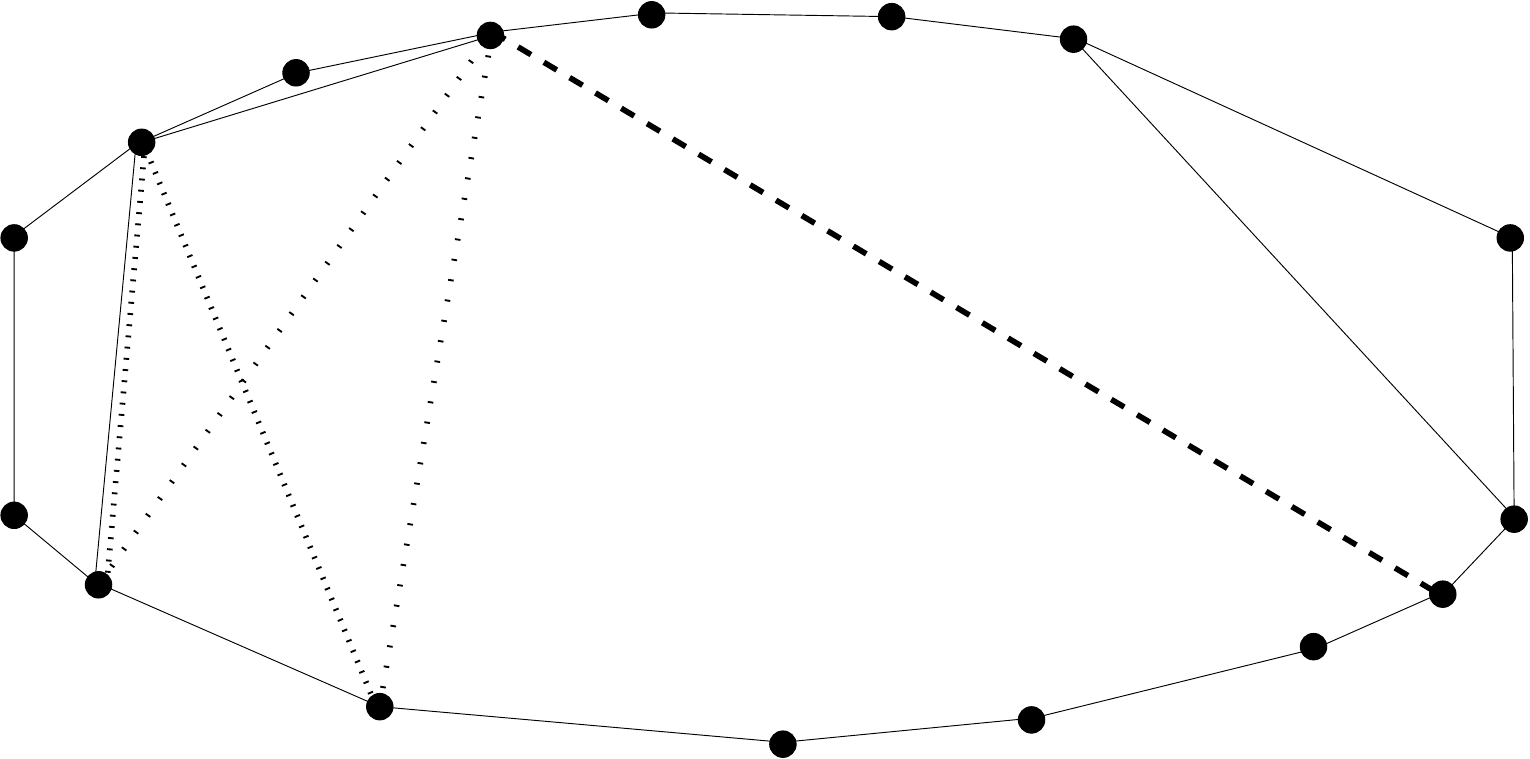}
				            \put(-8,53){$0$}
				            \put(-8,20){$1$}
				            \put(7,9){$2$}
				            \put(14,67){$3$}
				            \put(23,73){$4$}
				            \put(35,-3){$5$}
				            \put(45,78){$6$}
				            \put(62,80){$7$}
				            \put(80,-10){$8$}
				            \put(95,80){$9$}
				            \put(110,-5){$10$}
				            \put(115,77){$11$}
				            \put(138,2){$12$}
				            \put(152,10){$13$}
				            \put(163,20){$14$}
				            \put(163,53){$15$}
				            \put(102,52){$\delta$}
					     \end{overpic}
				         \end{center}
				         \caption[]{Consider $I$ as in Figure~\ref{fig:example_cases_1}.
				 					For $\delta=\{6,13\}$, the up and down interval decomposition of $R_\delta$ is 
									of the required sub-type of $(1,1)$. $\delta$ is associated to 
									$S\in\left\{\{3,5,6,8,10,12\}, \{5,6,8,10,12\}, \{3,6,8,10,12\}, \{6,8,10,12\}\right\}$
									since $(2,6)\cap I = \{3,5\}$ and $(13,14)\cap I = \varnothing$ and some of the diagonals 
									are associated to the interval decomposition of~$S$. The contribution of~$\delta$ to~$y_I$ 
									vanishes. Diagonals of the required sub-type of~$(1,1)$ and non-vanishing contribution 
									to~$y_I$ are the diagonals~$\delta$ that are associated to precisely one subset $J\subseteq I$, 
									that is, $\delta_3=\{3,14\}$ and $\delta_4=\{3,13\}$ in this figure. }
						         \label{fig:example_cases_2a}
				      \end{minipage}
				      \end{center}
				\end{figure}

	    \item $\delta=\{\tilde a, \tilde \alpha\}$\\
			\label{case:two_b}
			Observe first that $R_\delta=(\tilde a,n+1)_{\Do_c}\cup[\tilde\alpha,u_m]_{\Up_c}$ 
			with $a\leq \tilde a<\tilde \alpha$. Since we assume that $\delta$ appears in the right-hand side 	
			of~(\ref{horrible_equation}), we have $\tilde \alpha\in I$ and may consider $J=R_\delta\cap I$.
			
		\begin{figure}[b]
		      \begin{center}
			  \rule{0mm}{5mm}	
			
		      \begin{minipage}{0.95\linewidth}
		         \begin{center}
		         \begin{overpic}
		            [width=0.4\linewidth]{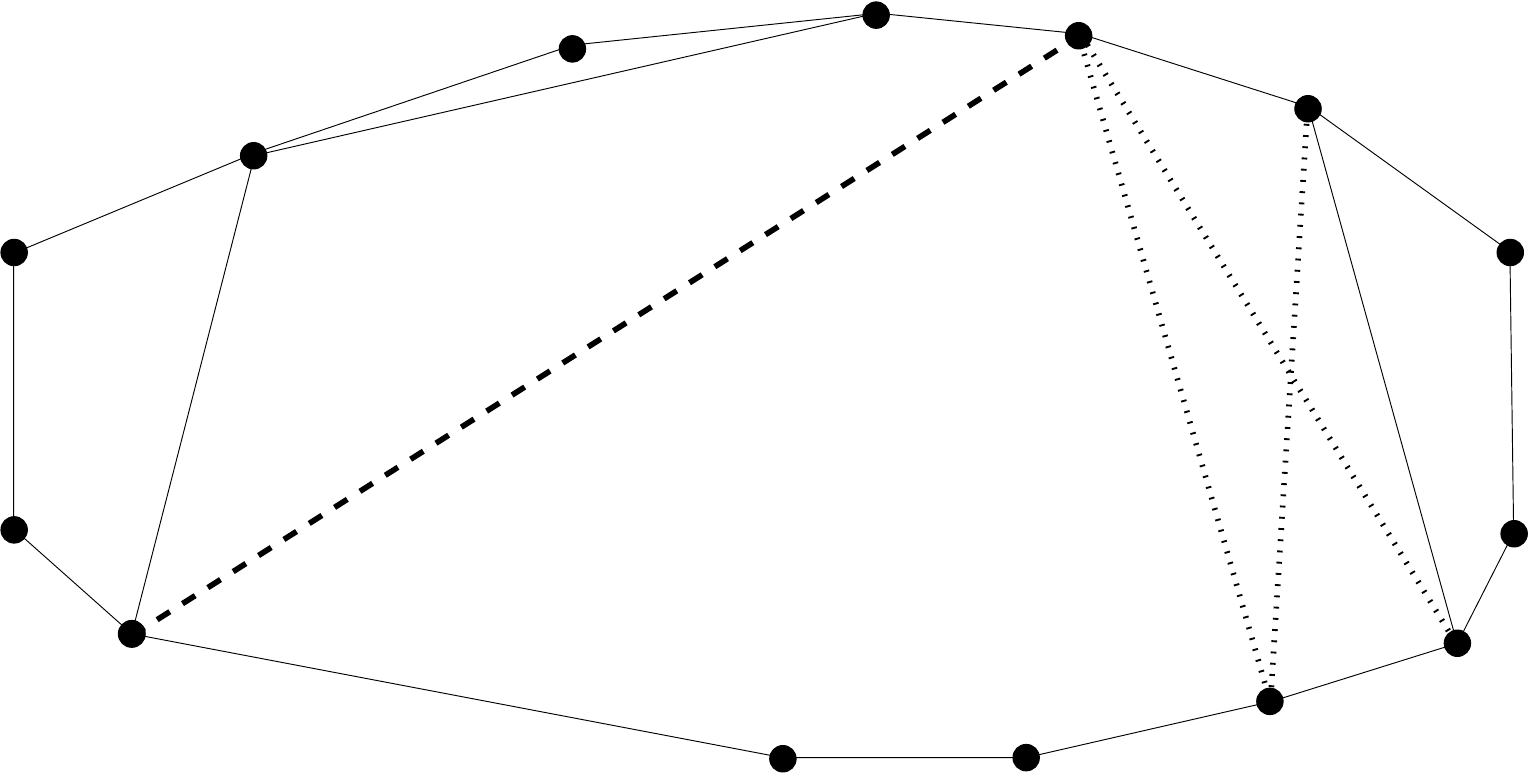}
		            \put(-8,53){$0$}
		            \put(-8,20){$1$}
		            \put(7,9){$2$}
		            \put(19,67){$3$}
		            \put(55,80){$4$}
		            \put(80,-10){$5$}
		            \put(95,82){$6$}
		            \put(112,-5){$7$}
		            \put(116,79){$8$}
		            \put(137,0){$9$}
		            \put(140,70){$10$}
		            \put(156,6){$11$}
		            \put(163,20){$12$}
		            \put(163,53){$13$}
		            \put(48,42){$\delta$}
			     \end{overpic}
		         \end{center}
		         \caption[]{Let $I=\{3,5,6,7,8,9,10\}$ with $\gamma=3$ and $\Gamma=10$.
		 					Its up and down decomposition is $I=(2,11)_{\Do_c}\cup[3,3]_{\Up_c}\cup[6,10]_{\Up_c}$.\\
							\rule{2mm}{0cm} For $\delta=\{2,8\}$, the up and down interval decomposition of $R_\delta$ is
							of the required sub-type of $(1,1)$. Since $(2,2)\cap I=\varnothing$ and $(8,11)\cap I = \{9,10\}$, 
							$\delta$ is associated to $S\in\left\{\{5,7,8\}, \{5,7,8,9\}, \{5,7,8,10\}, \{5,7,8,9,10\}\right\}$. 
							Thus~$\delta$ does not contribute to~$y_I$. \\
							\rule{2mm}{0cm} The only diagonals~$\delta$ associated to $J\subseteq I$ of the required sub-type
							of~$(1,1)$ that contribute to~$y_I$ are diagonals associated to precisely one subset $J\subseteq I$. 
							In this figure, only $\delta_2=\{2,10\}$ contributes $-(\tilde z_{R_{\delta_2}}^c-\tilde z_{[n]})$. }
		         \label{fig:example_cases_2bi}
		      \end{minipage}
			  	$ $
				
				\rule{0mm}{9mm}
		      \end{center}
		\end{figure}
			If $S\subseteq I$ with $\delta\in \mathcal D_S$, then~$\delta$ (associated to~$S$) can be of type 
			$\tilde z_{R_{\delta}}^c$ or $(\tilde z_{R_{\delta}}^c-\tilde z^c_{[n]})$ in the expression for~$y_I$. 
			The diagonal~$\delta$ is of type $\tilde z_{R_{\delta}}^c$ if and only if 
			$R_\delta=R_\delta\cap I$ and $S=R_\delta\cup M$ for some subset $M\subseteq (a,\tilde a)\cap I$.
			The diagonal~$\delta$ is of type $(\tilde z_{R_{\delta}}^c-\tilde z^c_{[n]})$ for all other subsets 
			$S\subseteq I$ with $\delta\in \mathcal D_S$, in particular, we conclude $R_\delta \supset R_\delta\cap S$.
			
			We now distinguish two sub-cases: either $\delta$ (associated to~$S$) is of type 
			$(\tilde z_{R_{\delta}}^c-\tilde z^c_{[n]})$ (in the expression for $y_I$) for all $S\subseteq I$ 
			with $\delta\in\mathcal D_S$ or there is a $S\subseteq I$ with $\delta\in \mathcal D_S$ such that $\delta$ 
			(associated to~$S$) is of type $\tilde z_{R_{\delta}}^c$ (in the expression for $y_I$).
			\begin{compactenum}[i.]
				\item $\delta$ is of type $(\tilde z_{R_{\delta}}^c-\tilde z^c_{[n]})$ for all $S\subseteq I$ 
					with $\delta\in \mathcal D_S$, see Figure~\ref{fig:example_cases_2bi}.\\
					\label{case:two_b_i}
					As mentioned, we have $R_\delta\supset R_\delta\cap S$ for all sets $S\subseteq I$ with $\delta\in \mathcal S$.
					Moreover, these sets are in bijection to the subsets of
					$\bigl((a,\tilde a)\cap I\bigr)\cup\bigl((\tilde\alpha,b)\cap I\bigr)$:
					\[
						S=\bigl(R_\delta\cap (I\setminus B)\bigr)\cup A\qquad\text{for $A\subseteq (a,\tilde a)\cap I$
						and $B \subseteq (\tilde\alpha,b)\cap I$.}
					\]
					If there is more than one set~$S\subseteq I$ with $\delta\in \mathcal D_S$, then collecting all the summands 
					$(\tilde z_{R_{\delta}}^c-\tilde z^c_{[n]})$ in the expression for~$y_I$ yields a vanishing alternating sum. 
					If there is only one set~$S\subseteq I$ with $\delta\in\mathcal D_S$ as associated diagonal then 
					$\bigl((a,\tilde a)\cap I\bigr)\cup\bigl((\tilde\alpha,b)\cap I\bigr)= \varnothing$ and it follows that
					$\Gamma=\tilde\alpha\in\Up_c$ and $\tilde a\in\{a,\gamma\}\cap\Do_c$. 
					
					For later use in this proof, we note that $\gamma\in\Do_c$ implies $\delta\in\{\delta_2,\delta_4\}$. 
					The only possible contributions in the expression for~$y_I$ are therefore
					\[
						(-1)^{|I\setminus R_{\delta_2}|}(\tilde z_{R_{\delta_2}}^c -\tilde z^c_{[n]})
						\qquad\text{and}\qquad
						(-1)^{|I\setminus R_{\delta_4}|}(\tilde z_{R_{\delta_4}}^c -\tilde z^c_{[n]}).
					\]
					But since the corresponding subsets $R_{\delta_2}\cap I$ and $R_{\delta_4}\cap I$ differ by~$\gamma$, 
					the effective contribution to $y_I$ is
					\[
						(-1)^{|I\setminus R_{\delta_2}|}\tilde z_{R_{\delta_2}}^c 
						+
						(-1)^{|I\setminus R_{\delta_4}|}\tilde z_{R_{\delta_4}}^c.
					\]
					If $\gamma\in\Up_c$, then $\delta=\delta_2$ and we obtain
					\[
						(-1)^{|I\setminus R_{\delta_2}|}(\tilde z_{R_{\delta_2}}^c-\tilde z^c_{[n]})
					\]
					as contribution for $y_I$.
				\item There is a $S\subseteq I$ with $\delta\in \mathcal D_S$ such that $\delta$ is of type $\tilde z_{R_{\delta}}^c$,
				 	see Figure~\ref{fig:example_cases_2bii}. \\
					\label{case:two_b_ii}
					Since $\delta$ must be the `rightmost' diagonal associated to $S$ if $\delta$ (associated 
					to $S$) is of type $\tilde z_{R_{\delta}}^c$ (in the expression for $y_I$), we conclude 
					$R_\delta=R_\delta\cap I$.    
					\noindent
					In particular, we have $\Gamma=n$ and $b=n+1$ and thus $(\tilde \alpha,b)\cap I\neq\varnothing$ 
					and $(\tilde \alpha,b)\cap I=(\tilde\alpha,b)$. If $(a,\tilde a)\cap I\neq \varnothing$, then 
					collecting the terms for $\tilde z_{R_{\delta}}^c$ and~$\tilde z^c_{[n]}$ in the expression for~$y_I$ 
					again yields no contribution. We may therefore assume $(a,\tilde a)\cap I= \varnothing$, that is
					$\tilde a\in\{a,\gamma\}\cap \Do_c$. First suppose that $\gamma\in\Do_c$. Then $\delta$ is either
					$\delta_a=\{a,\tilde\alpha\}$ or $\delta_\gamma=\{\gamma,\tilde\alpha\}$. Now $\delta$ is of type 
					$\tilde z_{R_{\delta}}^c$ in the expression of $y_I$ if and and only if $\delta$ is associated to $R_{\delta_a}$ 
					or $R_{\delta_\gamma}$. In all other situations, $\delta$ is of type 
					$(\tilde z_{R_{\delta}}^c-\tilde z^c_{[n]})$ in the expression of $y_I$ and is associated to a set 
					$R_{\delta_a}\setminus M$ or $R_{\delta_\gamma}\setminus M$ with non-empty 
					$M\subseteq (\tilde \alpha,n+1)$. Collecting the terms for $\tilde z_{R_{\delta_a}}^c$, 
					$\tilde z_{R_{\delta_\gamma}}^c$, and $\tilde z_{[n]}^c$ yields a vanishing contribution 
					as desired (collecting the terms for $\tilde z_{[n]}^c$ for fixed $\delta$ does not yield a 
					vanishing contribution, but the terms from $\delta_a$ and $\delta_\gamma$ cancel). If $\gamma\in\Up_c$ 
					then a similar argument gives 
					\[
						(-1)^{|I\setminus R_\delta|}\tilde z_{[n]}^c 
						\qquad
						\text{for $\delta=\{a,\tilde \alpha\}$ with $\tilde \alpha\in\Up_c$ and $R_\delta=R_\delta\cap I$}
					\]
					as contribution for~$y_I$.
			\end{compactenum}
		\end{compactenum}
		\begin{figure}[b]
		      \begin{center}
			  \rule{0mm}{5mm}
			
		      \begin{minipage}{0.95\linewidth}
		         \begin{center}
		         \begin{overpic}
		            [width=0.4\linewidth]{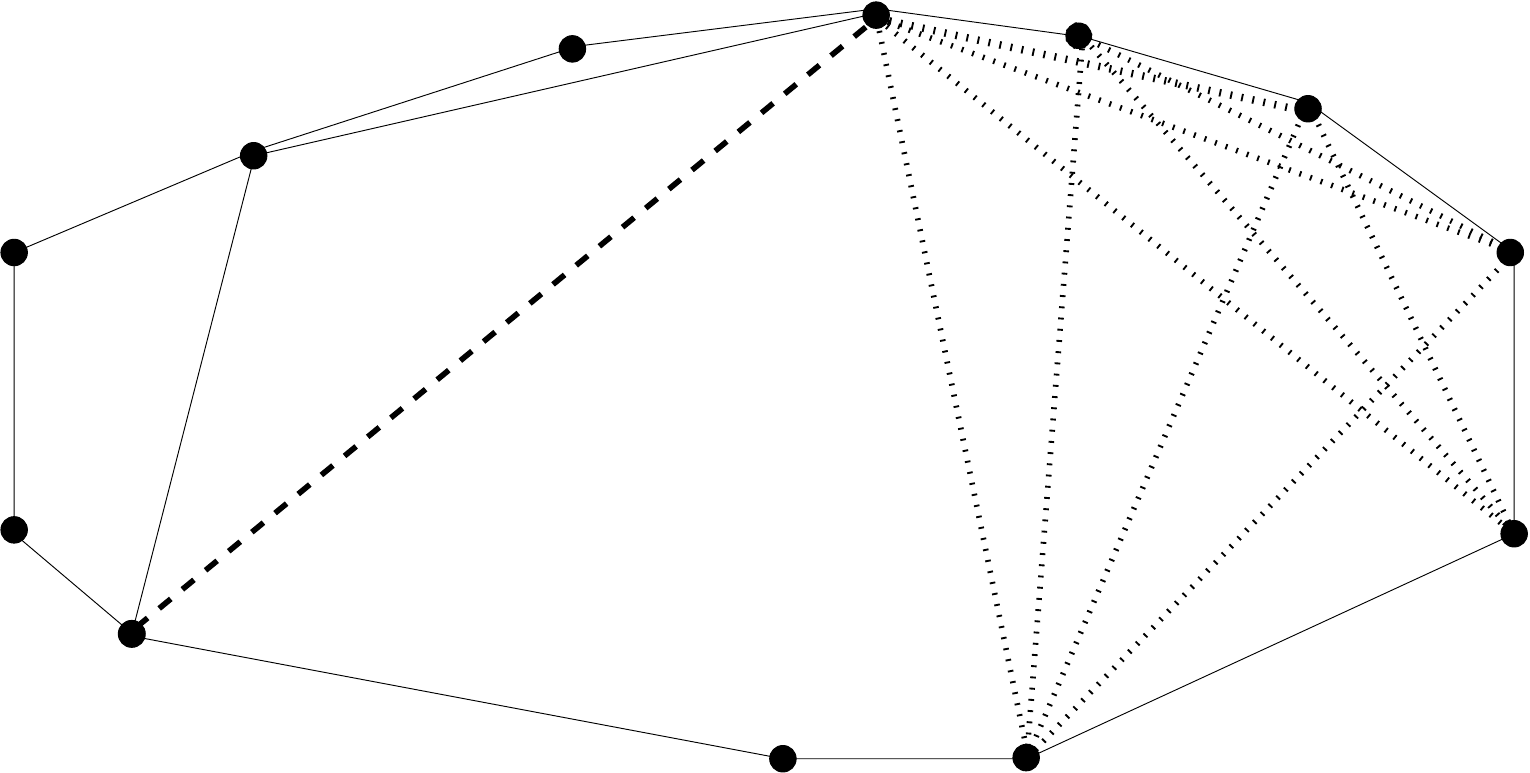}
		            \put(-8,53){$0$}
		            \put(-8,20){$1$}
		            \put(7,9){$2$}
		            \put(19,67){$3$}
		            \put(55,80){$4$}
		            \put(80,-10){$5$}
		            \put(95,82){$6$}
		            \put(112,-5){$7$}
		            \put(116,79){$8$}
		            \put(140,70){$9$}
		            \put(163,20){$10$}
		            \put(163,53){$11$}
		            \put(48,50){$\delta$}
			     \end{overpic}
		         \end{center}
		         \caption[]{Consider $I=\{3,5,6,7,8,9,10\}$ with $\gamma=3$ and $\Gamma=10$. The up and down
		 					interval decomposition is $I=(2,11)_{\Do_c}\cup[3,3]_{\Up_c}\cup[6,9]_{\Up_c}$.\\ 
							\rule{2mm}{0cm} For $\delta=\{2,6\}$, the up and down interval decomposition 
							of~$R_\delta$ has the required sub-type of~$(1,1)$. We have $R_\delta= R_\delta\cap I$,
							so $J=\{ 5,6,7,8,9,10\}$ is the unique~$J\subseteq I$ such that $\delta\in\mathcal D_J$ 
							is of type~$\tilde z_{R_\delta}^c$ (in the expression for~$y_I$). For all other 
							$S\subseteq I$ with $\delta\in\mathcal D_S$, $\delta$ is of type 
							$(\tilde z_{R_{\delta}}^c-\tilde z^c_{[n]})$ in the expression of $y_I$.
							Since $\gamma\in\Up_c$,~$\delta$ contributes 
							$(-1)^{|I\setminus R_\delta|}\tilde z_{[n]}^c=-\tilde z_{[n]}^c$ to~$y_I$. 
							We have other diagonals contributing to~$y_I$. In this example (and considering 
							only the specific sub-type) these are the proper diagonals~$\delta$ with 
							$R_{\delta}=R_{\delta}\cap I$ and end-point~$a=2$, that is, $\delta=\{2,8\}$
							and $\delta = \{2,9\}$, their contribution is $-\tilde z_{[n]}^c$ and 
							$\tilde z_{[n]}^c$.}
		         \label{fig:example_cases_2bii}
		      \end{minipage}
			  	$ $
				
				\rule{0mm}{9mm}
		      \end{center}
		\end{figure}
		
	\item $R_\delta$ has up and down decomposition of type~$(1,2)$.\\
		\label{case:three}
		If $R_\delta$ is of type~$(1,2)$ then $\delta=\{\alpha,\beta\}$ with $\alpha,\beta\in \Up_c$ and there 
		is $u\in\Up_c$ such that $a<\alpha<u<\beta< b$. This in turn gives
		\[
		  R_{\delta}=(0, n+1)_{\Do_c}\cup [u_1,\alpha]_{\Up_c}\cup [\beta,u_m]_{\Up_c}
		\]
		as up and down interval decomposition for $R_\delta$. 
		By arguments as before, we conclude that collecting the terms for~$\tilde z_{R_\delta}^c$ and 
		$\tilde z^c_{[n]}$ yields a vanishing contribution to $y_I$ if $(a,\alpha)\cap I\neq \varnothing$. 
		We therefore assume that $(a,\alpha)\cap I= \varnothing$ which is equivalent to $\gamma=\alpha\in\Up_c$. 
		As a consequence, $\delta$ is an associated diagonal of~$S\subseteq I$ if and only if $S=(R_\delta\cap I)\setminus M$ 
		for some $M\subseteq (\beta,b)\cap I$.
		
		\medskip
		We now distinguish two cases: either there is a $S\subseteq I$ with $\delta\in\mathcal D_S$ such that $\delta$ 
		(associated to~$S$) is of type~$\tilde z_{R_{\delta}}^c$ (in the expression for~$y_I$) or not.		
		\begin{compactenum}[a.]
			\item There is no $S\subseteq I$ with $\delta\in\mathcal D_S$ such that $\delta$ (associated to~$J$) is of 
				type~$\tilde z_{R_{\delta}}^c$, see Figure~\ref{fig:example_cases_3a}.\\
				\label{case:three_a}
				If $(\beta,b)\cap I\neq \varnothing$ then collecting the terms $\tilde z_{R_\delta}^c$ and $\tilde z^c_{[n]}$ 
				cancel respectively. If $(\beta,b)\cap I= \varnothing$ then we have $\Gamma=\beta\in\Up_c$ and $\delta=\delta_4$. 
				In this situation, $\delta$ has a unique contribution to~$y_I$ which equals 
				$(-1)^{|I\setminus R_{\delta_4}|}(\tilde z_{R_{\delta_4}}^c - \tilde z^c_{[n]})$.
			\item There is a set $S\subseteq I$ with $\delta\in \mathcal D_S$ such that $\delta$ is of type $\tilde z_{R_{\delta}}^c$, 
				see Figure~\ref{fig:example_cases_3b}. \\
				\label{case:three_b}
				Since $\delta$ is the `rightmost' diagonal associated to~$S\subseteq I$ and since $(a,\alpha)\cap I =\varnothing$,
				we conclude that $b=n+1$ and $\Gamma=n\in\Do_c$ (recall that we also have $\alpha=\gamma\in \Up_c$). Now observe
				that the set $S\subseteq I$ with $\delta\in\mathcal D_S$ such that~$\delta$ (associated to~$S$) is of 
				type~$\tilde z_{R_{\delta}}^c$ (in the expression for~$y_I$) is unique: it is $R_\delta\cap I$. In particular, 
				we have $[\beta,n]\cap I=[\beta,n]$. Collecting the terms $\tilde z_{R_\delta}^c$ for all subsets~$S\subseteq I$ 
				with $\delta\in\mathcal D_S$ cancel, but collecting the terms $\tilde z^c_{[n]}$ does not vanish: we have a
				contribution of $(-1)^{|I\setminus R_\delta|}\tilde z_{[n]}^c$ to~$y_I$. We conclude that every 
				diagonal~$\delta=\{\gamma,\beta\}$ with $\beta\in\Up_c$, $[\beta,n]\cap I=[\beta,n]$ and 
				$\{ \gamma,\beta \} \neq \{ u_r,u_{r+1}\}$ contributes $(-1)^{|I\setminus R_\delta|}\tilde z_{[n]}^c$ to $y_I$.
		\end{compactenum}
	\end{compactenum}
	\begin{figure}[b]
		  $ $\\[-5mm]
	      \begin{center}
	      \begin{minipage}{0.95\linewidth}
	         \begin{center}
	         \begin{overpic}
	            [width=0.4\linewidth]{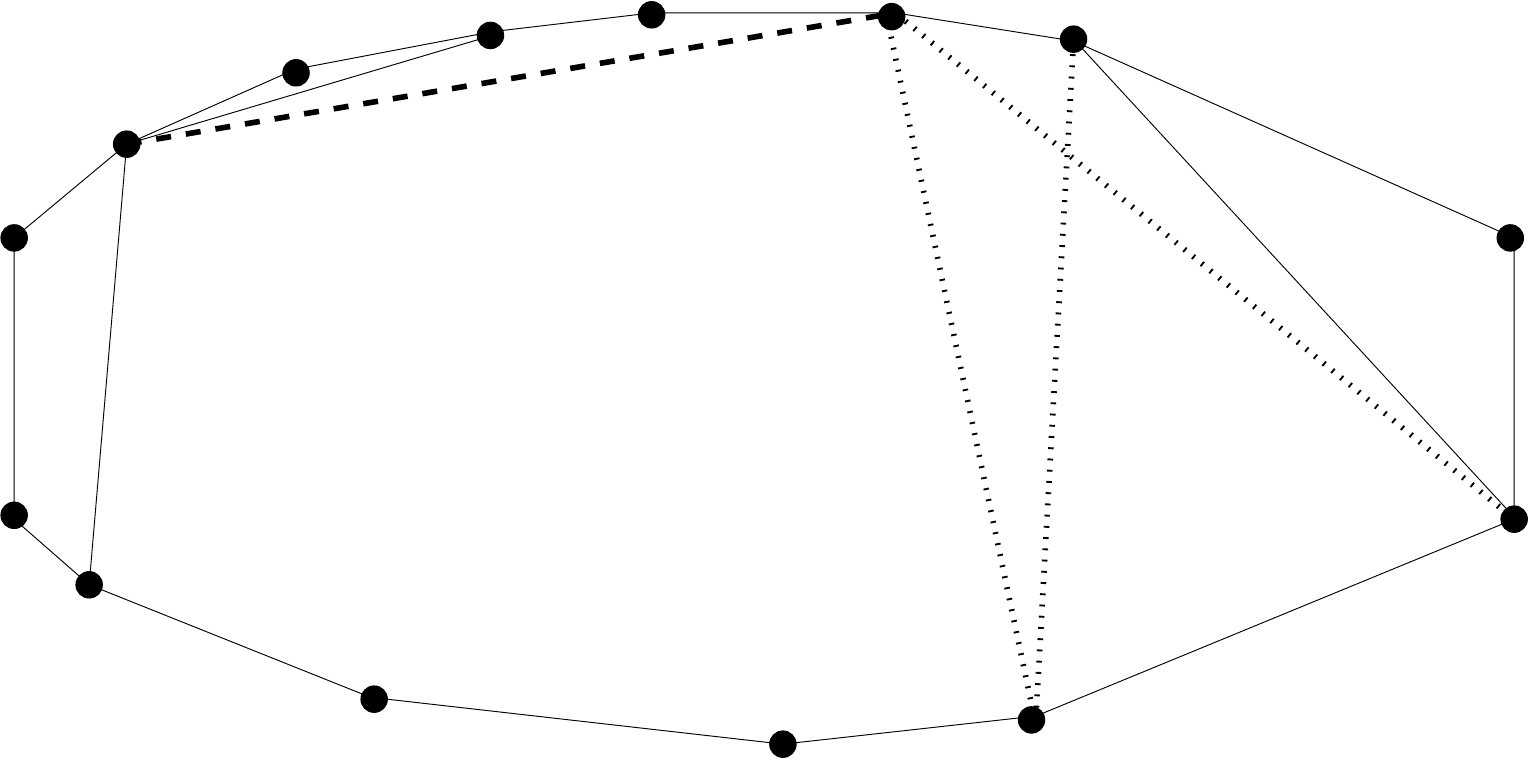}
	            \put(-8,53){$0$}
	            \put(-8,20){$1$}
	            \put(7,9){$2$}
	            \put(8,68){$3$}
	            \put(23,73){$4$}
	            \put(35,-3){$5$}
	            \put(45,78){$6$}
	            \put(62,80){$7$}
	            \put(80,-10){$8$}
	            \put(95,80){$9$}
	            \put(110,-5){$10$}
	            \put(115,77){$11$}
	            \put(163,20){$12$}
	            \put(163,53){$13$}
	            \put(65,65){$\delta$}
		     \end{overpic}
	         \end{center}
	         \caption[]{Consider $I=\{3,5,6,7,8,9,10,11\}$ with $\gamma=3$ and $\Gamma=11$. The up and down
	 					interval decomposition is $I=(2,12)_{\Do_c}\cup[3,3]_{\Up_c}\cup[6,11]_{\Up_c}$.\\ 
						\rule{2mm}{0cm} For~$\delta=\{3,9\}$, the up and down interval decomposition of~$R_\delta$
						is of the required sub-type of~$(1,2)$. Since $(\beta,b)\cap I= (9,12) = \{10,11\}$, the 
						diagonal~$\delta$ is not associated to a unique $J\subseteq I$ and does not contribute to~$y_I$. 
						Only diagonals of the required sub-type of~$(1,2)$ that are associated to a unique $J\subseteq I$
						contribute to~$y_I$. In this example, only $\delta_4=\{3,11\}$ is of this type. }
	         \label{fig:example_cases_3a}
	      \end{minipage}
		  	$ $
			
			\rule{0mm}{9mm}
	      \end{center}
	\end{figure}
	\begin{figure}[b]
      \begin{center}
      \begin{minipage}{0.95\linewidth}
         \begin{center}
         \begin{overpic}
            [width=0.4\linewidth]{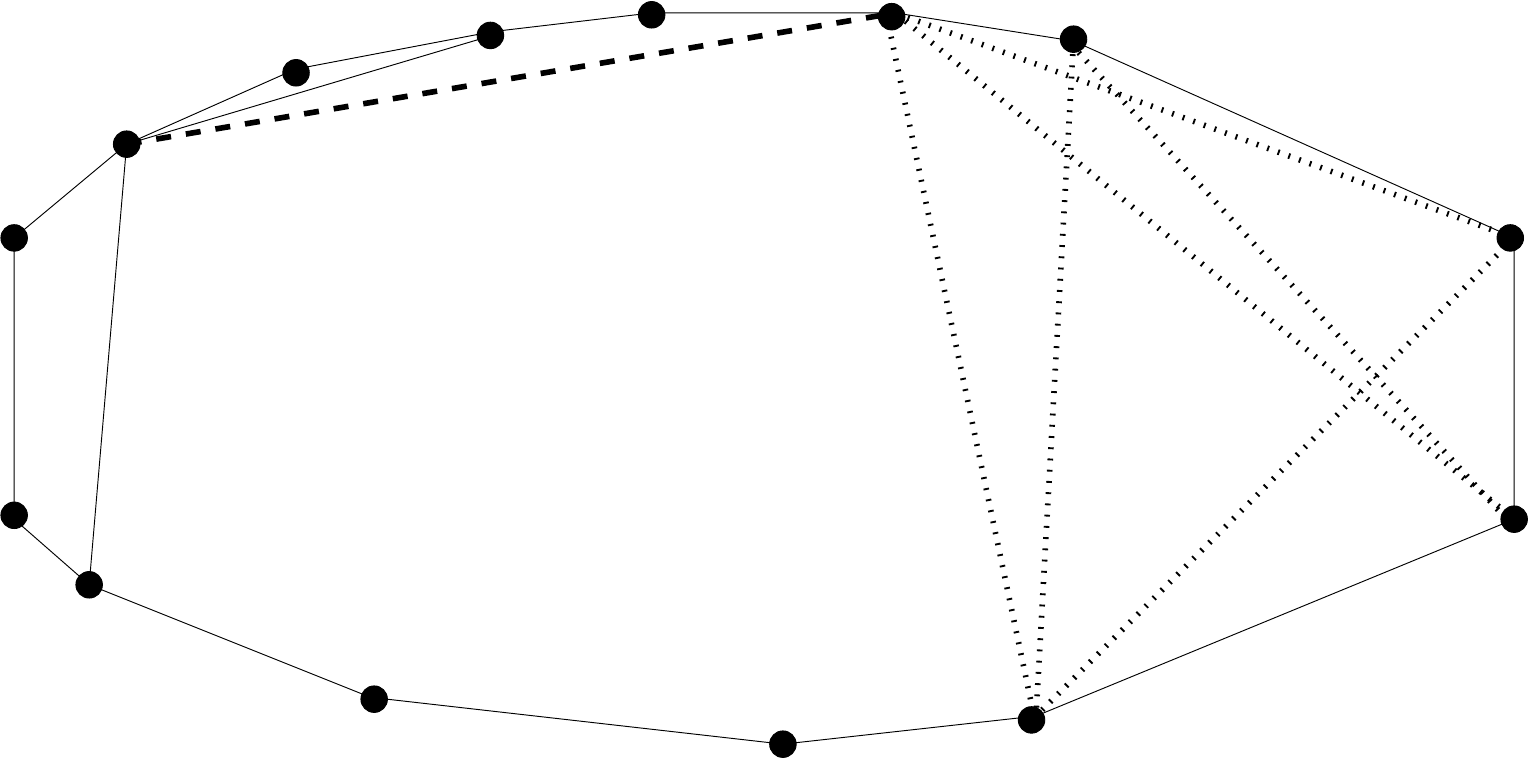}
            \put(-8,53){$0$}
            \put(-8,20){$1$}
            \put(7,9){$2$}
            \put(8,68){$3$}
            \put(23,73){$4$}
            \put(35,-3){$5$}
            \put(45,78){$6$}
            \put(62,80){$7$}
            \put(80,-10){$8$}
            \put(95,80){$9$}
            \put(110,-5){$10$}
            \put(115,77){$11$}
            \put(163,20){$12$}
            \put(163,53){$13$}
            \put(65,65){$\delta$}
	     \end{overpic}
         \end{center}
         \caption[]{Consider $I=\{3,5,6,7,8,9,10,11,12\}$ with $\gamma=3$ and $\Gamma=12$. The up and down  interval 
					decomposition is $I=(2,13)_{\Do_c}\cup[3,3]_{\Up_c}\cup[6,11]_{\Up_c}$.\\
					\rule{2mm}{0cm} For $\delta=\{3,9\}$, the up and down interval decomposition of $R_\delta$ is
					of the required sub-type of~$(1,2)$ and since $(\beta,b)\cap = I(9,13)\cap I=\{10,11,12\}\neq \varnothing$, 		
					the diagonal~$\delta$ is associated to eight sets. The contribution of~$\delta$ to~$y_I$ is 
					$-\tilde z^c_{[n]}$.\\
					\rule{2mm}{0cm} In this example, the four diagonals
					$\delta^\prime\in \left\{\{3,6\},\ \{3,7\},\ \{3,9\},\ \{3,11\}\right\}$ are of the required sub-type 
					of~$(1,2)$ each contributes $(-1)^{|I\setminus R_{\delta'}|}\tilde z_{R_{\delta'}}^c$ to $y_I$. }
         \label{fig:example_cases_3b}
      \end{minipage}
	  	$ $
		
		\rule{0mm}{5mm}
      \end{center}
	\end{figure}
	\medskip
	After this analysis for the possible contributions to~$y_I$ induced by proper diagonals, we now prove 
	\[
		y_I = \sum_{\delta\in\mathscr D_I}(-1)^{|I\setminus R_{\delta}|}\tilde z_{R_{\delta}}^c
	\]
	where we assume that $I$ is a non-empty proper subset of $[n]$ with a nested up and down decomposition and
	$|\mathscr D_I|=4$. We distinguish the following four cases: 
	\begin{compactenum}[1.]
		\item $\gamma, \Gamma\in\Do_c$.\\
			Then $\delta_1$, $\delta_2$, $\delta_3$, and $\delta_4$ contribute $(-1)^{|I\setminus R_\delta|}\tilde z_{R_{\delta}}^c$ 
			to~$y_I$ according to Case~$1$ and no other diagonal contributes according to the previous analysis. The claim follows
			immediately.
		\item $\gamma\in\Do_c$ and $\Gamma\in\Up_c$.\\
			Then $\delta_1$ and $\delta_3$ do contribute $(-1)^{|I\setminus R_\delta|}\tilde z_{R_{\delta}}^c$ to~$y_I$ according to 
			Case~$1$, while $\delta_2$ and $\delta_4$ contribute $(-1)^{|I\setminus R_\delta|}\tilde z_{R_{\delta}}^c$ to~$y_I$ 
			according to Case~\ref{case:two_b_i}. No other diagonal contributes to $y_I$. The claim follows immediately.
		\item $\gamma,\Gamma\in\Up_c$.\\
			The only diagonals with a contribution to $y_I$ are~$\delta_1$ (by Case~$1$), $\delta_2$ (by Case~\ref{case:two_b_i}),
			$\delta_3$ (by Case~\ref{case:two_a}) and $\delta_4$ (by Case~\ref{case:three_a}). Taking their contribution into account, 
			we obtain
			\[
				y_I = (-1)^{|I\setminus R_{\delta_1}|}\tilde z_{R_{\delta_1}}^c
						+ (-1)^{|I\setminus R_{\delta_2}|}(\tilde z_{R_{\delta_2}}^c-\tilde z_{[n]}^c)
						+ (-1)^{|I\setminus R_{\delta_3}|}\tilde z_{r_{\delta_3}}^c
						+ (-1)^{|I\setminus R_{\delta_4}|}(\tilde z_{R_{\delta_4}}^c-\tilde z_{[n]}^c).
			\]
			The claim follows since $I\setminus R_{\delta_2}$ and $I\setminus R_{\delta_4}$ differ by~$\gamma$.
		\item $\gamma\in\Up_c$ and $\Gamma\in\Do_c$.\\
			We distinguish the two sub-cases $\Gamma\neq n$ and $\Gamma=n$.
			\begin{compactenum}[(a)]
				\item $\Gamma\neq n$ implies that there is no $u\in\Up_c$ such that $[u,n]=[u,n]\cap I$.\\
					  In this situation, $\delta_1$ and $\delta_3$ contribute $(-1)^{|I\setminus R_\delta|}\tilde z_{R_{\delta}}^c$ 
					  to~$y_I$ according to Case~\ref{case:one} and $\delta_2$ and~$\delta_4$ contribute 
					  $(-1)^{|I\setminus R_\delta|}\tilde z_{R_{\delta}}^c$ according to Case~\ref{case:two_a}. No other 
					  diagonal contributes, so the claim follows immediately.
				\item $\Gamma=n$. \\
					  If there is no $u\in\Up_c$ such that $[u,n]=[u,n]\cap I$ then $\delta_1$ and $\delta_2$ contribute according
					  to Case~\ref{case:one} and $\delta_3$ and $\delta_4$ contribute according to Case~\ref{case:two_a}. No other
					  diagonal contributes, so the claim follows immediately. 
					
					  If there exists $u\in\Up_c$ such that $[u,n]=[u,n]\cap I$ then denote by~$u_{min}$ the smallest element 
					  of $\Up_c$ such that $[u_{min},n]=[u_{min},n]\cap I$. Now diagonals $\delta_1$ and $\delta_2$ contribute
					  to $y_I$ according to Case~\ref{case:one} and diagonals $\delta_3$, $\delta_4$ according to Case~\ref{case:two_a}.
					  But in this situation, according to Cases~\ref{case:two_b_ii} and~\ref{case:three_b}, we also have 
					  contributions of diagonals $\{a,u\}$ and $\{\gamma,u\}$ for $u\in[u_{min},u_m]_{\Up_c}$. This yields 
					  \[
						\sum_{\delta\in\mathscr D_I}(-1)^{|I\setminus R_\delta|}\tilde z_{R_\delta}
						+ \sum_{\begin{smallmatrix} \delta=\{a,\alpha\} \text{ with}\\ \
														\alpha\in [u_{min},u_m]_{\Up_c} 
								\end{smallmatrix}}
								(-1)^{|I\setminus R_\delta|}\tilde z_{[n]}^c
						+\sum_{\begin{smallmatrix} \delta=\{\gamma,\alpha\}\not\in\partial Q \text{ with}\\ 
														\alpha\in [u_{min},u_m]_{\Up_c}
								\end{smallmatrix}}
								(-1)^{|I\setminus R_\delta|}\tilde z_{[n]}^c.
					  \]
					  But the second and third sum cancel, so we end up with the claim.
			\end{compactenum}
	\end{compactenum}
\end{proof}

\noindent
In fact, the methods used in the proof of Proposition~\ref{prop:y_I_for_non-degenerate_D_I} suffice to prove the 
degenerate cases $\mathscr D_I\neq \{\delta_1,\delta_2,\delta_3,\delta_4\}$ as well. But before we try to analyse
these cases, we remark that some subsets of~$\{ \delta_1,\delta_2,\delta_3,\delta_4\}$ never form a set~$\mathscr D_I$ 
associated to some~$I\subseteq[n]$ and Coxeter element~$c$.
\begin{lem}$ $\\
	Let $n\geq 3$ and $I\subset [n]$ be non-empty with up and down interval decomposition of type~$(1,w)$. Then
	\label{lem:possible_D_I}
	\begin{compactenum}[(a)]
		\item There is no partition $[n]=\Do_c\sqcup \Up_c$ induced by a Coxeter element~$c$ and no non-empty 
			  $I\subset[n]$ such that $\mathscr D_I$ is one of the following sets:\label{lem_which_D_I_possible_a}
			  \[
			    \varnothing,\quad
				\{ \delta_2\}, \quad \{\delta_3\},\quad \{\delta_4\},\quad \{\delta_1,\delta_2\},\quad 
				\{\delta_1,\delta_3\},\quad \{\delta_2,\delta_4\},\quad\text{or}\quad \{\delta_3,\delta_4\}.
			  \]
		\item There is a partition $[n]=\Do_c\sqcup \Up_c$ induced by a Coxeter element~$c$ and a non-empty $I\subset[n]$
			  such that $\mathscr D_I$ is one of the following sets:\label{lem_which_D_I_possible_b}
			  \[
		    	\qquad\{\delta_1\},\
				\{ \delta_1,\delta_4\},\ \{\delta_2,\delta_3\},\ \{\delta_1,\delta_2,\delta_3\},\
				\{\delta_1,\delta_2,\delta_4\},\ \{\delta_1,\delta_3,\delta_4\},\ \{\delta_2,\delta_3,\delta_4\},
				\text{ or } \{\delta_1,\delta_2,\delta_3,\delta_4\}.
		  	  \]
	\end{compactenum}
\end{lem}
\noindent
The proof of Part~(\ref{lem_which_D_I_possible_a}) is left to the reader, while the situation of 
Part~(\ref{lem_which_D_I_possible_b}) is carefully discussed in Section~\ref{sec_characterisation_of_possible_D_I}.
\begin{lem}\label{lem:other_special_D_I}$ $\\
	Let $n\geq 3$ and $I\subset [n]$ be non-empty with up and down interval decomposition of type~$(1,w)$ and
	$|\mathscr D_I|<4$. Then
	\begin{compactenum}[(a)]
		\item Suppose that $I$ satisfies one of the following conditions\label{cor:other_special_D_I_first_case}
			  \begin{compactenum}[(i)]
				 \item $\mathscr D_I=\{ \delta_1\}$ (Lemma~\ref{lem_characterise_D_I=delta1}),
			  	 \item $\mathscr D_I=\{ \delta_1,\delta_3,\delta_4\}$, $(a,b)_\Do=\{\Gamma\}$, and $\gamma\in\Up_c$,
						that is, Cases (\ref{lem:delta1_delta3_delta4_b}) and (\ref{lem:delta1_delta3_delta4_c})
						of Lemma~\ref{lem:delta1_delta3_delta4},
				 \item $\mathscr D_I=\{ \delta_1,\delta_2,\delta_4\}$, $(a,b)_\Do=\{\gamma\}$, and $\Gamma\in\Up_c$,
						that is, Cases (\ref{lem:delta1_delta2_delta4_b}) and (\ref{lem:delta1_delta2_delta4_c})
						of Lemma~\ref{lem:delta1_delta2_delta4},
				 \item $\mathscr D_I=\{ \delta_1,\delta_2,\delta_3\}$ and $(a,b)_\Do=\{\gamma,\Gamma\}$, 
						that is, Case (\ref{lem:delta1_delta2_delta3_a}) of Lemma~\ref{lem:delta1_delta2_delta3}, or
				 \item $\mathscr D_I=\{ \delta_2,\delta_3,\delta_4\}$ and $(a,b)_\Do=\varnothing$,
						that is, Lemma \ref{lem:delta2_delta3_delta4}
			  \end{compactenum}
			  Then the Minkowski coefficient~$y_I$ of $\As^c_{n-1}$ is
			  \[
				y_I = \sum_{\delta\in\mathscr D_I}(-1)^{|I\setminus R_{\delta}|}z_{R_{\delta}}.
			  \]
		\item Suppose that $I$ satisfies one of the following conditions\label{cor:other_special_D_I_second_case}
			  \begin{compactenum}[(i)]
				 \item $\mathscr D_I=\{\delta_1,\delta_4\}$, that is, Cases (\ref{delta1_delta4_case_a}) and (\ref{delta1_delta4_case_b})
				 		of Lemma \ref{lem_characterise_D_I=delta1_delta4},
				 \item $\mathscr D_I=\{\delta_2,\delta_3\}$, that is, Case (\ref{delta2_delta3_a}) and (\ref{delta2_delta3_b}) 
						of Lemma \ref{lem_characterise_D_I=delta2_delta3},
			  	 \item $\mathscr D_I=\{ \delta_1,\delta_3,\delta_4\}$ and $\bigcup_{i=1}^k [\alpha_i,\beta_i]_{\Up_c}=\{\Gamma\}$,	
						that is, Case (\ref{lem:delta1_delta3_delta4_a}) of Lemma~\ref{lem:delta1_delta3_delta4},
				 \item $\mathscr D_I=\{ \delta_1,\delta_2,\delta_4\}$ and $\bigcup_{i=1}^k [\alpha_i,\beta_i]_{\Up_c}=\{\gamma\}$,
						that is, Case (\ref{lem:delta1_delta2_delta4_a}) of Lemma~\ref{lem:delta1_delta2_delta4},
				 \item $\mathscr D_I=\{ \delta_1,\delta_2,\delta_3\}$ and 
						$\bigcup_{i=1}^k [\alpha_i,\beta_i]_{\Up_c}=\{\gamma,\Gamma\}$, that is, Case 
						(\ref{lem:delta1_delta2_delta3_b}) of Lemma \ref{lem:delta1_delta2_delta3}.
			  \end{compactenum}
			  Then the Minkowski coefficient~$y_I$ of $\As^c_{n-1}$ is
			  \[
				y_I = (-1)^{|\{\gamma,\Gamma\}|}z_{[n]} + \sum_{\delta\in\mathscr D_I}(-1)^{|I\setminus R_{\delta}|}z_{R_{\delta}}.
			  \]
	\end{compactenum}
\end{lem} 
\begin{proof}
	The proof of the claim is a study of the $14$ mentioned cases that characterise the non-empty proper subsets 
	$I \subset [n]$ with $\mathscr D_I\neq \{\delta_1,\delta_2,\delta_3,\delta_4\}$. These $14$ cases are described in 
	detail in Section~\ref{sec_characterisation_of_possible_D_I}, the proofs are along the lines of the proof of
	Proposition~\ref{prop:y_I_for_non-degenerate_D_I}.
\end{proof}

\begin{lem}\label{lem:which_R_delta_i}$ $\\
	For $n\geq 3$, let $I$ be non-empty proper subset of~$[n]$ with up and down interval decomposition of type~$(1,w)$
	and $|\mathscr D_I|<4$.
	\begin{compactenum}[(a)]
		\item In all cases of Part~(\ref{cor:other_special_D_I_first_case}) of Lemma~\ref{lem:other_special_D_I}
			  we have $R_\delta=\varnothing$ if $\delta\in\{ \delta_1,\delta_2,\delta_3,\delta_4\}\setminus \mathscr D_I$. Thus
			  \[
				y_I = \sum_{i=1}^4(-1)^{|I\setminus R_{\delta_i}|}z_{R_{\delta_i}}.
			  \]
		\item In all cases of Part~(\ref{cor:other_special_D_I_second_case}) of Lemma~\ref{lem:other_special_D_I}
			  there is precisely one $\delta\in\{ \delta_1,\delta_2,\delta_3,\delta_4\}\setminus \mathscr D_I$
			  with $R_{\delta}=[n]$:
			  \begin{compactenum}[(i)]
				\item $R_{\delta_2}=[n]$ in Case~(\ref{delta1_delta4_case_a}) of Lemma~\ref{lem_characterise_D_I=delta1_delta4} 
					  and Case~(\ref{lem:delta1_delta3_delta4_a}) of Lemma~\ref{lem:delta1_delta3_delta4} and we have\\[2mm]
					  \centerline{$
						y_I = (-1)^{|I\setminus R_{\delta_2}|}z_{R_{\delta_2}} + \sum_{\delta\in\mathscr D_I}(-1)^{|I\setminus R_{\delta}|}z_{R_{\delta}}
						    = \sum_{i=1}^4(-1)^{|I\setminus R_{\delta_i}|}z_{R_{\delta_i}}.
					  $}\\[-1mm]
				\item $R_{\delta_3}=[n]$ in Case~(\ref{delta1_delta4_case_b}) of Lemma~\ref{lem_characterise_D_I=delta1_delta4} and 
					  Case~(\ref{lem:delta1_delta2_delta4_a}) of Lemma~\ref{lem:delta1_delta2_delta4} and we have\\[2mm]
					  \centerline{$
						y_I = (-1)^{|I\setminus R_{\delta_3}|}z_{R_{\delta_3}} + \sum_{\delta\in\mathscr D_I}(-1)^{|I\setminus R_{\delta}|}z_{R_{\delta}}
						    = \sum_{i=1}^4(-1)^{|I\setminus R_{\delta_i}|}z_{R_{\delta_i}}.
					  $}\\[-1mm]
				\item $R_{\delta_4}=[n]$ in both cases of Lemma~\ref{lem_characterise_D_I=delta2_delta3} and in 
					  Case~(\ref{lem:delta1_delta2_delta3_b}) of Lemma~\ref{lem:delta1_delta2_delta3} and we have\\[2mm]
					  \centerline{$
						y_I = (-1)^{|I\setminus R_{\delta_4}|+|\{\gamma,\Gamma\}|}z_{R_{\delta_4}} + \sum_{\delta\in\mathscr D_I}(-1)^{|I\setminus R_{\delta}|}z_{R_{\delta}}.
					  $}\\[-1mm]
			  \end{compactenum}
			  Moreover, we have $\gamma\neq\Gamma$ except for Case~(\ref{delta2_delta3_a}) of Lemma~\ref{lem_characterise_D_I=delta2_delta3}
			  when $\gamma=\Gamma\in\Up_c$.
	\end{compactenum}
\end{lem}
\begin{proof}
	The first case is trivial, since we only add vanishing terms to $\sum_{\delta\in\mathscr D_I}(-1)^{|I\setminus R_{\delta}|}z_{R_{\delta}}$.
	
	\medskip
	The second case is a bit more involved. First observe that $\gamma\neq\Gamma$ except for Case~(\ref{delta2_delta3_a}) of 
	Lemma~\ref{lem_characterise_D_I=delta2_delta3} when $\gamma=\Gamma\in\Up_c$. Now, using the description given in 
	Section~\ref{sec_characterisation_of_possible_D_I}, it is straighforward to check 
	$(-1)^{|\{\gamma,\Gamma\}|}=(-1)^{|I\setminus R_{\delta_2}|}$ for the first subcase, 
	$(-1)^{|\{\gamma,\Gamma\}|}=(-1)^{|I\setminus R_{\delta_3}|}$ for the second subcase and 
	$(-1)^{|\{\gamma,\Gamma\}|}=(-1)^{|I\setminus R_{\delta_4}|+|\{\gamma,\Gamma\}|}$
	for the third subcase.
\end{proof}

\begin{lem}$ $\\
	Let $I$ be non-empty subset of~$[n]$ with up and down interval decomposition of type~$(1,w)$. Then
	\[
		(-1)^{|I\setminus R_{\delta_1}|}=(-1)^{|I\setminus R_{\delta_2}|+1}
										=(-1)^{|I\setminus R_{\delta_3}|+1}
										=(-1)^{|I\setminus R_{\delta_4}|+|\{\gamma,\Gamma\}|}.
	\]\label{lem_relate_signs}
\end{lem}
\begin{proof}
	The claim for~$\delta_2$ follows from 
	\[
		R_{\delta_2}\cap I = \begin{cases}
								(R_{\delta_1}\cap I)\sqcup\{\Gamma\}, 	& \Gamma\in\Up_c,\\
								(R_{\delta_1}\cap I)\setminus\{\Gamma\}	& \Gamma\in\Do_c.
		\end{cases}
	\]
	The case for~$\delta_3$ is similar. For~$\delta_4$ we have to consider
	\[
		R_{\delta_4}\cap I = \begin{cases}
								(R_{\delta_1}\cap I)\sqcup\{\Gamma,\gamma\}							& \gamma,\Gamma\in\Up_c\\
								\left((R_{\delta_1}\cap I)\sqcup\{\gamma\}\right)\setminus\{\Gamma\}	& \gamma\in\Up_c,\ \Gamma\in\Do_c\\
								\left((R_{\delta_1}\cap I)\sqcup\{\Gamma\}\right)\setminus\{\gamma\}	& \Gamma\in\Up_c,\ \gamma\in\Do_c\\
								(R_{\delta_1}\cap I)\setminus\{\Gamma,\gamma\}						& \gamma,\Gamma\in\Do_c.
							 \end{cases}
	\]
\end{proof}
\medskip
\noindent
We combine Proposition~\ref{prop:y_I_for_non-degenerate_D_I}, Lemma~\ref{lem:which_R_delta_i}, and Lemma~\ref{lem_relate_signs} to 
obtain Theorem~\ref{thm:first_y_I_thm} if $I\subset [n]$ has an up and down interval decomposition of type~$(1,w)$:
\begin{cor}\label{cor:the_case_v=1}$ $\\
	Let $I$ be non-empty proper subset of~$[n]$ with up and down interval decomposition of type~$(1,w)$
	and $\mathscr D_I\subseteq \{ \delta_1,\delta_2,\delta_3,\delta_4\}$. Then
	\[
	y_I = (-1)^{|I\setminus R_{\delta_1}|}
			\left(
				z_{R_{\delta_1}}^c-z_{R_{\delta_2}}^c-z_{R_{\delta_3}}^c+z_{R_{\delta_4}}^c
			\right).
	\]
\end{cor}

The techniques to prove Proposition~\ref{prop:y_I_for_non-degenerate_D_I} also enable us to compute the Minkowski 
coefficient~$y_I$ of~ $\As^c_{n-1}$ if the up and down interval decomposition of~$I$ is of type~$(v,w)$, $v>1$, 
and $I\neq [n]$. 

\begin{lem}\label{lem_y_I_for_multiple_components}$ $\\
	Let $I$ be a non-empty proper subset of~$[n]$ with up and down interval decomposition of type~$(v,w)$ with $v>1$.
	Then $y_I=0$ for the Minkowski coefficient of of~$\As^c_{n-1}$.
\end{lem}
\begin{proof}
	For every proper diagonal~$\delta=\{d_1,d_2\}$ with $d_1<d_2$ that appears in the expression for~$y_I$, there
	is a nested component $N=(a_i,b_i)_\Do \sqcup \bigsqcup_{j=1}^{w_i}[\alpha_{i,j},\beta_{i,j}]_\Up$ of~$I$ such that
	$a_i\leq d_1<d_2\leq b_i$. Now $\delta$ appears in the expression for $y_I$ for every set $S$ where 
	$R_\delta\cap N \subseteq S \subseteq I$. Since $v>1$, the diagonal~$\delta$ never contributes to $y_I$.
\end{proof}
\noindent
We now analyse the remaining case $I=[n]$ and consider $(0,n+1)_{\Do}\sqcup [u_1,u_m]_{\Up}$ as up and down interval decomposition of~$I$.
\begin{lem}$ $\\
	For any partition~$\Do_c\sqcup \Up_c=[n]$ induced by some Coxeter element~$c$, the Minkowski coefficient~$y_{[n]}$ 
	satisfies\\[2mm]
	\centerline{$
		y_{[n]}= (-1)^{|[n]\setminus R_{\delta_1}|}
			  \left(
			 	 z_{R_{\delta_1}}^c-z_{R_{\delta_2}}^c-z_{R_{\delta_3}}^c+z_{R_{\delta_4}}^c
			  \right).
	$}\label{lem:y_I_for_[n]}
\end{lem}

\begin{proof}
	For $I=[n]$, we have $a=0$, $\gamma=1$, $\Gamma=n$, and $b=n+1$. We associate to the up and down interval 
	decomposition of~$[n]$ precisely one diagonal that is not proper and rewrite the formula for~$y_{[n]}$ as
	\[
		y_{[n]}=z_{[n]}+\sum_{J\subset [n]}(-1)^{|[n]\setminus J|}z_J.
	\]
	We are now interested in the contribution of proper diagonals that are associated to $J\subset [n]$ and 
	distinguish four cases. To find the contributions in each case, we proceed along the lines of the proof 
	of Proposition~\ref{prop:y_I_for_non-degenerate_D_I}.
	\begin{compactenum}[(1)]
		\item $\Up_c\neq\varnothing$ and $\Do_c\neq \{ 1,n\}$. \\
			  Then $\mathscr D_{[n]}=\{\delta_1, \delta_2, \delta_3, \delta_4\}$ and each diagonal of $\mathscr D_{[n]}$ 
			  contributes to $y_{[n]}$ as well as all proper diagonals $\{0,u\}$ and $\{1,u\}$ with $u\in \Up_c$ 
			  since $a=0$ and $\gamma=1$. Hence we have 
			  \[
			  	 \sum_{\delta\in\mathscr D_{[n]}}(-1)^{|[n]\setminus R_{\delta}|} z_{R_\delta}^c
					   + \sum_{\begin{smallmatrix} \delta=\{0,\alpha\} \text{ with}\\ 
														\alpha\in [u_2,u_m]_{\Up_c} 
								\end{smallmatrix}}
							(-1)^{|[n]\setminus R_\delta|}z_{[n]}^c
					   + \sum_{\begin{smallmatrix} \delta=\{1,\alpha\} \text{ with}\\ 
													\alpha\in [u_1,u_m]_{\Up_c}
							    \end{smallmatrix}}
							(-1)^{|[n]\setminus R_\delta|}z_{[n]}^c
			  \]
			  for $\sum_{J\subset [n]}(-1)^{|[n]\setminus J|}z_J$. Since $\{0,u_1\}$ is not a proper diagonal, the 
			  second and third sum do not cancel and the term $(-1)^{|[n]\setminus R_{\{1,u_1\}}|}z_{[n]}^c$ remains.
			  Now, $|[n]\setminus R_{\{1,u_1\}}|=1$ and 
			  \[
				\sum_{\delta\in\mathscr D_{[n]}}(-1)^{|[n]\setminus R_{\delta}|} z_{R_\delta}^c
				= (-1)^{|[n]\setminus R_{\delta_1}|}
					  \left(
					 	 z_{R_{\delta_1}}^c-z_{R_{\delta_2}}^c-z_{R_{\delta_3}}^c+z_{R_{\delta_4}}^c
					  \right)
			 \] 
			 imply the claim. 
		\item $\Up_c=\varnothing$ and $\Do_c\neq \{ 1,n\}$. \\
			Then $\mathscr D_{[n]}=\{\delta_2, \delta_3, \delta_4\}$ and we have
			\[
				\sum_{J\subset [n]}(-1)^{|[n]\setminus J|}z_J 
				= \sum_{\delta\in\mathscr D_{[n]}}(-1)^{|[n]\setminus R_{\delta}|} z_{R_\delta}^c.
			\]
			The claim follows now from $R_{\delta_1} = [n]$ and Lemma~\ref{lem_relate_signs}. 
		\item $\Up_c\neq\varnothing$ and $\Do_c= \{ 1,n\}$. \\
			  We have $\mathscr D_{[n]}=\{\delta_1, \delta_2, \delta_3\}$.
			  Now each diagonal of $\mathscr D_{[n]}$ and all proper diagonals $\{0,u\}$ and $\{1,u\}$ with $u\in \Up_c$ contribute 
			  to~$y_I$ since $a=0$ and $\gamma=1$. Similar to the fist case, a term $(-1)^{|[n]\setminus R_{\{1,u_1\}}|}z_{[n]}^c$ 	
			  is not canceled and we obtain
			  \[
				 y_{[n]} = (-1)^{|[n]\setminus R_{\delta_1}|}
								  	\left(
								 	 	z_{R_{\delta_1}}^c-z_{R_{\delta_2}}^c-z_{R_{\delta_3}}^c
								  	\right).
				\]	
				Since $z_{R_{\delta_4}}=z_\varnothing=0$, the claim follows.
		\item $\Up_c=\varnothing$ and $\Do_c= \{ 1,n\}$.\\ 
			  We have $\mathscr D_{[n]}=\{\delta_2, \delta_3\}$, 
			  $R_{\delta_1}=[n]$ and $R_{\delta_4}=\varnothing$. Hence\\
			  \centerline{$
					 y_{[n]} = z_{[n]} + \sum_{\delta\in\mathscr D_{[n]}}(-1)^{|[n]\setminus R_{\delta}|} z_{R_\delta}^c
					     = (-1)^{|[n]\setminus R_{\delta_1}}
								\left(
									z_{R_{\delta_1}}^c-z_{R_{\delta_2}}^c-z_{R_{\delta_3}}^c+z_{R_{\delta_4}}^c
								\right). 
					$}
	\end{compactenum}
\end{proof}

\section{Characterisation of $\mathscr D_I\neq \{\delta_1,\delta_2,\delta_3,\delta_4\}$ for $I\subset[n]$} \label{sec_characterisation_of_possible_D_I}
As stated in Lemma~\ref{lem:possible_D_I}, not all $15$ proper subsets of $\{\delta_1,\delta_2,\delta_3,\delta_4\}$ appear as 
set of proper diagonals~$\mathscr D_I$ for $I\subset [n]$ with up and down decomposition of type~$(1,w)$ and some Coxeter
element~$c$. Certain subsets are never obtained this way, a complete list is given in Lemma~\ref{lem:possible_D_I}. The proof
that they do not appear is not difficult, for example, we can show that if~$\mathscr D_I$ contains certain diagonal(s) 
then $\mathscr D_I$ is forced to contain certain others. In this section we discuss the second statement of Lemma~\ref{lem:possible_D_I}
in detail and study the sets~$\mathscr D_I$ with $|\mathscr D_I|<4$ that do appear. The seven proper subset of $\{\delta_1,\delta_2,\delta_3,\delta_4\}$ that are possible are characterised in 
Lemma~\ref{lem_characterise_D_I=delta1} -- Lemma~\ref{lem:delta2_delta3_delta4}. We identified~$14$ conditions 
for to charactisatise the $7$ subsets. As before, we assume that $\Do_c=\{d_1=1<d_2<\ldots<d_{\ell-1}<d_\ell=n\}$ and 
$\Up_c=\{u_1<u_2\ldots<u_m\}$ partition~$[n]$.

\begin{lem}\label{lem_characterise_D_I=delta1}
	If $\mathscr D_I=\{ \delta_1\}$, then $I=\{d_r\}$ with $1\leq r \leq \ell$.
\end{lem}
\begin{proof}
	$\delta_1\in\mathscr D_I$ implies $(a,b)_{\Do_c}\neq \varnothing$ and $\delta_2,\delta_3,\delta_4\not\in\mathscr D_I$
	imply $\gamma=\Gamma\in\Do_c$.
\end{proof}

\begin{lem}[compare Figure~\ref{fig:example_I}]\label{lem_characterise_D_I=delta1_delta4}$ $\\
	If $\mathscr D_I=\{ \delta_1,\delta_4\}$, then either
	\begin{compactenum}[(a)]
		\item $I=\{d_1,u_1\}$ and $u_1<d_2$, or \label{delta1_delta4_case_a}
		\item $I=\{u_m,d_{\ell}\}$ and $d_{\ell-1}<u_m$. \label{delta1_delta4_case_b}
	\end{compactenum}
\end{lem}

\begin{figure}
      \begin{center}
      \begin{minipage}{0.95\linewidth}
         \begin{center}
			\begin{tikzpicture}[thick]
				\filldraw [black] (-10:2cm) circle (2pt)
								  (10:2cm) circle (2pt)
							  	  (170:2cm) circle (2pt)
								  (190:2cm) circle (2pt);
				\draw (-10:2cm) node[anchor=west] {\scriptsize$d_\ell=n$} -- (10:2cm) node[anchor=west] {\scriptsize$n+1$}; 
				\draw (170:2cm) node[anchor=east] {\scriptsize$a=0$} -- (190:2cm) node[anchor=east] {\scriptsize$\gamma=d_1=1$}; 
				\draw (10:2cm) arc (26:154:2.2cm and 0.8cm);
				\draw (190:2cm) arc (206:334:2.2cm and 0.8cm);
				\filldraw [black] (210:1.35cm) node[anchor=north] {\scriptsize$b=d_{2}$} circle (2pt)
								  (158:1.6cm) node[anchor=south] {\scriptsize$\Gamma=u_{1}\ $} circle (2pt);
				\draw (158:1.6cm) -- (210:1.35cm);
				\draw [dashed] (170:2cm) -- (210:1.35cm);
				\draw [dashed] (190:2cm) -- (158:1.6cm);
		 \end{tikzpicture}
		 \qquad
		 \begin{tikzpicture}[thick]
			\filldraw [black] (-10:2cm) circle (2pt)
							  (10:2cm) circle (2pt)
							  (170:2cm) circle (2pt)
							  (190:2cm) circle (2pt);
			\draw (-10:2cm) node[anchor=west] {\scriptsize$\Gamma=d_\ell=n$} -- (10:2cm) node[anchor=west] {\scriptsize$b=n+1$}; 
			\draw (170:2cm) node[anchor=east] {\scriptsize$0$} -- (190:2cm) node[anchor=east] {\scriptsize$d_1=1$}; 
			\draw (10:2cm) arc (26:154:2.2cm and 0.8cm);
			\draw (190:2cm) arc (206:334:2.2cm and 0.8cm);
			\filldraw [black] (23.5:1.58cm) node[anchor=south] {\scriptsize$\gamma=u_{m}$} circle (2pt)
							  (330:1.35cm) node[anchor=north] {\scriptsize$\qquad a=d_{\ell-1}$} circle (2pt);
			\draw (330:1.35cm) -- (23.5:1.58cm);
			\draw [dashed] (330:1.35cm) -- (10:2cm);
			\draw [dashed] (23.5:1.58cm) -- (-10:2cm);
		 \end{tikzpicture}
		 \centerline{(a) $I=\{1,u_1\}$ and $u_1<d_2$ \rule{2.5cm}{0cm} (b) $I=\{u_m,n\}$ and $d_{\ell-1}<u_m$}
	  \end{center}
	  \caption[]{Schematic illustrations the two cases of $\mathscr D_{I}=\{\delta_1,\delta_4\}$ (Lemma~\ref{lem_characterise_D_I=delta1_delta4}).}
      \label{fig:example_I}
      \end{minipage}
      \end{center}
\end{figure}

\begin{proof}
	$\delta_2,\delta_3\not\in\mathscr D_I$ imply that $\{a,\Gamma\}$ and $\{\gamma,b\}$ are (non-degenerate) edges of $Q$. 
	In particular, neither $\gamma,\Gamma\in\Do_c$ nor $\gamma,\Gamma\in\Up_c$ is possible. 
	
	Firstly, suppose $\gamma\in\Do_c$ and $\Gamma\in\Up_c$. Then $\delta_2\not\in\mathscr D_I$ implies~$a=0$, $\Gamma=u_1$, 
	and~$\gamma=d_1=1$. Now $\delta_3\not\in\mathscr D_I$ yields~$b=d_2$ and~$\Gamma=u_1$ requires~$u_1<d_2$ and we have  
	shown~(\ref{delta1_delta4_case_a}).
	
	Secondly, suppose $\gamma\in\Up_c$ and $\Gamma\in\Do_c$. Then $\delta_3\not\in\mathscr D_I$ implies $b=n+1$, $\gamma=u_m$,
	and~$\Gamma=d_\ell=n$. Now $\delta_2\not\in\mathscr D_I$ yields~$a=d_{\ell-1}$ and $\gamma=u_m$ requires~$d_{\ell-1}<u_m$.
	This gives~(\ref{delta1_delta4_case_b}).
\end{proof}
\begin{lem}[compare Figure~\ref{fig:example_II}]\label{lem_characterise_D_I=delta2_delta3}$ $\\
	If $\mathscr D_I=\{ \delta_2,\delta_3\}$ then either
	\begin{compactenum}[(a)]
		\item $I=\{u_s\}$ with $1\leq s\leq m$, or\label{delta2_delta3_a}
		\item $I=\{u_s,u_{s+1}\}$ with $1\leq s<m$.\label{delta2_delta3_b}
	\end{compactenum} 
\end{lem}

\begin{figure}[b]
      \begin{center}
      \begin{minipage}{0.95\linewidth}
         \begin{center}
		  $ $\\[3mm]
			\begin{tikzpicture}[thick]
				\filldraw [black] (-10:2cm) circle (2pt)
								  (10:2cm) circle (2pt)
							  	  (170:2cm) circle (2pt)
								  (190:2cm) circle (2pt);
				\draw (-10:2cm) node[anchor=west] {\scriptsize$d_\ell=n$} -- (10:2cm) node[anchor=west] {\scriptsize$n+1$}; 
				\draw (170:2cm) node[anchor=east] {\scriptsize$0$} -- (190:2cm) node[anchor=east] {\scriptsize$d_1=1$}; 
				\draw (10:2cm) arc (26:154:2.2cm and 0.8cm);
				\draw (190:2cm) arc (206:334:2.2cm and 0.8cm);
				\filldraw [black] (210:1.35cm) node[anchor=north] {\scriptsize$a=d_{r-1}$} circle (2pt)
								  (330:1.35cm) node[anchor=north] {\scriptsize$\qquad b=d_{r}$} circle (2pt)
								  (90:0.8cm) node[anchor=south] {\scriptsize$\gamma=\Gamma=u_{s}$} circle (2pt);
				\draw (210:1.35cm) -- (90:0.8cm);
				\draw (90:0.8cm) -- (330:1.35cm);
		 \end{tikzpicture}
		 \qquad
		 \begin{tikzpicture}[thick]
			\filldraw [black] (-10:2cm) circle (2pt)
							  (10:2cm) circle (2pt)
							  (170:2cm) circle (2pt)
							  (190:2cm) circle (2pt);
			\draw (-10:2cm) node[anchor=west] {\scriptsize$\Gamma=d_\ell=n$} -- (10:2cm) node[anchor=west] {\scriptsize$b=n+1$}; 
			\draw (170:2cm) node[anchor=east] {\scriptsize$0$} -- (190:2cm) node[anchor=east] {\scriptsize$d_1=1$}; 
			\draw (10:2cm) arc (26:154:2.2cm and 0.8cm);
			\draw (190:2cm) arc (206:334:2.2cm and 0.8cm);
			\filldraw [black] (210:1.35cm) node[anchor=north] {\scriptsize$a=d_{r-1}$} circle (2pt)
							  (330:1.35cm) node[anchor=north] {\scriptsize$\qquad b=d_{r}$} circle (2pt)
							  (130:1.0cm) node[anchor=south] {\scriptsize$\gamma=u_{s}$} circle (2pt)
							  (50:1.0cm) node[anchor=south] {\scriptsize$\Gamma=u_{s+1}$} circle (2pt);
			\draw (210:1.35cm) -- (130:1.0cm);
			\draw (50:1.0cm) -- (330:1.35cm);
			\draw [dashed] (210:1.35cm) -- (50:1.0cm);
			\draw [dashed] (130:1.0cm) -- (330:1.35cm);
		 \end{tikzpicture}
		 \centerline{(a) $I=\{u_s\}$ and $1\leq s<m$ \rule{2.1cm}{0cm} (b) $I=\{u_s,u_{s+1}\}$ and $1\leq s<m$}
	  \end{center}
	  \caption[]{Schematic illustrations the two cases of $\mathscr D_{I}=\{\delta_2,\delta_3\}$ (Lemma~\ref{lem_characterise_D_I=delta2_delta3}).}
      \label{fig:example_II}
      \end{minipage}
      \end{center}
\end{figure}
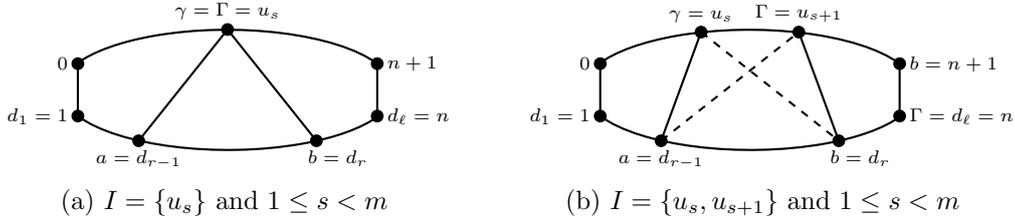
\begin{figure}[b]
      \begin{center}
      \begin{minipage}{0.95\linewidth}
         \begin{center}
		  $ $\\[3mm]
			\begin{tikzpicture}[thick]
				\filldraw [black] (-10:2cm) circle (2pt)
								  (10:2cm) circle (2pt)
							  	  (170:2cm) circle (2pt)
								  (190:2cm) circle (2pt);
				\draw (-10:2cm) node[anchor=west] {\scriptsize$d_\ell=n$} -- (10:2cm) node[anchor=west] {\scriptsize$n+1$}; 
				\draw (170:2cm) node[anchor=east] {\scriptsize$0$} -- (190:2cm) node[anchor=east] {\scriptsize$d_1=1$}; 
				\draw (10:2cm) arc (26:154:2.2cm and 0.8cm);
				\draw (190:2cm) arc (206:334:2.2cm and 0.8cm);
				\filldraw [black] (200:1.65cm) node[anchor=north] {\scriptsize$a=d_{r-1}\qquad\quad$} circle (2pt)
								  (220:1.15cm) node[anchor=north] {\scriptsize$\ \ \gamma=d_{r}$} circle (2pt)
								  (340:1.65cm) node[anchor=north] {\scriptsize$\qquad\quad b=d_{r+2}$} circle (2pt)
								  (320:1.15cm) node[anchor=north] {\scriptsize$\ \ \ \Gamma=d_{r+1}$} circle (2pt)
								  (120:0.9cm) node[anchor=south] {\scriptsize$u_{s}$} circle (2pt)
								  (60:0.9cm) node[anchor=south] {\scriptsize$u_{t}$} circle (2pt);
				\draw (200:1.65cm) -- (120:0.9cm);
				\draw (60:0.9cm) -- (340:1.65cm);
				\draw [dashed] (200:1.65cm) -- (340:1.65cm);
				\draw [dashed] (200:1.65cm) -- (320:1.15cm);
				\draw [dashed] (220:1.15cm) -- (340:1.65cm);
		 \end{tikzpicture}
		 \qquad
		 \begin{tikzpicture}[thick]
			\filldraw [black] (-10:2cm) circle (2pt)
							  (10:2cm) circle (2pt)
							  (170:2cm) circle (2pt)
							  (190:2cm) circle (2pt);
			\draw (-10:2cm) node[anchor=west] {\scriptsize$\Gamma=d_\ell=n$} -- (10:2cm) node[anchor=west] {\scriptsize$b=n+1$}; 
			\draw (170:2cm) node[anchor=east] {\scriptsize$0$} -- (190:2cm) node[anchor=east] {\scriptsize$d_1=1$}; 
			\draw (10:2cm) arc (26:154:2.2cm and 0.8cm);
			\draw (190:2cm) arc (206:334:2.2cm and 0.8cm);
			\filldraw [black] (220:1.15cm) node[anchor=north] {\scriptsize$a=d_{q}$} circle (2pt)
							  (320:1.15cm) node[anchor=north] {\scriptsize$b=d_{r}$} circle (2pt)
							  (120:0.9cm) node[anchor=south] {\scriptsize$\gamma=u_{s}$} circle (2pt)
							  (60:0.9cm) node[anchor=south] {\ \ \ \ \ \ \ \scriptsize$\Gamma=u_{s+1}$} circle (2pt);
			\draw (220:1.15cm) -- (120:0.9cm);
			\draw (60:0.9cm) -- (320:1.15cm);
			\draw [dashed] (220:1.15cm) -- (320:1.15cm);
			\draw [dashed] (220:1.15cm) -- (60:0.9cm);
			\draw [dashed] (120:0.9cm) -- (320:1.15cm);
		 \end{tikzpicture}
		\centerline{
		\begin{minipage}{0.35\linewidth}
			\begin{compactenum}[(a)]
				\item $I=\{d_r,d_{r+1}\}\cup M$\\ with $1\leq r< \ell$\\ and $M\subseteq [d_r,d_{r+1}]\cap \Up_c$
			\end{compactenum}
		\end{minipage}
		$\qquad\qquad\quad$
		\begin{minipage}{0.35\linewidth}
			\begin{compactenum}[(b)]
				\item $I=M\cup\{u_s,u_{s+1}\}$\\ with $1\leq s < m$\\ and $M = [u_s,u_{s+1}]\cap \Do_c \neq \varnothing$
			\end{compactenum}
		\end{minipage}
		}
	  \end{center}
	  \caption[]{Schematic illustrations the two cases of $\mathscr D_{I}=\{\delta_1,\delta_2,\delta_3\}$ (Lemma~\ref{lem:delta1_delta2_delta3}).}
      \label{fig:example_III}
      \end{minipage}
      \end{center}
\end{figure}

\begin{proof}
	From $\delta_1\not\in\mathscr D_I$, we obtain $(a,b)_{\Do_c}=\varnothing$, thus $a<\gamma\leq\Gamma<b$ and $\gamma,\Gamma\in\Up_c$.
	Now $\delta_4\not\in\mathscr D_I$ implies that $\{\gamma,\Gamma\}$ is either degenerate or an edge of~$Q$. This proves the claim.
\end{proof}

\begin{lem}[Compare Figure~\ref{fig:example_III}]\label{lem:delta1_delta2_delta3}$ $\\
	If $\mathscr D_I=\{ \delta_1,\delta_2,\delta_3\}$, then either
	\begin{compactenum}[(a)]
		\item $I=\{d_r,d_{r+1}\}\sqcup M$ with $1\leq r< \ell$ and $M\subseteq [d_r,d_{r+1}]\cap\Up_c$
 			  or \label{lem:delta1_delta2_delta3_a}
		\item $I= M \sqcup \{u_s,u_{s+1}\}$ with $1\leq s < m$ and $M=[u_s,u_{s+1}]\cap\Do_c \neq \varnothing$
			  \label{lem:delta1_delta2_delta3_b}
	\end{compactenum}
\end{lem}
\begin{proof}
	$\delta_1\in\mathscr D_I$ implies $(a,b)_{\Do_c}\neq \varnothing$, while $\delta_4\not\in \mathscr D_I$ implies
	that $\{\gamma,\Gamma\}$ is either an edge of~$Q$ or $\gamma=\Gamma$. Suppose first $\gamma=\Gamma$.
	Then $\gamma=\Gamma\in\Do_c$ implies the contradiction $\mathscr D_I=\{\delta_1\}$, while $\gamma=\Gamma\in\Up_c$
	implies $(a,b)_{\Do_c}=\varnothing$, contradicting $\delta_1\in\mathscr D_I$. We therefore assume $\gamma\neq\Gamma$ and only
	have to distinguish the cases $\gamma,\Gamma\in\Do_c$ and $\gamma,\Gamma\in\Up_c$, the other cases $\gamma\in\Do_c$, $\Gamma\in\Up_c$
	and $\gamma\in\Up_c$ and $\Gamma\in\Do_c$ are not possible since $\delta_4\not\in\mathscr D_I$.
	
	Firstly, suppose $\gamma,\Gamma\in\Do_c$. Then $\gamma=d_r$ and $\Gamma=d_{r+1}$ for some $1\leq r \leq \ell-1$, since
	$\delta_4\not\in\mathscr D_I$. But this implies $I=\{d_r,d_{r+1}\}\cup\bigl([d_r,d_{r+1}]\cap \Up_c\bigr)$, which is the 
	claim of (\ref{lem:delta1_delta2_delta3_a}). Secondly, suppose $\gamma,\Gamma\in\Up_c$. Then
	$\gamma=u_s$ and $\Gamma=u_{s+1}$ for some $1\leq s \leq m-1$. But this implies $[u_s,u_{s+1}]\cap\Do_c=(d_q,d_r)_{\Do_c}\neq \varnothing$ 
	and $I=\bigl( [u_s,u_{s+1}]\cap\Do_c \bigr)\cup [u_s,u_{s+1}]_{\Up_c}$, which proves (\ref{lem:delta1_delta2_delta3_b}).
\end{proof}

\begin{figure}[b]
      \begin{center}
		\rule{0mm}{5mm}
		
      \begin{minipage}{0.95\linewidth}
		 \begin{center}
		  $ $\\[3mm]
		 	\begin{tikzpicture}[thick]
				\filldraw [black] (-10:2cm) circle (2pt)
								  (10:2cm) circle (2pt)
							  	  (170:2cm) circle (2pt)
								  (190:2cm) circle (2pt);
				\draw (-10:2cm) node[anchor=west] {\scriptsize$\Gamma=d_\ell=n$} -- (10:2cm) node[anchor=west] {\scriptsize$b=n+1$}; 
				\draw (170:2cm) node[anchor=east] {\scriptsize$0$} -- (190:2cm) node[anchor=east] {\scriptsize$d_1=1$}; 
				\draw (10:2cm) arc (26:154:2.2cm and 0.8cm);
				\draw (190:2cm) arc (206:334:2.2cm and 0.8cm);
				\filldraw [black] (330:1.35cm) node[anchor=north] {\scriptsize$a=d_{r}$} circle (2pt)
								  (23.5:1.58cm) node[anchor=south] {\scriptsize$\gamma=u_{m}$} circle (2pt);
				\draw (330:1.35cm) -- (23.5:1.58cm);
				\draw [dashed] (330:1.35cm) -- (10:2cm);
				\draw [dashed] (330:1.35cm) -- (-10:2cm);
				\draw [dashed] (23.5:1.58cm) -- (-10:2cm);
	 		\end{tikzpicture}
		 \end{center}
		 \centerline{\begin{minipage}{0.35\linewidth}
					 \begin{compactenum}[(a)]
		 				\item $I = \{ d_{r+1},\ldots,d_\ell\} \cup \{ u_m \}$\\ with $d_r < u_m < d_{r+1} < d_\ell$
		 			 \end{compactenum}
				 	 \end{minipage}
					}
         \begin{center}
			\begin{tikzpicture}[thick]
				\filldraw [black] (-10:2cm) circle (2pt)
								  (10:2cm) circle (2pt)
							  	  (170:2cm) circle (2pt)
								  (190:2cm) circle (2pt);
				\draw (-10:2cm) node[anchor=west] {\scriptsize$d_\ell=n$} -- (10:2cm) node[anchor=west] {\scriptsize$n+1$}; 
				\draw (170:2cm) node[anchor=east] {\scriptsize$0$} -- (190:2cm) node[anchor=east] {\scriptsize$d_1=1$}; 
				\draw (10:2cm) arc (26:154:2.2cm and 0.8cm);
				\draw (190:2cm) arc (206:334:2.2cm and 0.8cm);
				\filldraw [black] (200:1.65cm) node[anchor=north] {\scriptsize$a=d_{r-1}\qquad\quad$} circle (2pt)
								  (220:1.15cm) node[anchor=north] {\scriptsize$\ \ \gamma=d_{r}$} circle (2pt)
								  (320:1.15cm) node[anchor=north] {\scriptsize$\ \ \ b=d_{r+1}$} circle (2pt)
								  (120:0.9cm) node[anchor=south] {\scriptsize$u_{s}$} circle (2pt)
								  (60:0.9cm) node[anchor=south] {\scriptsize$\Gamma=u_{t}$} circle (2pt);
				\draw (200:1.65cm) -- (120:0.9cm);
				\draw (60:0.9cm) -- (320:1.15cm);
				\draw [dashed] (200:1.65cm) -- (320:1.15cm);
				\draw [dashed] (200:1.65cm) -- (60:0.9cm);
				\draw [dashed] (220:1.15cm) -- (60:0.9cm);
		 \end{tikzpicture}
		 \qquad
		 \begin{tikzpicture}[thick]
			\filldraw [black] (-10:2cm) circle (2pt)
							  (10:2cm) circle (2pt)
							  (170:2cm) circle (2pt)
							  (190:2cm) circle (2pt);
			\draw (-10:2cm) node[anchor=west] {\scriptsize$d_\ell=n$} -- (10:2cm) node[anchor=west] {\scriptsize$n+1$}; 
			\draw (170:2cm) node[anchor=east] {\scriptsize$a=0$} -- (190:2cm) node[anchor=east] {\scriptsize$\gamma=d_1=1$}; 
			\draw (10:2cm) arc (26:154:2.2cm and 0.8cm);
			\draw (190:2cm) arc (206:334:2.2cm and 0.8cm);
			\filldraw [black] (161:1.75cm) node[anchor=south] {\scriptsize$u_{1}\ $} circle (2pt)
							  (145:1.25cm) node[anchor=south] {\scriptsize$\ \ \Gamma=u_{s}$} circle (2pt)
							  (220:1.15cm) node[anchor=north] {\scriptsize$b=d_{2}$} circle (2pt);
			\draw (145:1.25cm) -- (220:1.15cm);
			\draw [dashed] (170:2cm) -- (143:1.2cm);
			\draw [dashed] (190:2cm) -- (145:1.25cm);
			\draw [dashed] (170:2cm) -- (220:1.15cm);
		 \end{tikzpicture}
		 \centerline{\begin{minipage}{0.3\linewidth}
		 			 \begin{compactenum}[(b)]
		 			 	\item $I=\{d_r\}\cup M$\\ with $1<r<\ell$,\\ and $M\subseteq[d_r,d_{r+1}]\cap\Up_c$\\ and $M\neq \varnothing$
		 			 \end{compactenum}
		 			 \end{minipage}
				 	 $\qquad\qquad\qquad\qquad$
					 \begin{minipage}{0.3\linewidth}
					 \begin{compactenum}[(c)]
						\item $I=\{d_1\}\cup M$\\ with $M\subseteq[d_1,d_2]\cap\Up_c$\\ and $M\setminus\{u_1\}\neq \varnothing$
					 \end{compactenum}
				 	 \end{minipage}
		}
	  \end{center}
	  \caption[]{Schematic illustrations the three cases of $\mathscr D_{I}=\{\delta_1,\delta_2,\delta_4\}$ (Lemma~\ref{lem:delta1_delta2_delta4}).}
      \label{fig:example_IV}
      \end{minipage}
	  $ $
	
	  \rule{0mm}{5mm}
      \end{center}
\end{figure}
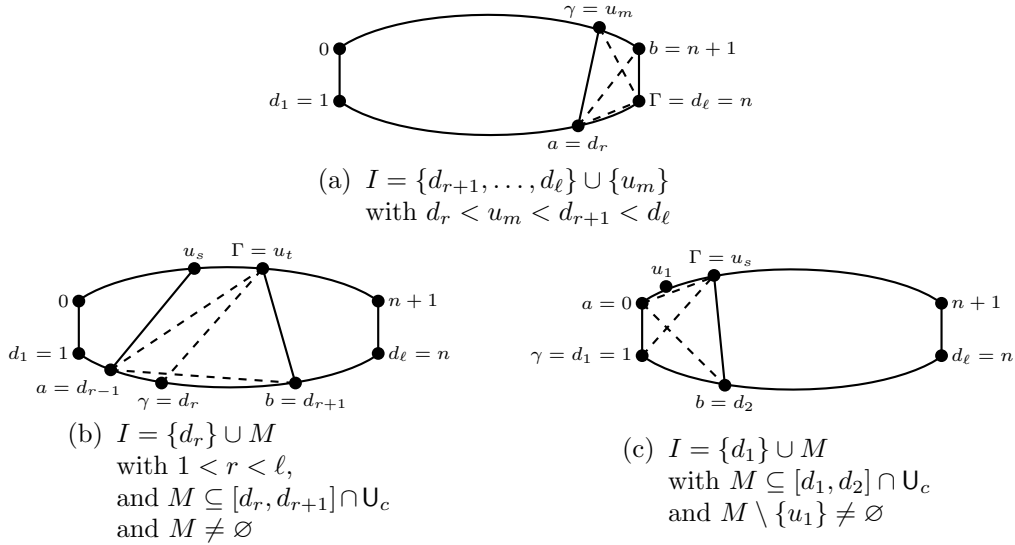
\noindent
Lemma~\ref{lem:delta1_delta2_delta4} is symmetric to Lemma~\ref{lem:delta1_delta3_delta4}, their proofs are along the same lines.
\begin{lem}[Compare Figure~\ref{fig:example_IV}]\label{lem:delta1_delta2_delta4}$ $\\
	If $\mathscr D_I=\{ \delta_1,\delta_2,\delta_4\}$, then either
	\begin{compactenum}[(a)]
		\item $I=\{d_{r+1}, \ldots, d_\ell\}\cup \{u_m\}$ with $d_r<u_m<d_{r+1}<d_\ell$
			\label{lem:delta1_delta2_delta4_a}
		\item $I=\{d_r\}\cup M$ with $1<r<\ell$ and $\varnothing\neq M\subseteq [d_r,d_{r+1}]\cap \Up_c$
			\label{lem:delta1_delta2_delta4_b}
		\item $I=\{d_1\}\cup M$ with $M\subseteq [d_1,d_2]\cap \Up_c$ and $M\setminus\{u_1\} \neq \varnothing$
			\label{lem:delta1_delta2_delta4_c}
	\end{compactenum}
\end{lem}

\begin{proof}
	Since $\delta_1\in\mathscr D_I$, we have $(a,b)_{\Do_c}\neq \varnothing$, that is, $a,b$ are not consecutive numbers 
	in~$\Do_c$. From $\delta_3\not\in\mathscr D_I$, we deduce that $\{\gamma,b\}$ is an edge of $Q$ and $\gamma,\Gamma\in\Up_c$
	is therefore impossible unless $\gamma=\Gamma$. Moreover, $\delta_4\in \mathscr D_I$ implies that $\gamma=\Gamma$ is impossible. 
	We now have two cases to distinguish. 
	
	Firstly, suppose $\gamma=u_m$ and $b=n+1$. Then $\Gamma=d_\ell=n$ and $\delta_2\in\mathscr D_I$ implies 
	$(a,\Gamma)_{\Do_c}\neq \varnothing$. Together with $a=\max\set{d\in\Do_c}{d<u_m}$ we have $a=d_r$ 
	for some $1\leq r\leq \ell-2$ with $u_m<d_{r+1}$ and $I=(d_r,n+1)_{\Do_c}\cup[u_m,u_m]_{\Up_c}$, this 
	shows~(\ref{lem:delta1_delta2_delta4_a}).
	
	Secondly, suppose $\gamma=d_r$ and $b=d_{r+1}$ for some $1\leq r\leq \ell-1$, and $\Gamma\in(\gamma,b)\cap\Up_c$. 
	If $\gamma=1$ then $\delta_2\in\mathscr D_I$ implies $\Gamma\neq u_1$, so we distinguish the cases $\gamma=1$ and $\gamma\neq 1$.
	Suppose first that $\gamma=d_r$ with $r>1$. If $[d_r,d_{r+1}]\cap \Up_c\neq \varnothing$ then we immediately have the claim
	for every non-empty $M\subseteq [d_r,d_{r+1}]\cap \Up_c$. If $[d_r,d_{r+1}]\cap \Up_c= \varnothing$ then $\gamma=\Gamma\in\Do_c$
	which is impossible. Thus we have shown~(\ref{lem:delta1_delta2_delta4_b}). Suppose now that $\gamma=d_1=1$. Then $a=0$, $b=d_2$,
	and $\delta_2\in \mathscr D_I$ implies $\Gamma\in\Up_c\setminus \{u_1\}$. This proves~(\ref{lem:delta1_delta2_delta4_c}).
\end{proof}
\enlargethispage{\baselineskip}
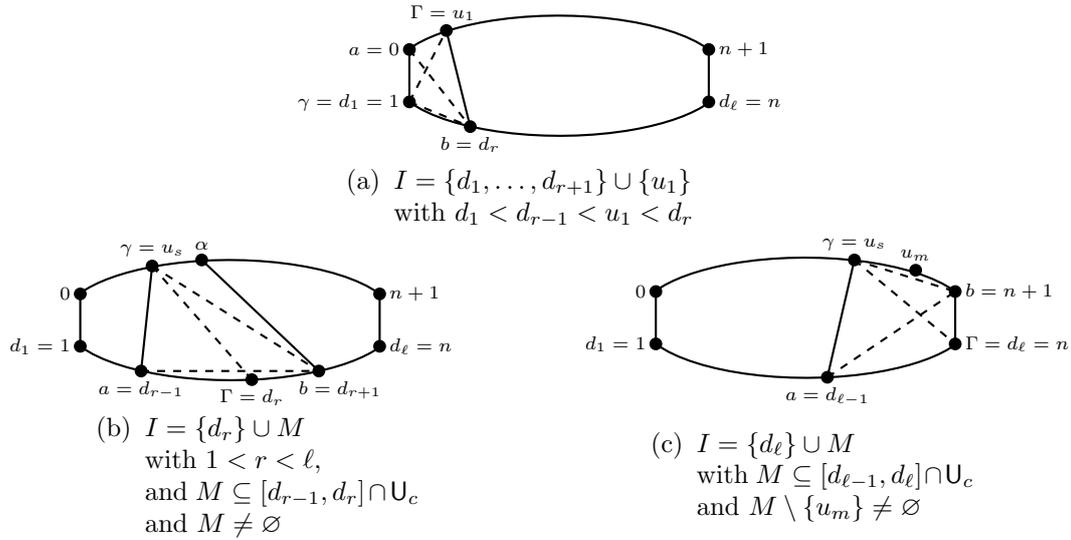
\begin{figure}
      \begin{center}
      \begin{minipage}{0.95\linewidth}
		 \begin{center}
		 	\begin{tikzpicture}[thick]
				\filldraw [black] (-10:2cm) circle (2pt)
								  (10:2cm) circle (2pt)
							  	  (170:2cm) circle (2pt)
								  (190:2cm) circle (2pt);
				\draw (-10:2cm) node[anchor=west] {\scriptsize$d_\ell=n$} -- (10:2cm) node[anchor=west] {\scriptsize$n+1$}; 
				\draw (170:2cm) node[anchor=east] {\scriptsize$a=0$} -- (190:2cm) node[anchor=east] {\scriptsize$\gamma=d_1=1$}; 
				\draw (10:2cm) arc (26:154:2.2cm and 0.8cm);
				\draw (190:2cm) arc (206:334:2.2cm and 0.8cm);
				\filldraw [black] (210:1.35cm) node[anchor=north] {\scriptsize$b=d_{r}$} circle (2pt)
								  (158:1.6cm) node[anchor=south] {\scriptsize$\Gamma=u_{1}\ $} circle (2pt);
				\draw (158:1.6cm) -- (210:1.35cm);
				\draw [dashed] (170:2cm) -- (210:1.35cm);
				\draw [dashed] (190:2cm) -- (158:1.6cm);
				\draw [dashed] (190:2cm) -- (210:1.35cm);
	 		\end{tikzpicture}
		 \end{center}
		 \centerline{\begin{minipage}{0.35\linewidth}
					 \begin{compactenum}[(a)]
		 				\item $I = \{ d_1,\ldots,d_{r+1}\} \cup \{ u_1 \}$\\ with $d_1 < d_{r-1} < u_1 < d_r$
		 			 \end{compactenum}
				 	 \end{minipage}
					}
			\begin{tikzpicture}[thick]
				\filldraw [black] (-10:2cm) circle (2pt)
								  (10:2cm) circle (2pt)
							  	  (170:2cm) circle (2pt)
								  (190:2cm) circle (2pt);
				\draw (-10:2cm) node[anchor=west] {\scriptsize$d_\ell=n$} -- (10:2cm) node[anchor=west] {\scriptsize$n+1$}; 
				\draw (170:2cm) node[anchor=east] {\scriptsize$0$} -- (190:2cm) node[anchor=east] {\scriptsize$d_1=1$}; 
				\draw (10:2cm) arc (26:154:2.2cm and 0.8cm);
				\draw (190:2cm) arc (206:334:2.2cm and 0.8cm);
				\filldraw [black] (210:1.35cm) node[anchor=north] {\scriptsize$a=d_{r-1}$} circle (2pt)
								  (330:1.35cm) node[anchor=north] {\scriptsize$\qquad b=d_{r+1}$} circle (2pt)
								  (290:0.84cm) node[anchor=north] {\scriptsize$\Gamma=d_{r}$} circle (2pt)
								  (145:1.25cm) node[anchor=south] {\scriptsize$\gamma=u_{s}$} circle (2pt)
								  (115:0.88cm) node[anchor=south] {\scriptsize$\alpha$} circle (2pt);
				\draw (210:1.35cm) -- (145:1.25cm);
				\draw (115:0.88cm) -- (330:1.35cm);
				\draw [dashed] (210:1.35cm) -- (330:1.35cm);
				\draw [dashed] (145:1.25cm) -- (330:1.35cm);
				\draw [dashed] (145:1.25cm) -- (290:0.84cm);
		 \end{tikzpicture}
		 \qquad
		 \begin{tikzpicture}[thick]
			\filldraw [black] (-10:2cm) circle (2pt)
							  (10:2cm) circle (2pt)
							  (170:2cm) circle (2pt)
							  (190:2cm) circle (2pt);
			\draw (-10:2cm) node[anchor=west] {\scriptsize$\Gamma=d_\ell=n$} -- (10:2cm) node[anchor=west] {\scriptsize$b=n+1$}; 
			\draw (170:2cm) node[anchor=east] {\scriptsize$0$} -- (190:2cm) node[anchor=east] {\scriptsize$d_1=1$}; 
			\draw (10:2cm) arc (26:154:2.2cm and 0.8cm);
			\draw (190:2cm) arc (206:334:2.2cm and 0.8cm);
			\filldraw [black] (50:1cm) node[anchor=south] {\scriptsize$\gamma=u_{s}$} circle (2pt)
							  (23.5:1.58cm) node[anchor=south] {\scriptsize$u_{m}$} circle (2pt)
							  (290:0.84cm) node[anchor=north] {\scriptsize$a=d_{\ell-1}$} circle (2pt);
			\draw (290:0.84cm) -- (50:1cm);
			\draw [dashed] (290:0.84cm) -- (10:2cm);
			\draw [dashed] (50:1cm) -- (-10:2cm);
			\draw [dashed] (50:1cm) -- (10:2cm);
		 \end{tikzpicture}
		 \centerline{\begin{minipage}{0.3\linewidth}
		 			 \begin{compactenum}[(b)]
		 			 	\item $I=\{d_r\}\cup M$\\ with $1<r<\ell$,\\ and $M\subseteq[d_{r-1},d_r]\cap\Up_c$\\ and $M\neq \varnothing$
		 			 \end{compactenum}
		 			 \end{minipage}
				 	 $\qquad\qquad\qquad\qquad$
					 \begin{minipage}{0.3\linewidth}
					 \begin{compactenum}[(c)]
						\item $I=\{d_\ell\}\cup M$\\ with $M\subseteq[d_{\ell-1},d_\ell]\cap\Up_c$\\ and $M\setminus\{u_m\}\neq \varnothing$
					 \end{compactenum}
				 	 \end{minipage}
		}
	  \caption[]{Schematic illustrations the three cases of $\mathscr D_{I}=\{\delta_1,\delta_3,\delta_4\}$ (Lemma~\ref{lem:delta1_delta3_delta4}).}
      \label{fig:example_V}
      \end{minipage}
      \end{center}
\end{figure}
\begin{lem}[Compare Figure~\ref{fig:example_V}]\label{lem:delta1_delta3_delta4}$ $\\
	If $\mathscr D_I=\{ \delta_1,\delta_3,\delta_4\}$, then either
	\begin{compactenum}[(a)]
		\item $I=\{d_1,\ldots,d_{r-1}\}\cup \{u_1\}$ with $d_1<d_{r-1}<u_1<d_r$, or
			  \label{lem:delta1_delta3_delta4_a}
		\item $I=\{d_r\}\cup M$ with $1 < r < \ell$ and $\varnothing\neq M\subseteq[d_{r-1},d_r]\cap\Up_c$ or
	 		  \label{lem:delta1_delta3_delta4_b}
		\item $I=\{d_\ell\}\cup M$ with $M\subseteq [d_{\ell-1},d_\ell]\cap\Up_c$ and $M\setminus\{u_m\}\neq \varnothing$.
		 	  \label{lem:delta1_delta3_delta4_c}
	\end{compactenum}
\end{lem}

\begin{lem}\label{lem:delta2_delta3_delta4}
	If $\mathscr D_I=\{ \delta_2,\delta_3,\delta_4\}$, then $I=\{u_s,\ldots, u_t\}$ with $s+1<t$ and $(u_s,u_t)\cap \Do_c=\varnothing$. \end{lem}
\begin{proof}
	From $\delta_1\not\in\mathscr D_I$, we obtain $(a,b)_{\Do_c}=\varnothing$, in particular, $a=d_r$ and $b=d_{r+1}$ for some
	$1\leq r \leq \ell-1$. Thus $\gamma,\Gamma\in\Up_c$ and because of $\delta_4\in\mathscr D_I$ we have~$\gamma=u_s$ and~$\Gamma=u_t$ 
	for some $1\leq s<s+1<t\leq u_m$. But then $I=M$ for some $M\subseteq [u_s,u_t]_{\Up_c}$ with $u_s,u_t \in M$.
\end{proof}

\section*{Acknowledgments}
The author was partially supported by DFG Forschergruppe 565 \emph{Polyhedral Surfaces}. Extended abstracts of preliminary versions were presented at FPSAC~2011 and CCCG~2011. I thank the various anonymous referees for their helpful comments, in particular, one referee for \emph{Discrete and Computational Geometry}. 

\bibliographystyle{plain}
\bibliography{minkowski}
\label{sec:biblio}

\end{document}